\documentclass[a4paper, 10pt]{article}
\usepackage[utf8]{inputenc}
\usepackage{amsmath,amsthm}
\numberwithin{equation}{section}
\usepackage{amssymb,esint,hyperref}
\usepackage{amscd}
\usepackage{xspace}
\usepackage{fancyhdr}
\usepackage{verbatim}
\usepackage{graphicx}
\usepackage{cite}
\usepackage{stmaryrd}
\usepackage{bbm}
\usepackage{xcolor}
\usepackage{mathrsfs}
\newtheorem{theorem}{Theorem}[section]

\newtheorem{lemma}[theorem]{Lemma}

\newtheorem{remark}[theorem]{Remark}

\def\da{\mathsf{Data}}
\def\B{\mathsf{B}}
\def\A{\mathsf{A}}
\def\M{\mathsf{M}}

\def\sH{\mathsf{H}}

\def\K{\mathsf{K}}
\def\sL{\mathsf{L}}

\def\R{\mathsf{R}}

\def\x{\mathbf{x}}

\newcommand{\eps}{\varepsilon}
	\title{Well-posedness for the mean curvature flow on the half-space and on bounded domains}
	\author{
		{{\bf Ke Chen,\thanks{E-mail address: k1chen@polyu.edu.hk, Department of Applied Mathematics, The Hong Kong Polytechnic University, Kowloon, Hong Kong, PR China.}
			~~Ruilin Hu\thanks{E-mail address: rh2488@bath.ac.uk, Department of Mathematical Sciences, University of Bath, Bath, BA2 7AY, UK.}
			~~Quoc-Hung Nguyen\thanks{E-mail address: qhnguyen@amss.ac.cn, Academy of Mathematics and Systems Science, Chinese Academy of Sciences, Beijing, 100190, China.}}}}
\begin{document}
	\maketitle
	\begin{abstract}
		We study graphical mean curvature flow in arbitrary codimension over the half-space and over smooth bounded domains, subject to homogeneous Dirichlet boundary conditions. At the scaling-critical Lipschitz regularity, we prove local well-posedness for initial graphs that can be approximated in $W^{1,\infty}$ by smooth profiles compatible with the boundary condition. Within this class, if the initial Lipschitz seminorm is sufficiently small, the corresponding solution is global; on a bounded domain, it also converges exponentially to the flat graph. Positive-time regularization is quantified by time-weighted H\"older estimates whose weighted quantities remain bounded as $t\downarrow0$. The main analytic ingredient is a boundary Schauder theory for variable-coefficient parabolic systems. On the half-space, it combines coefficient freezing with parity extensions and boundary identities intrinsic to the graphical system. On curved domains, localization and boundary flattening lead to anisotropic estimates, from which normal derivatives are recovered recursively.
	\end{abstract}
\section{Introduction}
	\subsection*{Background and formulation of the problem}
	Mean curvature flow evolves an immersed submanifold in the direction of its mean curvature vector. Let $\mathcal M$ be a $d$-dimensional manifold and let $F_0:\mathcal M\to\mathbb R^{d+N}$ be an immersion. The flow is governed by
	\begin{equation*}
		\partial_tF=\mathbf H,\qquad F|_{t=0}=F_0,
	    \end{equation*}
	where $\mathbf H$ is the mean curvature vector of $F(t,\mathcal M)$. After a tangential reparametrization, a graphical solution has the form $F(t,x)=(x,f(t,x))$, where $f$ is $\mathbb R^N$-valued. The graph function then solves the quasilinear parabolic system
	\begin{equation}\label{mcf}
		\partial_tf=\A[\nabla f]:\nabla^2f,\qquad f|_{t=0}=f_0,
	\end{equation}
	with
	\begin{equation*}
		\A[\nabla f]=\left(\mathrm{Id}+\sum_{\alpha=1}^{N}\nabla f^\alpha\otimes\nabla f^\alpha\right)^{-1},
		\qquad
		\A[\nabla f]:\nabla^2f=\sum_{i,j=1}^d\A_{ij}[\nabla f]\,\partial_i\partial_jf.
	\end{equation*}
	The coefficient matrix depends only on the slope of the graph and satisfies
	\begin{equation*}
		\frac{|\xi|^2}{1+|\nabla f|^2}
		\leq \xi\cdot\A[\nabla f]\xi
		\leq |\xi|^2,
		\qquad \xi\in\mathbb R^d,
	\end{equation*}
Thus a bound for $\|\nabla f\|_{L^\infty}$ yields uniform parabolicity. The difficulty is to use this ellipticity without losing derivatives through the nonlinear dependence of $\A$ on $\nabla f$.

	On $\mathbb R^d$ and $\mathbb R^d_+$, equation \eqref{mcf} is invariant under the parabolic scaling
	\[
		f(t,x)\longmapsto f_\lambda(t,x):=\lambda^{-1}f(\lambda^2t,\lambda x).
	\]
	The seminorm $\|\nabla f_0\|_{L^\infty}$ is invariant under this transformation, making $\dot W^{1,\infty}$ the natural critical space. The scaling also dictates the time weights in our estimates: for every integer $r\geq1$ and $\alpha\in(0,1)$,
	\[
		t^{(r-1+\alpha)/2}\|\nabla^rf(t)\|_{\dot C^\alpha},
	\]
	is invariant. We seek estimates that retain this critical scaling and remain uniform as $t\downarrow0$, although the initial graph has only one bounded derivative.
    
	For entire hypersurfaces, Ecker and Huisken established long-time existence for graphical mean curvature flow and derived interior estimates that yield immediate regularization from locally Lipschitz initial graphs \cite{EH1989,EH1991}. Clutterbuck and Schn\"urer later proved stability of graphical self-expanders arising from mean-convex cones \cite{ClutterbuckSchnurer2011}. The situation changes in higher codimension. The nonsmooth minimal graphs of Lawson and Osserman \cite{LO1977} show that Lipschitz regularity alone does not imply smoothness, even for stationary solutions. Long-time theories are nevertheless available under additional geometric hypotheses, notably the area-decreasing conditions of Wang and of Tsui and Wang \cite{W2002,TsuiWang2004}, and in the Lagrangian setting \cite{CCH2009,CCY2013}.

	At the scale-critical Lipschitz level, Koch and Lamm proved global existence, uniqueness, and analyticity for entire graphical mean curvature flow with sufficiently small slope in arbitrary codimension \cite{Koch2011}; their rough-data framework is surveyed in \cite{KochLamm2015}. A complementary approach starts from rough sets rather than a fixed graphical parametrization. Hershkovits obtained short-time smoothing and uniqueness for sufficiently flat Reifenberg sets, first in codimension one and then in arbitrary codimension \cite{Hershkovits2017,Hershkovits2018}. More recent results include long-time existence for uniformly area-decreasing entire graphs of codimension two and, under an additional asymptotically conical hypothesis, convergence to a self-expander \cite{SavasHalilajSmoczyk2024}, as well as a heat-kernel proof of well-posedness and continuous dependence for hypersurfaces with bounded second fundamental form \cite{Han2026}. These results concern the whole space or boundaryless submanifolds; they do not provide the fixed-boundary estimates developed here.

	{ Earlier, Wang proved instantaneous smoothing for Lipschitz submanifolds with sufficiently small local Lipschitz norm and treated the Dirichlet problem for the higher-codimensional minimal surface system through the associated graphical parabolic system \cite{W2004b,W2004a}.}

	The presence of a boundary introduces a different set of issues: the boundary condition must be propagated, normal regularity must be recovered from tangential estimates, and flattening a curved boundary creates variable coefficients and commutator terms. Huisken studied nonparametric mean curvature flow with prescribed boundary data in \cite{Huisken1989Boundary}; related gradient estimates and capillary-type boundary problems may be found in \cite{AC09,Jang23}. In a distinct class of geometric parabolic problems, Sire, Wei, and Zheng constructed finite-time boundary singularities for harmonic-map flow with free boundary on the half-plane \cite{SireWeiZheng2023}. Lin, Sire, Wei, and Zhou constructed singularities at prescribed boundary or interior points for nematic-liquid-crystal flow with partially free boundary \cite{LinSireWeiZhou2023}, while Sire, Wu, and Zhou obtained global weak solutions with partial regularity on smooth bounded planar domains \cite{SireWuZhou2024}.

	{ Boundary-sensitive regularity questions also arise in other parabolic systems. For the Stokes and Navier--Stokes equations, Chang and Kang showed that bounded solutions may have singular normal derivatives near the boundary and gave conditions ensuring bounded, indeed H\"older-continuous, gradients \cite{ChangKang}. For critical SQG with Dirichlet fractional dissipation, a sequence of works by Constantin and Ignatova developed nonlinear and commutator estimates for the Dirichlet fractional Laplacian, together with interior and boundary estimates in bounded domains \cite{CI17,CI16,CI20}; their subsequent work with the third author established global smooth well-posedness \cite{Constantin2023}.}

	Related analytical themes also occur in critical Yang--Mills heat flow \cite{SireWeiZhengYM2022,SireWeiZhengZhouYM2026}, nonlocal drift--diffusion equations \cite{NguyenNowakSireWeidner2024}, and nonlocal elliptic boundary regularity \cite{KimWeidner2025}. These equations and their solution theories differ from \eqref{mcf}; we mention them only as broader context for critical regularity and boundary effects.

	The purpose of this paper is to develop a critical, time-weighted theory up to a fixed boundary for Lipschitz graphs in arbitrary codimension. It is the third part of a series originating from \cite{KHN2024}. Part I \cite{KHN2025} develops a kernel-based coefficient-freezing method and proves time-weighted Schauder estimates for variable-coefficient parabolic systems on the whole space. Part II \cite{KHN2026} combines that framework with elliptic transmission estimates for the two-phase Muskat problem with surface tension near the log-critical Lipschitz threshold. Here we adapt the method to fixed boundaries. New ingredients include parity arguments on the half-space, boundary identities for the graphical system, anisotropic estimates after boundary flattening, and a long-time argument on bounded domains.

	More precisely, we consider the initial-boundary value problem
	\begin{equation}\label{meancur1}
	\begin{aligned}
	&\partial_t f=\A[\nabla f]:\nabla^2 f &&\text{in }(0,T)\times\Omega,\\
	&f(0,\cdot)=f_0 &&\text{in }\Omega,\\
	&f=0 &&\text{on }(0,T)\times\partial\Omega,
	\end{aligned}
	\end{equation}
		Here $f:[0,T]\times\Omega\to\mathbb R^N$, the domain $\Omega$ is either $\mathbb R_+^d$ or a bounded domain with $C^{2m+3}$ boundary, and $\A$ is given by \eqref{mcf}. The trace of $f_0$ is assumed to vanish on $\partial\Omega$. We prove local well-posedness when $f_0$ can be approximated in the Lipschitz norm by smooth functions satisfying the same boundary condition. Within this approximable class, smallness of $\|\nabla f_0\|_{L^\infty}$ gives global existence; on bounded domains, the solution also decays exponentially. We do not claim a local theory for arbitrary Lipschitz data in higher codimension.
	
	\subsection*{Main results}
	The interior estimates follow from the time-weighted whole-space Schauder theory established in Part I \cite{KHN2025}, recalled below as Theorem~\ref{propb=0}. The new content lies in estimates up to the boundary and their application to \eqref{meancur1}.

	Our first result concerns the half-space. The local statement permits a large slope, but requires the initial graph to lie in the $W^{1,\infty}$-closure of smooth boundary-compatible profiles. Once one such profile is chosen sufficiently close to the data, its higher norms determine the lifespan. The global statement additionally assumes that the critical Lipschitz seminorm is small.
	\begin{theorem}[Half-space]\label{eqmcls}
	Let $m\in\mathbb N$, let $\kappa\in(\frac{3}{4},1)$, and set $\Omega=\mathbb R_+^d$. There exists $\eps_0>0$ such that the following assertions hold.
	\begin{enumerate}
	\item Suppose $f_0\in W^{1,\infty}(\mathbb R_+^d;\mathbb R^N)$ has zero boundary trace and that there exists a sequence $\{\phi_\eps\}\subset C_b^\infty(\overline{\mathbb R_+^d};\mathbb R^N)$, also with zero boundary trace, such that
	\[
		\lim_{\eps\rightarrow 0}\|f_0-\phi_\eps\|_{W^{1,\infty}(\mathbb R^d_+)}=0.
	\]
	Choose $\eps>0$ so that $\|f_0-\phi_\eps\|_{W^{1,\infty}(\mathbb R^d_+)}\leq\eps_0$, and write $\phi=\phi_\eps$.
	Then there exists $T_0>0$ for which \eqref{meancur1} has a unique solution in the following class,
	\begin{align*}
	&\left\{f\in L^\infty((0,T_0);C_b^1(\overline{\mathbb{R}^d_+})) \cap C((0,T_0];C_b^{2m+2}(\overline{\mathbb{R}_+^d})):\right. \\
    &\quad\left.\sup_{0<t\leq T_0}\left[
	\|\nabla f(t)\|_{L^\infty}
	+t^{m+\frac{1+\kappa}{2}}
	\left(\|\nabla^{2m+2}f(t)\|_{\dot C^\kappa}
	+\|\nabla^{2m}\partial_tf(t)\|_{\dot C^\kappa}\right)
	\right]
	\leq C\|f_0\|_{\dot W^{1,\infty}}\right\}.
	\end{align*}
	Here $T_0$ and $C$ may depend on $d,N,m,\kappa$ and $\|\phi\|_{C_b^{2m+3}}$; $\eps_0$ depends only on $d,N,m$, and $\kappa$.

	\item If in addition $\|f_0\|_{\dot W^{1,\infty}}\leq\eps_0$, then the solution is global and
	\begin{equation*}
	\sup_{t>0}\left[
	\|\nabla f(t)\|_{L^\infty}
	+t^{m+(1+\kappa)/2}
	\left(\|\nabla^{2m+2}f(t)\|_{\dot C^\kappa}
	+\|\nabla^{2m}\partial_tf(t)\|_{\dot C^\kappa}\right)
	\right]
	\leq C\|f_0\|_{\dot W^{1,\infty}}.
	\end{equation*}
	\end{enumerate}
	In both cases the positive-time solution satisfies the additional boundary identities
	\begin{equation}\label{flat-boundary-identities-intro}
	\left.\Delta^kf(t)\right|_{\partial\mathbb R_+^d}=0,
	\qquad k=1,\ldots,m.
	\end{equation}
	\end{theorem}

	On a bounded domain, the Dirichlet heat semigroup provides a natural family of smooth, boundary-compatible profiles. The local assumption below requires convergence of this family to $f_0$ in $W^{1,\infty}$. This is a genuine restriction on Lipschitz data, not an automatic consequence of strong continuity in a weaker norm.
	\begin{theorem}[Bounded domains]\label{mcbdglo}
	Let $m\geq2$, let $\kappa\in(\frac{3}{4},1)$ be as in Theorem~\ref{eqmcls}, and let $\Omega\subset\mathbb R^d$ be a bounded domain with $C^{2m+3}$ boundary. There exists $\eps_0>0$ such that the following assertions hold.
	\begin{enumerate}
	\item Let $f_0\in W^{1,\infty}(\Omega;\mathbb R^N)$ have zero boundary trace. Assume that
	\begin{equation*}
		\lim_{\eps_1\rightarrow 0}\|f_0-e^{\eps_1\Delta_\Omega}f_0\|_{W^{1,\infty}(\Omega)}=0,
	\end{equation*}
	where $e^{t\Delta_\Omega}$ is the Dirichlet heat semigroup. Choose $\eps_1>0$ so that the norm in the preceding display is at most $\eps_0$. Then there exists $T_0>0$, depending on $\eps_1$, $d,N,m,\kappa$, the $C^{2m+3}$ geometry of $\Omega$, and $\|e^{\eps_1\Delta_\Omega}f_0\|_{C^{2m+3}(\overline\Omega)}$, such that \eqref{meancur1} has a unique solution in the class
	\begin{align*}
	&\left\{f\in L^\infty((0,T_0);C^1(\overline\Omega))\cap C((0,T_0];C^{2m+1}(\overline\Omega)):\right.\\
    &\quad\quad\left.\sup_{0<t\leq T_0}\left[
	\|\nabla f(t)\|_{L^\infty}
	+t^{m+\kappa/2}\|\nabla^{2m+1}f(t)\|_{C^\kappa}
	\right]
	\leq C\|f_0\|_{W^{1,\infty}}\right\}.
	\end{align*}

	\item If in addition $\|f_0\|_{W^{1,\infty}(\Omega)}\leq\eps_0$, then the solution is global, and there are constants $c,C>0$ such that
	\begin{equation}\label{bddmcgloe}
		\sup_{t>0}\sum_{j=1}^{2m} \min\{t,1\}^{(j-1)/2}
		\|\nabla^jf(t)\|_{L^\infty}
		\leq C\bigl(\|f_0\|_{W^{1,\infty}},m,\Omega\bigr).
	\end{equation}
	\end{enumerate}
		The constants may depend on $d,N,m,\kappa$, and the $C^{2m+3}$ geometry of $\Omega$; in the local assertion, $T_0$ and $C$ may also depend on $\|e^{\eps_1\Delta_\Omega}f_0\|_{C^{2m+3}(\overline\Omega)}$.
		\end{theorem}

		\begin{remark}
		The approximation condition in Theorem~\ref{mcbdglo}(i) holds for $f_0\in C^1_0(\overline\Omega;\mathbb R^N)$. For general Lipschitz data it is an additional hypothesis, since $e^{t\Delta_\Omega}f_0\to f_0$ in $W^{1,\infty}$ need not hold.
		\end{remark}

		\begin{remark}
		Compared with the flat-boundary estimate, the boundary-patch norm in Theorem~\ref{mcbdglo} controls derivatives of total parabolic order at most $2m+2$ but restricts the number of pure normal derivatives to at most $2m+1$; thus one normal derivative is lost relative to the highest tangential order. The precise anisotropic norm is defined in Section~\ref{sec4.2}. The isotropic consequence stated in the theorem is the resulting $C^{2m+1}$ estimate.
		\end{remark}

	\begin{remark}[Flat and curved boundary identities]
		The identities \eqref{flat-boundary-identities-intro} are specific to the flat boundary; see Lemmas~\ref{bdcon} and \ref{lembdc}. After a curved boundary is flattened, geometric lower-order terms enter the forcing. Lemma~\ref{mcbdlembdy} then expresses the highest normal derivatives in terms of derivatives of the forcing and lower-order normal and tangential derivatives; see also Remark~\ref{bdydiff}.
	\end{remark}

	\subsection*{Main ideas of the proofs}
	Both theorems are proved in a perturbative, time-weighted framework. Flat and curved boundaries, however, require different linear estimates.

	\medskip\noindent
	\emph{1. Perturbation around a smooth profile.}
	For the local theory, we fix a smooth function $\phi$ satisfying the homogeneous boundary condition and regard the equation as a perturbation of the operator with coefficient $\A[\nabla\phi]$. The term
	\[
		\big(\A[\nabla f]-\A[\nabla\phi]\big):\nabla^2f
	\]
	is small in the critical norm when $\|f_0-\phi\|_{W^{1,\infty}}$ is small, while higher norms of $\phi$ determine the lifespan. On a bounded domain we take $\phi=e^{\eps_1\Delta_\Omega}f_0$. For the small-data theory we set $\phi=0$ and use the quadratic expansion $\A[\nabla f]-\mathrm{Id}=O(|\nabla f|^2)$.

	\medskip\noindent
	\emph{2. The flat boundary: freezing coefficients, parity, and compatibility.}
	For the first half-space model in Section~\ref{sec2}, we freeze the coefficient matrix at a point $x_0$ and use a linear transformation that preserves $\mathbb R^d_+$ to reduce the frozen operator to the Laplacian. The resulting error contains the increment $\B(x)-\B(x_0)$. Moment estimates for the Dirichlet heat kernel recover this increment before we set $x_0=x$. This is the boundary counterpart of the coefficient-freezing argument in \cite{KHN2025}.

	A key point is that the critical forcing $F^2$ vanishes on $\partial\mathbb R_+^d$. Its odd extension therefore has enough regularity to move derivatives onto the heat kernel. This is not an additional compatibility condition on the initial data. For positive times, the Dirichlet condition and the structure of \eqref{mcf} imply $\Delta f=0$ on the boundary; iteration gives \eqref{flat-boundary-identities-intro}. These identities permit higher-order estimates without uncontrolled boundary traces.

	The time weights absorb the singular behavior of higher derivatives near $t=0$. We then fix a coefficient field $g$ and solve
	\[
		\partial_tf=\A[\nabla g]:\nabla^2f,
	\]
	with the prescribed initial and boundary data. The resulting solution map preserves a small ball in the weighted space and is contractive. This proves existence, uniqueness, and positive-time regularization without imposing classical compatibility conditions of all orders at $t=0$.

	\medskip\noindent
	\emph{3. Curved boundaries: localization and anisotropic estimates.}
	On a bounded domain, we separate interior and boundary patches. Theorem~\ref{propb=0} controls the interior patches. Near $\partial\Omega$, we flatten the boundary, freeze the principal coefficient, and compare the localized equation with a half-space problem. The oscillation of both the geometric coefficients and $\A[\nabla f]$ is absorbed by choosing the patches sufficiently small.

	Boundary flattening and commutation with cutoffs generate lower-order terms, including a forcing of the form $F^1+\partial_dF^2$. Here $F^2$ need not vanish on the flattened boundary, so the direct odd-extension argument is no longer available. The second model, \eqref{mcbd2}, treats this normal-divergence forcing by means of the Dirichlet Green function. It yields anisotropic estimates with a full H\"older gain in tangential directions. The equation and the identities in Lemma~\ref{mcbdlembdy} then recover normal derivatives recursively. This closes the estimates at the level recorded in Theorem~\ref{mcbdglo}, with a loss of one normal derivative relative to the flat case.

	\medskip\noindent
	\emph{4. Global control on bounded domains.}
	The localized Schauder estimates contain lower-order terms that are not uniformly controlled by the local argument alone. On a bounded domain, the homogeneous Dirichlet condition and uniform ellipticity give an $L^2$ energy inequality. Poincar\'e's inequality then yields exponential decay. Interpolation transfers this decay to the lower-order norms that occur in the localized estimates. A continuation argument preserves the smallness of the critical norm, and restarting the local theory on fixed time intervals gives the uniform higher-order bound \eqref{bddmcgloe}.

	\medskip
	The paper is organized as follows. Section~\ref{sec2} establishes boundary Schauder estimates for two half-space models. Section~\ref{sec3} proves the local and global results on $\mathbb R_+^d$. Section~\ref{sec4} treats bounded domains by localization, boundary flattening, anisotropic estimates, and energy decay. Appendix~\ref{sec5} contains the boundary identities used in the higher-order estimates.

\section{Preliminaries}\label{sec2}
This section collects the notation and linear Schauder estimates that will be used throughout the nonlinear analysis.
\subsection{Notation}
This section collects the notation and linear Schauder estimates that will be used throughout the nonlinear analysis.
\begin{itemize}
\item For a Schwartz function $f:\mathbb{R}^d\to\mathbb{R}$, we use
\begin{equation*}
    \widehat f(\xi)=\int_{\mathbb R^d}f(x)e^{-ix\cdot\xi}\,dx,
    \qquad f(x)=(2\pi)^{-d}\int_{\mathbb{R}^d}\widehat f(\xi)e^{ix\cdot\xi}\,d\xi.
\end{equation*}
\item For $x\in\mathbb{R}^d$, write $x'=(x_1,\ldots,x_{d-1})$, so that $x=(x',x_d)$. Similarly, write $\xi=(\xi',\xi_d)$.
\item For any $i\in\{1,\ldots,d\}$, we abbreviate $\partial_{x_i}$ by $\partial_i$.
	    \item We use $\mathbb N=\{0,1,2,\ldots\}$. For any $\beta=(\beta_1,\ldots,\beta_d)\in\mathbb{N}^d$, we denote $\partial^\beta$ or $\nabla^\beta$ to denote $\partial_{1}^{\beta_1}\cdots\partial_{d}^{\beta_d}$.
	    \item For $\eta\in\mathbb R$, the homogeneous multiplier $|D|^\eta=(-\Delta)^{\eta/2}$ is initially defined for Schwartz functions whose Fourier transforms vanish in a neighborhood of $\xi=0$ by
    \begin{equation*}
    |D|^\eta f(x)=(2\pi)^{-d}\int_{\mathbb{R}^d}|\xi|^\eta \widehat f(\xi)e^{ix\cdot\xi}\,d\xi.
    \end{equation*}
    and is extended to the homogeneous spaces used below in the usual tempered-distribution sense, modulo polynomials when necessary. No value of the homogeneous symbol at $\xi=0$ is prescribed.
	    \item Similarly, write $\Delta_{x'}=\sum_{i=1}^{d-1}\partial_{i}^2$. For
    \[
      \widehat f(\xi',x_d):=\int_{\mathbb R^{d-1}}f(x',x_d)e^{-ix'\cdot\xi'}\,dx',
    \]
    define $|D_{x'}|^\eta=(-\Delta_{x'})^{\eta/2}$ by
    \begin{equation*}
     |D_{x'}|^{\eta}f(x)=(2\pi)^{-(d-1)}\int_{\mathbb{R}^{d-1}}|\xi'|^{\eta}\hat{f}(\xi',x_d)e^{i\xi'\cdot x'}d\xi'.
    \end{equation*}
   For functions on $\mathbb{R}_+^d$, these multipliers are applied only after a specified extension to $\mathbb{R}^d$. We define $|D_d|^\eta=(-\partial_d^2)^{\eta/2}$ in the same way.
	    \item For any function $g:\mathbb{R}^m\rightarrow\mathbb{R}^n$, define the Jacobian matrix $(\nabla g)_{ij}=\partial_jg^i$. For any $f:\mathbb{R}^d\rightarrow\mathbb{R}$, define the Hessian matrix $(\nabla^2f)_{ij}=\partial_{ij}f$.
	    \item For two quantities $I$ and $J$, we write $I\lesssim J$ if $I\leq CJ$ for a constant $C>0$ independent of the variables displayed in the estimate; any additional parameter dependence is stated where the estimate is used. We write $I\sim J$ if $I\lesssim J$ and $J\lesssim I$.
	    \item For any matrices $\mathsf X,\mathsf Y\in\mathbb{R}^{m\times n}$, we denote the Frobenius inner product $\mathsf X:\mathsf Y:=\sum_{\substack{1\leq i\leq m\\ 1\leq j\leq n}}\mathsf X_{ij}\mathsf Y_{ij}$.
    \item For any norm $\|\cdot\|_Z$, function $f:\mathbb{R}^d\rightarrow\mathbb{R}$ and $n\in\mathbb{N}$, denote
    \begin{equation*}
	        \|\nabla^nf\|_Z=\sum_{|\beta|=n}\|\partial^\beta f\|_Z.
    \end{equation*}
	    \item For $\gamma\in(0,1)$, $\dot C^{\gamma}$ denotes the homogeneous H\"older class with seminorm
    \begin{equation*}
        \|f\|_{\dot C^{\gamma}}=\sup_{x\neq y}\frac{|f(x)-f(y)|}{|x-y|^{\gamma}},\quad\forall \gamma\in(0,1),
    \end{equation*}
    For any $m\in\mathbb{N}$, set $\|f\|_{\dot C^{m+\gamma}}=\sum_{|\alpha|=m}\|\partial^\alpha f\|_{\dot C^{\gamma}}$ and adopt the integer-order convention $\|f\|_{\dot C^m}:=\|\nabla^m f\|_{L^\infty}$ (in particular, $\dot C^0=L^\infty$ at the level of norms). For the inhomogeneous norm, set
    $\|f\|_{C^{m+\gamma}}:=\sum_{j=0}^{m}\|\nabla^j f\|_{L^\infty}+\|f\|_{\dot C^{m+\gamma}}$.\quad
	    We write $C_b^m(\overline U)$ for functions on $\overline U$ whose derivatives through order $m$ are bounded and continuous up to the boundary.
    \item For the indexed H\"{o}lder space $\dot C_{x'}^{\gamma}$, we define 
    \begin{equation*}
        \|f\|_{\dot C^{\gamma}_{x'}}=\sup_{x_d}\sup_{x'\neq y'}\frac{|f(x',x_d)-f(y',x_d)|}{|x'-y'|^{\gamma}},\quad\forall \gamma\in(0,1),
    \end{equation*}
    and similarly we can define $\dot C^{\gamma}_{x_d}$.
	    \item We denote $\dot B_{p,q}^s$ to be the homogeneous Besov space. Its definition and properties can be found in \cite{BCD}.
    \item For $\alpha\in\mathbb{R}^d$, define $\delta_\alpha f(x)=f(x)-f(x-\alpha)$ whenever both $x$ and $x-\alpha$ belong to the domain of $f$. If $f=f(x,y)$ has several variables, use $\delta_\alpha^x$ or $\delta_\alpha^y$ to specify the variable in which the difference is taken. In addition, we use $\delta_\alpha^{x'}f(x',x_d)$ to denote $f(x',x_d)-f(x'-\alpha,x_d)$ where $\alpha\in\mathbb{R}^{d-1}$, and $\delta_\alpha^{x_d}f=f(x',x_d)-f(x',x_d-\alpha)$ for $\alpha\in\mathbb{R}$.
\end{itemize}
\subsection{The whole-space Schauder estimate}
This subsection recalls the whole-space time-weighted Schauder estimate needed to control the equation on interior localization patches. We recall the order-two estimate from Part I \cite{KHN2025}, which will be used on interior localization patches. Let $u,u_0,f,g$ be $\mathbb R^N$-valued and consider
\begin{equation}\label{eqpara}
\begin{aligned}
    \partial_tu+\mathcal Lu&=\mathcal P_\gamma f+g &&\text{in }(0,T)\times\mathbb R^d,\\
    u(0,\cdot)&=u_0.
\end{aligned}
\end{equation}
Here $\mathcal L$ is the pseudo-differential operator
\[
(\mathcal Lu)(t,x)=(2\pi)^{-d}\int_{\mathbb R^d}
\sL(t,x,\xi)\widehat u(t,\xi)e^{ix\cdot\xi}\,d\xi,
\]
with a symmetric $N\times N$ matrix symbol $\sL(t,x,\xi)$. For some $c_0\in(0,1)$ and $\M>1$, we assume
\begin{equation}\label{condop}
    \sL(t,x,\xi)\geq c_0|\xi|^2\mathrm{Id},
    \qquad
    \sum_{0\leq j,l\leq d+m+4}|\xi|^{l-2}
    \big|\nabla_x^j\nabla_\xi^l\sL(t,x,\xi)\big|\leq\M,
    \qquad \xi\neq0.
\end{equation}
These bounds hold uniformly for $(t,x)\in(0,T)\times\mathbb R^d$. The Fourier multiplier $\mathcal P_\gamma$ is defined in the same convention by an $N\times N$ matrix symbol $\B(\xi)$, which has order $\gamma$ in the sense that
\[
    |\nabla_\xi^j\B(\xi)|\lesssim|\xi|^{\gamma-j},
    \qquad 0\leq j\leq d+m+4,\quad \xi\neq0.
\]
Let $m\in\mathbb N$, $\kappa\in(0,1)$, $0<\gamma\leq2$, and set $\kappa_0=\kappa+\gamma-2>0$. Define
\begin{align*}
\da_T^{m,\kappa,\gamma}(f,g):={}&
\sup_{0<t\leq T}\Big(
 t^{\kappa/2}\|f(t)\|_{\dot C^{\kappa_0}}
+t^{(m+\kappa)/2}\|f(t)\|_{\dot C^{m+\kappa_0}}
+t^{m/2+1}\|g(t)\|_{\dot C^m}\Big)\\
&+\|g\|_{L^1((0,T);L^\infty)}.
\end{align*}
\begin{theorem}[Whole-space Schauder estimate \cite{KHN2025}]\label{propb=0}
Assume \eqref{condop} and
\[
    \|u_0\|_{L^\infty}+\da_T^{m,\kappa,\gamma}(f,g)<\infty.
\]
Then \eqref{eqpara} has a unique distributional solution
$u\in L^\infty((0,T)\times\mathbb R^d;\mathbb R^N)\cap C((0,T];\dot C^{m+\kappa}(\mathbb R^d;\mathbb R^N))$
satisfying
\begin{equation}\label{whole-space-estimate}
\sup_{0<t\leq T}\left(
\|u(t)\|_{L^\infty}
+t^{(m+\kappa)/2}\|u(t)\|_{\dot C^{m+\kappa}}
\right)
\leq C_1e^{C_2T\log(T+2)}
\left(\|u_0\|_{L^\infty}+\da_T^{m,\kappa,\gamma}(f,g)\right).
\end{equation}
The constants depend only on the structural parameters $d,N,m,\kappa,\gamma,c_0$, and $\M$. If $\sL$ is independent of $x$, the exponential factor in \eqref{whole-space-estimate} may be omitted. If, in addition, $u_0$ is bounded and uniformly continuous and $\da_t^{m,\kappa,\gamma}(f,g)\to0$ as $t\downarrow0$, then $u(t)\to u_0$ in $L^\infty$.
\end{theorem}
\subsection{Two toy models on the half-space}
Throughout this subsection, $\kappa\in(0,1)$ is fixed. This subsection introduces two toy-model problems on the half-space that capture the essential analytical difficulties of the nonlinear system. We use notation and estimates from \cite{KHN2025}. Let $\K(t,x)=(4\pi)^{-\frac{d}{2}}t^{-\frac{d}{2}}\exp(-\frac{|x|^2}{4t})$ be the classical heat kernel; then the following pointwise estimate holds.
\begin{lemma}
		This subsection develops two half-space model problems that isolate the boundary estimates required for the nonlinear mean curvature system. The following estimates hold:
        \begin{equation}\label{heates}
		\begin{aligned}
			&|\nabla_x^n \K(t,x)|\lesssim t^{\frac{2-n}{2}}\frac{1}{(t^\frac{1}{2}+|x|)^{d+2}},\quad\forall n\in\mathbb{N},\\
            &\int_{\mathbb{R}^d}|\delta_\alpha\nabla^l_x \K(t,x)||x|^{\sigma}dx\lesssim \frac{|\alpha|^{\sigma}}{t^{\frac{l}{2}}}\min\{1,\frac{|\alpha|}{t^{\frac{1}{2}}}\}^{1-\sigma},\quad \forall \alpha\in\mathbb{R}^d,\ \sigma\in [0,1),\ l\in\mathbb{N}.
		\end{aligned}
        \end{equation}
  Moreover, for $\beta>0$, let $\mathcal{P}_{x'}^\beta$ and $\mathcal{P}_{x_d}^\beta$ denote Fourier multipliers with symbols $P'(\xi')$ and $P_d(\xi_d)$, respectively. Assume that, for every $n\geq0$,
        \begin{equation*}
            \begin{aligned}
                &\mathcal{P}_{x'}^\beta f(x)=(2\pi)^{-d}\int_{\mathbb R^d} P'(\xi')\widehat f(\xi)e^{ix\cdot\xi}\,d\xi,\quad &&\mathcal{P}_{x_d}^\beta f(x)=(2\pi)^{-d}\int_{\mathbb R^d} P_d(\xi_d)\widehat f(\xi)e^{ix\cdot\xi}\,d\xi,\\
                &|\nabla_{\xi'}^nP'(\xi')|\lesssim|\xi'|^{\beta-n},\quad &&|\partial_{\xi_d}^nP_d(\xi_d)|\lesssim |\xi_d|^{\beta-n},
            \end{aligned}
        \end{equation*}
        then, for any $n_1,n_2\geq 0$, we have
        \begin{equation}\label{eq2.32}
            \begin{aligned}
                |\nabla_{x'}^{n_1}\partial_{x_d}^{n_2} \mathcal{P}_{x'}^\beta \K(t,x)|\lesssim (t^\frac{1}{2}+|x|)^{-d}(t^\frac{1}{2}+|x'|)^{-\beta-n_1}(t^\frac{1}{2}+|x_d|)^{-n_2},\\
                |\nabla_{x'}^{n_1}\partial_{x_d}^{n_2} \mathcal{P}_{x_d}^\beta \K(t,x)|\lesssim (t^\frac{1}{2}+|x|)^{-d}(t^\frac{1}{2}+|x'|)^{-n_1}(t^\frac{1}{2}+|x_d|)^{-\beta-n_2}.
            \end{aligned}
        \end{equation}
	\end{lemma}
    For the proof, see Section~2 of \cite{KHN2025}.

    The first toy model on the half-space is
		  \begin{equation}\label{mcbd1}
     \begin{aligned}
         &\partial_tu-\B:\nabla^2 u=F^1+F^2,\quad \text{in}\ [0,T]\times\mathbb{R}^d_+,\\
         &u|_{t=0}=u_0,\quad \text{in}\ \mathbb{R}^d_+,\\
     &u=0,\quad \text{on}\ [0,T]\times\partial\mathbb{R}^d_+,
     \end{aligned}
 \end{equation}
 where $\B\in C^\infty(\mathbb{R}^d_+;\mathbb{R}^{d\times d}_{sym})$ is uniformly elliptic, meaning there exists $c_0\geq1$ such that 
		    \begin{equation}
		        c_0^{-1}|\xi|^2\leq \xi^\top \B \xi  \leq c_0|\xi|^2,\ \ \ \forall\ \xi\in\mathbb{R}^d,
		    \end{equation}
and the forcing terms satisfy
\[
F^1,F^2\in C^\infty((0,T]\times \overline{\mathbb{R}}^d_+),\qquad F^2|_{(0,T]\times\partial \mathbb{R}^d_+}=0.
\]
Generally, $F^1$ contains lower-order derivatives, and $F^2$ consists of critical terms. We require $F^2$ to vanish on the boundary, so that derivatives may be moved from $F^2$ onto the kernel.

		\begin{lemma}\label{mcbdlemti}
		For every smooth function $f:[0,T]\times\mathbb R_+^d\to\mathbb R$ for which the right-hand sides are finite, the following estimates hold.
		\begin{enumerate}
		\item For any $\kappa\in[0,1]$,
		\begin{equation}\label{eq2.25}
		\begin{aligned}
		\sup_{0<\tau<t<T}\frac{\tau^{\frac{\kappa}{2}}
		\|\nabla f(t)-\nabla f(\tau)\|_{L^\infty(\mathbb{R}^d_+)}}
		{(t-\tau)^{\frac{\kappa}{2}}}
		&\lesssim \sup_{0<t\leq T}\Bigl(\|\nabla f(t)\|_{L^\infty(\mathbb{R}^d_+)}\\
		&\qquad+t^{\frac{1}{2}}
		\|(\partial_t,\nabla^2)f(t)\|_{L^\infty(\mathbb{R}^d_+)}\Bigr).
		\end{aligned}
		\end{equation}
		\item For any $\kappa\in(0,1)$, there exists $\varkappa\in(0,\kappa)$ such that
		\begin{equation}\label{eq2.26}
		\begin{aligned}
		\sup_{0<\tau<t<T}\frac{\tau^{\frac{1+\varkappa}{2}}
		\|\nabla^2 f(t)-\nabla^2 f(\tau)\|_{L^\infty(\mathbb{R}^d_+)}}
		{(t-\tau)^{\frac{\varkappa}{2}}}
		&\lesssim \sup_{0<t\leq T}\Bigl(\|\nabla f(t)\|_{L^\infty(\mathbb{R}^d_+)}\\
		&\qquad+t^{\frac{1+\kappa}{2}}
		\|(\partial_t,\nabla^2)f(t)\|_{\dot C^{\kappa}(\mathbb{R}^d_+)}\Bigr).
		\end{aligned}
		\end{equation}
		\end{enumerate}
		\end{lemma}
		The proof of Lemma~\ref{mcbdlemti} is given after the proof of Lemma~\ref{lemheat}.

		\begin{lemma}\label{lemheat}
		    Let $T>0$, $\kappa\in(0,1)$, and let $\eta>0$ satisfy $\kappa+\frac{3\eta}{2}<1$. Let
		    \[
		    u\in C([0,T];W_{\mathrm{loc}}^{1,\infty}(\overline{\mathbb{R}^d_+}))
		    \cap C^\infty((0,T]\times\overline{\mathbb R^d_+})
		    \]
		    be a solution to \eqref{mcbd1}, with $u_0\in \dot W^{1,\infty}(\mathbb{R}^d_+)$. Then
 \begin{equation}\label{re5.1}
    \begin{aligned}
  &  \sup_{0<t\leq T}\left(\|\nabla u(t)\|_{L^\infty(\mathbb{R}^d_+)}+t^\frac{1+\kappa}{2}\|\nabla^2 u(t)\|_{\dot C^{\kappa}(\mathbb{R}^d_+)}\right)\\
    &\quad \lesssim \|u_0\|_{\dot W^{1,\infty}(\mathbb{R}^d_+)}+\sup_{0<t\leq T}\left(t^{\frac{1}{2}}\|F^1(t)\|_{L^\infty(\mathbb{R}^d_+)}+t^\frac{1+\kappa}{2}\|F^1(t)\|_{\dot C^\kappa(\mathbb{R}^d_+)}+t^\frac{1+\kappa}{2}\|F^2(t)\|_{\dot C^\kappa(\mathbb{R}^d_+)}\right)\\
    &\quad+\sup_{0<\tau<t<T}\tau^{\frac{1+\kappa}{2}}\frac{\|F^1(t)-F^1(\tau)\|_{L^\infty(\mathbb{R}^d_+)}}{(t-\tau)^\frac{\kappa}{2}}\\
	    &\qquad+T^{\frac{\eta}{2}}(1+T)\left(\|\B\|_{\dot C^\kappa}+\|\B\|_{\dot W^{1,\infty}}\right)\sup_{0<t\leq T}\left(\|\nabla u(t)\|_{L^\infty(\mathbb{R}^d_+)}+t^{\frac{1+\kappa}{2}}\|(\partial_t,\nabla^2)u(t)\|_{\dot C^{\kappa}(\mathbb{R}^d_+)}\right),
 \end{aligned}
 \end{equation}
		 where the implicit constant depends only on $c_0,d,\kappa$, and $\eta$.
		\end{lemma}
		To close the estimate, use the equation to bound
		\[
		\|\partial_tu\|_{\dot C^\kappa}\lesssim
		\|\B\|_{C^\kappa}\big(\|\nabla^2u\|_{\dot C^\kappa}+\|\nabla^2u\|_{L^\infty}\big)
		+\|F^1+F^2\|_{\dot C^\kappa}.
		\]
		Hence the final term in \eqref{re5.1} is absorbed whenever
		$T^{\eta/2}(1+T)(\|\B\|_{\dot C^\kappa}+\|\B\|_{\dot W^{1,\infty}})$
		is sufficiently small, with the threshold depending only on $c_0,d,\kappa$, and $\eta$.
		\begin{proof}
		We freeze the coefficient at $x_0\in \mathbb{R}^d_+$ and denote $\B_{0}=\B(x_0)$; then \eqref{mcbd1} can be rewritten as
		\begin{align*}
		   \partial_tu(t,x)-\B_{0}:\nabla^2 u(t,x)=F^1(t,x)+F^2(t,x)+\mathsf{R}_{x_0}(t,x),
		\end{align*}
		where
		\begin{equation*}
		    \mathsf{R}_{x_0}(t,x)=(\B(x)-\B_{0}):\nabla^2 u(t,x).
		\end{equation*}
	Since $\B$ is uniformly elliptic, there exists a linear transformation $\psi_{x_0}$ such that $\B_0=\nabla\psi_{x_0}(\nabla\psi_{x_0})^{\top}$ and
    \begin{equation*}
        (\nabla\psi_{x_0})_{dj}=0,\quad\forall j<d.
    \end{equation*}
    This implies $((\B_{0}:\nabla^2u)\circ\psi_{x_0})(x)=(\Delta(u\circ\psi_{x_0}))(x)$ and $\psi_{x_0}$ is a bijection from $\mathbb{R}_+^d$ to $\mathbb{R}_+^d$. Furthermore, the uniform ellipticity ensures 
    \begin{equation}\label{eq2.5}
        |\nabla\psi_{x_0}|+|\nabla\psi_{x_0}^{-1}|\leq C(c_0).\end{equation}
    Let
    \begin{equation}\label{eq2.84}
        v_{x_0}(t,x)=u(t,\psi_{x_0}(x)),\qquad u(t,x)=v_{x_0}(t,\psi_{x_0}^{-1}(x)),
    \end{equation}
    then we obtain the heat equation on the half-space:
 \begin{equation}\label{mcbdheat}
    \begin{aligned}
    &\partial_t v_{x_0}-\Delta v_{x_0}=\tilde F^1_{x_0}+\tilde F^2_{x_0}+\tilde {\mathsf{R}}_{x_0}, \quad \text{in}\ [0,T]\times\mathbb{R}^d_+,\\
         &v_{x_0}|_{t=0}=u_0\circ\psi_{x_0},\quad \text{in}\ \mathbb{R}^d_+,\\
     &v_{x_0}(t,x)=0,\quad \text{on}\ [0,T]\times\partial\mathbb{R}^d_+,
 \end{aligned}
 \end{equation}
 with forcing terms
\begin{equation}\label{eq2.6}
     \tilde F^1_{x_0}(t,x)=(F^1(t)\circ\psi_{x_0})(x), \ \ \ \tilde F^2_{x_0}(t,x)=(F^2(t)\circ\psi_{x_0})(x),\ \ \ \tilde {\mathsf{R}}_{x_0}(t,x)=(\mathsf{R}_{x_0}(t)\circ\psi_{x_0})(x).
\end{equation}
		The transformed forcing $\tilde{F}_{x_0}^2$ still vanishes on the boundary because $\psi_{x_0}(\partial\mathbb{R}_+^d)=\partial\mathbb{R}_+^d$. The Dirichlet heat kernel on the half-space is
		  \begin{align}\label{defH}
		  \sH(t,x'-y',x_d,y_d)= \K(t,x'-y',x_d-y_d)-\K(t,x'-y',x_d+y_d),
		  \end{align}
		  where we denote $x=(x',x_d)\in\mathbb{R}^d$ with $x'\in\mathbb{R}^{d-1}$. Similarly we denote the kernel with Neumann condition
		  \begin{align*}
		  \tilde{\sH}(t,x'-y',x_d,y_d)= \K(t,x'-y',x_d-y_d)+\K(t,x'-y',x_d+y_d).
		  \end{align*}
		  From \eqref{heates}, we obtain for any $x\in\mathbb{R}_+^d$, $n\in\mathbb{N}$,
		  \begin{equation}\label{eskerha}
		  	  \begin{aligned}
			   &\int_{\mathbb{R}^d_+}|\nabla_x ^n \sH(t,x'-y',x_d,y_d)|dy
			   \lesssim t^{-\frac{n}{2}},
			   \quad\quad\int_{\mathbb{R}^d_+}|\nabla_x ^n \tilde{\sH}(t,x'-y',x_d,y_d)|dy
			   \lesssim t^{-\frac{n}{2}},\\
			   &\int_{\mathbb{R}^d_+}|\nabla_x ^n \sH(t,x'-y',x_d,y_d)||x-y|^\sigma dy
			   +\int_{\mathbb{R}^d_+}|\nabla_x ^n \tilde{\sH}(t,x'-y',x_d,y_d)||x-y|^\sigma dy
			   \lesssim t^{\frac{\sigma-n}{2}},\qquad \sigma\in[0,1],\\
			   &\int_{\mathbb{R}^d_+}|\delta_\alpha^x\nabla_x ^n \sH(t,x'-y',x_d,y_d)|
			   |x-y|^{\beta}dy\lesssim t^{-\frac{n}{2}}|\alpha|^{\beta}
			   \min\{1,|\alpha|t^{-\frac{1}{2}}\}^{1-\beta},\\
           &\int_{\mathbb{R}^d_+}|\delta_\alpha^x\nabla_x ^n \tilde{\sH}(t,x'-y',x_d,y_d)|
           |x-y|^{\beta}dy\lesssim t^{-\frac{n}{2}}|\alpha|^{\beta}
           \min\{1,|\alpha|t^{-\frac{1}{2}}\}^{1-\beta},\\
           &\hspace{7em}\alpha\in\mathbb{R}^d,\quad x-\alpha\in\mathbb{R}^d_+,
           \quad\beta\in[0,1).
		  \end{aligned}
		  \end{equation}
		 In fact, to get the last two inequalities of \eqref{eskerha} for $\sH$, we use a change of variable to obtain 
		  \begin{align*}
		    &\int_{\mathbb{R}^d_+}|\delta_\alpha\nabla_x ^n \sH(t,x'-y',x_d,y_d)||x-y|^{\beta}dy\\
		    &\quad\lesssim\int_{\mathbb{R}^d_+}|\delta_\alpha\nabla_x ^n \K(t,x-y)||x-y|^{\beta}dy+\int_{\mathbb{R}^d_-}|\delta_\alpha\nabla_x ^n \K(t,x-y)||(x'-y',x_d+y_d)|^{\beta}dy.
		  \end{align*}
	  For the second term, since $x\in\mathbb{R}_+^d$, and $y\in\mathbb{R}_-^d$, obviously $|(x'-y',x_d+y_d)|^{\beta}\leq |x-y|^{\beta}$, which leads to the estimate. Similar results also hold for $\tilde{\sH}$.

Using $\sH$, we rewrite the solution of \eqref{mcbdheat} as
		  \begin{equation}\label{mcbdfor}
     \begin{aligned}
         v_{x_0}(t,x)&=\int_{\mathbb{R}_+^d} \sH(t,x'-y',x_d,y_d)v_{x_0}(0,y)dy+\sum_{k=1}^2\int_0^t\int_{\mathbb{R}^d_+} \sH(t-\tau,x'-y',x_d,y_d)\tilde F^k_{x_0}(\tau,y)dyd\tau\\
         &\quad\quad+\int_0^t\int_{\mathbb{R}^d_+} \sH(t-\tau,x'-y',x_d,y_d)\tilde {\mathsf{R}}_{x_0}(\tau,y)dyd\tau\\
         &:=v_{L,x_0}(t,x)+\sum_{k=1,2}v_{N,x_0}^k(t,x)+v_{R,x_0}(t,x).
           \end{aligned}
           	  \end{equation}
              By \eqref{eq2.84}, for any $x,x-\alpha\in\mathbb{R}_+^d$, set
              \begin{equation*}
                  z_{x_0,x}:=\psi_{x_0}^{-1}(x),\qquad
                  \beta:=z_{x_0,x}-z_{x_0,x-\alpha}
                  =(\nabla\psi_{x_0}^{-1})\alpha.
              \end{equation*}
              We regard gradients of scalar functions as row vectors, consistently
              with the Jacobian convention in Section~\ref{sec2}. Since
              $\psi_{x_0}$ is linear, the chain rule gives
              \begin{align}
              \nabla u(t,x)
              &=\Bigl(\nabla v_{L,x_0}(t,z_{x_0,x})
              +\sum_{k=1,2}\nabla v_{N,x_0}^k(t,z_{x_0,x})
              +\nabla v_{R,x_0}(t,z_{x_0,x})\Bigr)
              \nabla\psi_{x_0}^{-1},\label{eq2.86}\\
              \delta_\alpha^x\nabla^2u(t,x)
              &=(\nabla\psi_{x_0}^{-\top})
              \Bigl(\delta_\beta^z\nabla^2v_{L,x_0}(t,z_{x_0,x})
              +\sum_{k=1,2}\delta_\beta^z\nabla^2v_{N,x_0}^k(t,z_{x_0,x})\nonumber\\
              &\hspace{12em}+\delta_\beta^z\nabla^2v_{R,x_0}(t,z_{x_0,x})\Bigr)
              (\nabla\psi_{x_0}^{-1}).\nonumber
              \end{align}
	             	  These identities hold for every fixed $x_0\in\mathbb{R}^d_+$. Since
              $\psi_{x_0}^{\pm1}$ are uniformly bounded linear maps,
                      \begin{equation*}
                          |\alpha|\lesssim |\beta|\lesssim |\alpha|.
                      \end{equation*}
                      Here and below, all derivatives and finite differences in the
transformed variable are first taken with $x_0$ fixed. Only afterwards
do we set $x_0=x$ and evaluate at $z_{x_0,x}=\psi_{x_0}^{-1}(x)$.
Thus, for instance, $\left.w_{x_0}(z_{x_0,x})\right|_{x_0=x}=w_x(\psi_x^{-1}(x)).$
	            	  With this convention,
           	  \begin{equation}\label{spliv}
           	       \begin{aligned}
           	    & |\nabla v_{x_0}(t)|\leq |(\nabla v_{L,x_0}(t))|+\sum_{k=1,2}|(\nabla v_{N,x_0}^k(t))|+|(\nabla v_{R,x_0}(t))|,\\
           	     &|\delta_\alpha\nabla^2v_{x_0}(t)|\leq |(\delta_\alpha\nabla^2v_{L,x_0}(t))|+\sum_{k=1,2}|(\delta_\alpha\nabla^2v_{N,x_0}^k(t))|\\
           	     &\quad\quad\quad\quad\quad\quad+|(\delta_\alpha\nabla^2v_{R,x_0}(t))|.
           	  \end{aligned}
           	  \end{equation}
           	  We first estimate the linear part $v_{L,x_0}$. Using integration by parts, since $v_{x_0}|_{\partial\mathbb{R}_+^d}=0$, one has 
 \begin{align*}
    &\nabla_{x'} v_{L,x_0}(t,x)=\int_{\mathbb{R}^d_+}\nabla_{x'} \sH(t,x'-y',x_d,y_d)v_{x_0}(0,y)dy
    =\int_{\mathbb{R}^d_+} \sH(t,x'-y',x_d,y_d)\nabla_{y'} v_{x_0}(0,y)dy,\\
    &\partial_{x_d} v_{L,x_0}(t,x)=\int_{\mathbb{R}^d_+}\partial_{x_d} \sH(t,x'-y',x_d,y_d)v_{x_0}(0,y)dy
    =\int_{\mathbb{R}^d_+} \tilde{\sH}(t,x'-y',x_d,y_d)\partial_{y_d} v_{x_0}(0,y)dy.
 \end{align*}
 Hence, by \eqref{eskerha} we obtain 
 \begin{equation*}
 \begin{aligned}
    &\|\nabla v_{L,x_0}(t)\|_{L^\infty}\lesssim \sup_{x\in\mathbb{R}_+^d}\int_{\mathbb{R}_+^d}|\tilde{\sH}(t,x'-y',x_d,y_d)|dy\|\nabla v_{x_0}(0)\|_{L^\infty(\mathbb{R}^d_+)}\\
    &\qquad+\sup_{x\in\mathbb{R}_+^d}\int_{\mathbb{R}_+^d}|\sH(t,x'-y',x_d,y_d)|dy\|\nabla v_{x_0}(0)\|_{L^\infty(\mathbb{R}^d_+)}\lesssim \|\nabla u_0\|_{L^\infty(\mathbb{R}^d_+)},
 \\
   &\| \delta_\alpha\nabla^2 v_{L,x_0}(t)\|_{L^\infty}\lesssim \int_{\mathbb{R}^d_+} |\delta_\alpha\nabla_x\tilde{\sH}(t,x'-y',x_d,y_d)|dy\|\nabla v_{x_0}(0)\|_{L^\infty(\mathbb{R}^d_+)}\\
   &\qquad+\int_{\mathbb{R}^d_+} |\delta_\alpha\nabla_x\sH(t,x'-y',x_d,y_d)|dy\|\nabla v_{x_0}(0)\|_{L^\infty(\mathbb{R}^d_+)}\lesssim \min\left\{1,\frac{|\alpha|}{t^{\frac{1}{2}}}\right\}t^{-\frac{1}{2}}\|\nabla u_0\|_{L^\infty(\mathbb{R}^d_+)},
 \end{aligned}
 \end{equation*}
holds uniformly for any $x_0\in\mathbb{R}_+^d$. Then we get
 \begin{align}\label{lllha}
     \|\nabla v_{L,x_0}(t)\|_{L^\infty}+t^\frac{1+\kappa}{2}\sup_{\substack{\alpha\in\mathbb R^d\\x,x-\alpha\in\mathbb{R}_+^d}}\frac{\|\delta_\alpha\nabla^2 v_{L,x_0}(t)\|_{L^\infty}}{|\alpha|^\kappa}\lesssim \|u_0\|_{\dot W^{1,\infty}},
 \end{align}
holds for any $x_0\in\mathbb{R}_+^d$. Then we consider $v_{N,x_0}^1(t,x)$. We have 
 \begin{align*}
    \nabla v_{N,x_0}^1(t,x)=\int_0^t\int_{\mathbb{R}^d_+} \nabla_x \sH(t-\tau,x'-y',x_d,y_d)\tilde F^1_{x_0}(\tau,y)dyd\tau.
 \end{align*}
Applying \eqref{eskerha} and using $\int_0^t(t-\tau)^{-\frac{1}{2}}\tau^{-\frac{1}{2}}\,d\tau\lesssim 1$, we obtain
 \begin{align}\label{unlinft}
    \|\nabla v_{N,x_0}^1(t)\|_{L^\infty}\lesssim  \int_0^t (t-\tau)^{-\frac{1}{2}}\|\tilde F^1_{ x_0}(\tau)\|_{L^\infty} d\tau\lesssim \sup_{0<t\leq T}t^{\frac{1}{2}}\|F^1(t)\|_{ L^\infty}.
 \end{align}
 For higher-order derivatives, it suffices to consider $\nabla_{x'}\nabla v_{N,x_0}^1$ and $\partial_{x_d}^2  v_{N,x_0}^1$ respectively. Note that
 \begin{align*}
    &\delta_\alpha^x\nabla_{x'}\nabla v_{N,x_0}^1(t,x)=\int_0^t\int_{\mathbb{R}^d_+}\delta_\alpha^x\nabla_{x'}\nabla_x \sH(t-\tau,x'-y',x_d,y_d)\tilde F^1_{x_0}(\tau,y)dyd\tau\\
    &\quad\quad=-\int_0^t\int_{\mathbb{R}^d_+}\delta_\alpha^x\nabla_{y'}|D_{y'}|^{-\varkappa}\nabla_x \sH(t-\tau,x'-y',x_d,y_d)|D_{y'}|^{\varkappa}\tilde F^1_{x_0}(\tau,y)dyd\tau\\
    &\quad\quad=-\int_0^t\int_{\mathbb{R}^d_+}
    \delta_\alpha^x\nabla_{y'}|D_{y'}|^{-\varkappa}\nabla_x
    \sH(t-\tau,x'-y',x_d,y_d)\\
    &\qquad\qquad\times
    \left(|D_{y'}|^{\varkappa}\tilde F^1_{x_0}(\tau,y',y_d)
    -|D_{y'}|^{\varkappa}\tilde F^1_{x_0}(\tau,x',y_d)\right)dyd\tau,
 \end{align*}
	for some $\varkappa\in(0,\kappa)$. By \eqref{eq2.32}, interpolation, and singular integral theory, we obtain
 \begin{equation}\label{dxpdx}
     \begin{aligned}
    &\|\delta_\alpha^x\nabla_{x'}\nabla  v_{N,x_0}^1(t) \|_{L^\infty}\\
    &\quad\lesssim \int _0^t \int_{\mathbb{R}^d_+} |\delta_\alpha^x \nabla_{y'}|D_{y'}|^{-\varkappa}\nabla_x\sH(t-\tau,x'-y',x_d,y_d)||x'-y'|^{\kappa-\varkappa} dy  \||D_{y'}|^{\varkappa}\tilde F^1_{x_0}(\tau)\|_{\dot C_{x'}^{\kappa-\varkappa}(\mathbb{R}^d_+)} d\tau\\
   &\quad\lesssim \int_0^t (t-\tau)^{-\frac{2-\varkappa}{2}}|\alpha|^{\kappa-\varkappa}
  \min\Big\{1,\tfrac{|\alpha|}{(t-\tau)^{1/2}}\Big\}^{1-\kappa+\varkappa}
  \tau^{-\frac{1+\kappa}{2}}\,d\tau
  \cdot \sup_{0<t\leq T} t^{\frac{1+\kappa}{2}}\|F^1(t)\|_{\dot C^\kappa}\\
   &\quad\lesssim  |\alpha|^\kappa t^{-\frac{1+\kappa}{2}}\sup_{0<t\leq T}t^\frac{1+\kappa}{2}\|F^1(t)\|_{\dot C^\kappa(\mathbb{R}^d_+)},
 \end{aligned}
 \end{equation}
 The last time integral is estimated by splitting at $\tau=t/2$ and, when $|\alpha|^2<t$, at $\tau=t-|\alpha|^2$; the resulting bound is $|\alpha|^\kappa t^{-(1+\kappa)/2}$ in every regime. We also use
 \begin{equation}\label{eq2.14}
     \begin{aligned}
         \||D_{y'}|^{\varkappa}\tilde F^1_{x_0}(\tau)\|_{\dot C_{x'}^{\kappa-\varkappa}}&\sim \||D_{y'}|^{\varkappa}\tilde F^1_{x_0}(\tau)\|_{L^\infty \dot B_{\infty,\infty}^{\kappa-\varkappa}(\mathbb{R}_+\times\mathbb{R}^{d-1})}\sim \|\tilde F^1_{x_0}(\tau)\|_{L^\infty \dot B_{\infty,\infty}^{\kappa}(\mathbb{R}_+\times\mathbb{R}^{d-1})}\\
         &\sim \|\tilde F^1_{x_0}(\tau)\|_{L^\infty \dot C^{\kappa}(\mathbb{R}_+\times\mathbb{R}^{d-1})}\lesssim \|\tilde F^1_{x_0}(\tau)\|_{\dot C^{\kappa}(\mathbb{R}_+^d)},
     \end{aligned}
 \end{equation}
and the uniform bound of $\psi_{x_0}$. For $\partial_{x_d}^2  v_{N,x_0}^1$, we use the following property of the kernel,
 \begin{equation}\label{eq2.15}
     \partial_{x_d}^2\sH(t,x'-y',x_d,y_d)=\partial_t\sH(t,x'-y',x_d,y_d)-\Delta_{x'}\sH(t,x'-y',x_d,y_d),
 \end{equation}
 then we can write 
 \begin{align*}
    \partial_{x_d}^2 v_{N,x_0}^1(t,x)=\int_0^t\int_{\mathbb{R}^d_+} \partial_t\sH(t-\tau,x'-y',x_d,y_d)\tilde F^1_{x_0}(\tau,y)dyd\tau- \Delta_{x'}v_{N,x_0}^1(t,x).
 \end{align*}
 By \eqref{dxpdx}, we know that 
 \begin{align*}
  \sup_{0<t\leq T}t^\frac{1+\kappa}{2} \sup_{\alpha} \frac{\|\delta_\alpha\Delta_{x'}v_{N,x_0}^1(t)\|_{L^\infty}}{|\alpha|^\kappa}\lesssim \sup_{0<t\leq T}t^\frac{1+\kappa}{2}\|{F}^1(t)\|_{\dot   C^\kappa(\mathbb{R}^d_+)}.
 \end{align*}
 Hence it suffices to consider 
 \begin{align*}
   & \int_0^t\int_{\mathbb{R}^d_+} \partial_t\sH(t-\tau,x'-y',x_d,y_d)\tilde F^1_{ x_0}(\tau,y)dyd\tau\\
    &\quad\quad\quad\quad\quad=  \int_0^t\int_{\mathbb{R}^d_+} \partial_t\sH(t-\tau,x'-y',x_d,y_d)(\tilde F^1_{ x_0}(\tau,y)-\tilde F^1_{x_0}(t,y))dyd\tau\\
    &\quad\quad\quad\quad\quad\quad\quad+ \int_0^t\int_{\mathbb{R}^d_+} \partial_t\sH(t-\tau,x'-y',x_d,y_d)\tilde F^1_{ x_0}(t,y)dyd\tau\\
    &\quad\quad\quad\quad\quad=  \int_0^t\int_{\mathbb{R}^d_+} \partial_t\sH(t-\tau,x'-y',x_d,y_d)(\tilde F^1_{ x_0}(\tau,y)-\tilde F^1_{x_0}(t,y))dyd\tau\\
    &\quad\quad\quad\quad\quad\quad\quad+ \int_{\mathbb{R}^d_+} \sH(t,x'-y',x_d,y_d)\tilde F^1_{ x_0}(t,y)dy-\tilde F^1_{ x_0}(t,x)\\
    &\quad\quad\quad\quad\quad=: I^1_{ x_0}(t,x)+I^2_{ x_0}(t,x)-\tilde F^1_{x_0}(t,x).
 \end{align*}
Here we use $\sH(t,x'-y',x_d,y_d)|_{t=0}=\delta_{x=y}$, where $\delta_{x=y}$ denotes the Dirac measure. By \eqref{eskerha}, the uniform bounds for $\psi_{x_0}$, and the fundamental theorem of calculus,
\begin{align*}
&\left|\delta_\alpha I^1_{ x_0}(t,x) \right|\\
&\quad\lesssim \int_0^t\int_{\mathbb{R}^d_+}| \delta_\alpha\partial_t\sH(t-\tau,x'-y',x_d,y_d)|dy|t-\tau|^\frac{\kappa}{2}\tau^{-\frac{1+\kappa}{2}}d\tau \sup_{0<\tau<t<T}\tau^{\frac{1+\kappa}{2}}\frac{\|\tilde{F}_{x_0}^1(t)-\tilde{F}^1_{x_0}(\tau)\|_{L^\infty}}{(t-\tau)^\frac{\kappa}{2}}\\
&\quad\lesssim |\alpha|^\kappa t^{-\frac{1+\kappa}{2}}\sup_{0<\tau<t<T}\tau^{\frac{1+\kappa}{2}}\frac{\|F^1(t)- F^1(\tau)\|_{L^\infty}}{(t-\tau)^\frac{\kappa}{2}},
\end{align*}
and
 \begin{align*}
    \left|\delta_\alpha I^2_{ x_0}(t,x)\right|\lesssim |\alpha|^\kappa t^{-\frac{\kappa}{2}}\|F^1(t)\|_{L^\infty}.
 \end{align*}
 Hence we conclude that 
 \begin{align*}
   &\sup_{0<t\leq T}t^\frac{1+\kappa}{2}\sup_\alpha\frac{ \|\delta_\alpha \partial_{x_d}^2 v_{N,x_0}^1(t)\|_{L^\infty}}{|\alpha|^\kappa}\\
   &\quad\quad\lesssim \sup_{0<\tau<t<T}\tau^{\frac{1+\kappa}{2}}\frac{\|F^1(t)-F^1(\tau)\|_{L^\infty(\mathbb{R}^d_+)}}{(t-\tau)^\frac{\kappa}{2}}+\sup_{0<t\leq T}\left(t^{\frac{1}{2}}\|F^1(t)\|_{L^\infty(\mathbb{R}^d_+)}+t^\frac{1+\kappa}{2}\|F^1(t)\|_{\dot C^\kappa(\mathbb{R}^d_+)}\right).
 \end{align*}
 Combining this with \eqref{unlinft} and \eqref{dxpdx}, and noting that these estimates hold for every $x_0$, in particular for $x_0=x$, we obtain
 \begin{equation}\label{unnha}
 \begin{aligned}
    & \sup_{0<t\leq T}\left(\|\nabla v_{N,x_0}^1(t)\|_{L^\infty(\mathbb{R}^d_+)}+t^\frac{1+\kappa}{2}\sup_{\alpha}\frac{\| \delta_\alpha\nabla^2 v_{N,x_0}^1(t)\|_{L^\infty(\mathbb{R}^d_+)}}{|\alpha|^\kappa}\right)\\
     &\quad\lesssim \sup_{0<\tau<t<T}\tau^{\frac{1+\kappa}{2}}\frac{\|F^1(t)-F^1(\tau)\|_{L^\infty(\mathbb{R}^d_+)}}{(t-\tau)^\frac{\kappa}{2}}+\sup_{0<t\leq T}\left(t^{\frac{1}{2}}\|F^1(t)\|_{L^\infty(\mathbb{R}^d_+)}+t^\frac{1+\kappa}{2}\|F^1(t)\|_{\dot C^\kappa(\mathbb{R}^d_+)}\right).
 \end{aligned}
 \end{equation}
 Next we consider $v_{N,x_0}^2$. Note that by assumption, $\tilde F^2_{ x_0}$ vanishes on boundary, hence we can denote the odd extension of $\tilde{F}^2_{x_0}$ on $(0,T]$ by
\begin{align*}
    \tilde F^{2,odd}_{ x_0}(t,x)=\begin{cases}
        \tilde F^2_{ x_0}(t,x',x_d),\ \ \ x\in\mathbb{R}^d_+,\\
        -\tilde F^2_{ x_0}(t,x',-x_d),\ \ \ x\in \mathbb{R}^d_-.
    \end{cases}
 \end{align*} 
 By the boundary condition of $\tilde{F}_{x_0}^2$, the odd extension is well-defined, and its $\dot C^\kappa$ seminorm is controlled by that of $\tilde F^2_{x_0}$; namely,
 \begin{equation}\label{eq2.16}
    \|\tilde F^{2,odd}_{x_0}\|_{\dot C^\kappa(\mathbb{R}^d)}\lesssim \| \tilde F^2_{x_0}\|_{\dot C^\kappa(\mathbb{R}^d_+)}\lesssim \|F^2\|_{\dot C^\kappa(\mathbb{R}^d_+)},\qquad \forall x_0\in\mathbb{R}^d_+.
 \end{equation}
By the odd extension, we can rewrite the equation of $v_{N,x_0}^2$ as
 \begin{align*}
   v_{N,x_0}^2(t,x)&=\int_0^t \int_{\mathbb{R}^d_+}\sH(t-\tau,x'-y',x_d,y_d)\tilde F^2_{ x_0}(\tau,y)dyd\tau\\
    &=\int_0^t \int_{\mathbb{R}^d}\K(t-\tau,x-y)\tilde F^{2,odd}_{x_0}(\tau,y)dyd\tau.
 \end{align*}
 We have 
 \begin{equation}\label{unlha}
    \begin{aligned}
    \|\nabla v_{N, x_0}^2(t)\|_{L^\infty}&\lesssim \left\|\int_0^t \int_{\mathbb{R}^d} \nabla \K(t-\tau,x-y)(\tilde F^{2,odd}_{ x_0}(\tau,y)-\tilde F^{2,odd}_{ x_0}(\tau,x))dyd\tau \right\|_{L^\infty}\\
    &\lesssim \int_0^t \int_{\mathbb{R}^d} |\nabla \K(t-\tau,y)||y|^\kappa dy \|\tilde F^{2,odd}_{x_0}(\tau)\|_{\dot C^\kappa} d\tau \\
    &\lesssim \int_0^t(t-\tau)^{-\frac{1-\kappa}{2}}\tau^{-\frac{1+\kappa}{2}}d\tau \sup_{0<t\leq T}t^\frac{1+\kappa}{2}\| F^{2}(t)\|_{\dot C^\kappa}\\
    &\lesssim \sup_{0<\tau\leq T}\tau^\frac{1+\kappa}{2}\| F^{2}(\tau)\|_{\dot C^\kappa}.
    \end{aligned}
 \end{equation}
 On the other hand, for H\"{o}lder estimates of $v_{N,x_0}^2$, we will use the methods in \cite{KHN2025}. For some $\varkappa\in(0,\kappa)$,
 set
 \[
 \mathscr K_{\alpha,\varkappa}(s,z)
 :=\delta_\alpha\nabla_z^2|D_z|^{-\varkappa}\K(s,z).
 \]
 \begin{equation}\label{unhha}
 \begin{aligned}
     \|\delta_\alpha^x \nabla^2v_{N,x_0}^2(t)\|_{L^\infty}
     &\lesssim \left\| \int_0^t\int_{\mathbb{R}^d}
     \mathscr K_{\alpha,\varkappa}(t-\tau,x-y)
     |D_y|^{\varkappa}\tilde F^{2,\mathrm{odd}}_{x_0}(\tau,y)dyd\tau\right\|_{L^\infty} \\
     &\lesssim \left\| \int_0^t\int_{\mathbb{R}^d}
     \mathscr K_{\alpha,\varkappa}(t-\tau,x-y)|D_y|^{\varkappa}
     \left(\tilde F^{2,\mathrm{odd}}_{x_0}(\tau,y)
     -\tilde F^{2,\mathrm{odd}}_{x_0}(\tau,x)\right)dyd\tau\right\|_{L^\infty} \\
     &\lesssim \int_0^t\left\|\int_{\mathbb{R}^d}
     \mathscr K_{\alpha,\varkappa}(t-\tau,x-y)|x-y|^{\kappa-\varkappa}dy
     \right\|_{L^\infty}
     \||D_y|^{\varkappa}\tilde{F}^{2,\mathrm{odd}}_{x_0}(\tau)
     \|_{\dot C^{\kappa-\varkappa}}d\tau \\
     &\lesssim \int_0^t(t-\tau)^{-\frac{2-\varkappa}{2}}
     |\alpha|^{\kappa-\varkappa}
     \min\left\{1,\frac{|\alpha|}{(t-\tau)^{\frac{1}{2}}}\right\}^{1-\kappa+\varkappa}
     \tau^{-\frac{1+\kappa}{2}}d\tau\\
     &\qquad\times\sup_{0<\tau\leq T}\tau^{\frac{1+\kappa}{2}}
     \|F^2(\tau)\|_{\dot C^{\kappa}}\\
     &\lesssim|\alpha|^{\kappa}t^{-\frac{1+\kappa}{2}}\sup_{0<\tau\leq T}\tau^{\frac{1+\kappa}{2}}\|F^2(\tau)\|_{\dot C^\kappa},
 \end{aligned}
 \end{equation}
 where we use \eqref{eq2.16}, the uniform bound of $\psi_{x_0}$ and a similar argument as \eqref{eq2.14}. We conclude from  \eqref{unlha} and \eqref{unhha} that 
  \begin{equation}\label{ununun2}
 \begin{aligned}
     & \sup_{0<t\leq T}\left(
     \|\nabla v_{N,x_0}^2(t)\|_{L^\infty(\mathbb{R}^d_+)}
     +t^\frac{1+\kappa}{2}\sup_{\alpha}
     \frac{\|\delta_\alpha\nabla^2 v_{N,x_0}^2(t)\|_{L^\infty(\mathbb{R}^d_+)}}
     {|\alpha|^\kappa}\right)\\
     &\qquad\lesssim \sup_{0<t\leq T}t^\frac{1+\kappa}{2}
     \|F^2(t)\|_{\dot C^\kappa}.
 \end{aligned}
 \end{equation}
 For the remainder term $v_{R,x_0}$, we follow the coefficient-freezing argument of \cite{KHN2025}. The increment $\B(y)-\B(x_0)$ supplies the spatial factor needed to treat this term as a lower-order contribution; we also use Lemma~\ref{mcbdlemti} for its time increments.
Since by \eqref{eq2.5}, $\nabla\psi_{x_0}$ is uniformly bounded, we have the following estimate for $\tilde{\R}_{x_0}$ defined in \eqref{eq2.6} that
 \begin{equation*}
     \begin{aligned}
         |\tilde {\mathsf{R}}_{ x_0}(t,y)|\lesssim |\B(x_0)-\B(\psi_{x_0}(y))| \|\nabla^2u(t)\|_{L^\infty}&\lesssim|x_0-\psi_{x_0}(y)|\|\B\|_{\dot {W}^{1,\infty}}\|\nabla^2u(t)\|_{L^\infty}\\
         &\lesssim |y-z_{x_0,x_0}|\|\B\|_{\dot {W}^{1,\infty}}\|\nabla^2u(t)\|_{L^\infty}.
     \end{aligned}
 \end{equation*}
 So by \eqref{eskerha}, we take $x_0=x$, and obtain
 \begin{equation*}
     \begin{aligned}
         \left|(\nabla v_{R,x_0}(t,z_{x_0,x}))|_{x_0=x}\right|
        & \lesssim \int_0^t \int_{\mathbb{R}_+^d}|\nabla_x\sH(t-\tau,(z_{x_0,x})'-y',(z_{x_0,x})_d,y_d)||z_{x_0,x}-y|dy\|\nabla^2u(\tau)\|_{L^\infty}d\tau\|\B\|_{\dot W^{1,\infty}}\\
         &\lesssim T^{\frac{1}{2}}\|\B\|_{\dot W^{1,\infty}}\sup_{0<\tau\leq T}\tau^{\frac{1}{2}}\|\nabla^2u(\tau)\|_{L^\infty}.
   \end{aligned}
 \end{equation*}
 For the normal derivative, the preceding calculation is justified by writing
 $\partial_{x_d}=\lim_{\beta\to0^+}\delta_{\beta e_d}^{x}/\beta$.
 For fixed $\beta>0$ the finite difference may be passed through the integral; the kernel bounds above provide an integrable majorant, and dominated convergence gives the asserted normal-derivative estimate.
By \eqref{eq2.86}, we will estimate $\delta_\beta^z(\nabla^2v_{R,x_0}(t,z_{x_0,x}))|_{x_0=x}$. For $\nabla^2 v_{R,x_0}$, we consider $\nabla_{z'}\nabla$ and $\partial_{d}^2$ separately. For the former one, note that
 \begin{equation*}
     \begin{aligned}
         &\delta_\beta^z\nabla_{z'}\nabla v_{R,x_0}(t,z_{x_0,x})|_{x_0=x} \\
         &\quad= \int_0^t\int_{\mathbb{R}_+^d}\delta_\beta^z \nabla_{z'}\nabla|D_{y'}|^{-\eta}\sH(t-\tau,(z_{x_0,x})'-y',(z_{x_0,x})_d,y_d)\left(|D_{y'}|^\eta\tilde{\R}_{x_0}(\tau,y)\right)dyd\tau\big|_{x_0=x},
     \end{aligned}
 \end{equation*}
 where we take $\eta\in(0,\min\{\kappa,\frac{2}{3}(1-\kappa)\})$. We use the fractional-Laplacian representation
 \begin{equation*}
     |D_{y'}|^{\eta}f(y)=C_{d,\eta}P.V.\int_{\mathbb{R}^{d-1}}\frac{\delta_h^{y'} f(y)}{|h|^{d-1+\eta}}dh,
 \end{equation*}
where we recall that $\delta_\alpha^{y'}f(y)=f(y',y_d)-f(y'-\alpha,y_d)$ is the finite difference with respect to $y'$. This leads to
 \begin{equation*}
     |D_{y'}|^{\eta}(fg)(y)=g(y)|D_{y'}|^{\eta}f(y)+C_{d,\eta}\,\mathrm{p.v.}\!\int_{\mathbb{R}^{d-1}}\frac{f(y'-h,y_d)\delta_h^{y'} g(y)}{|h|^{d-1+\eta}}\,dh.
 \end{equation*}
 By this equality, we denote $\tilde{\B}_{x_0}(y)=\B(\psi_{x_0}(y))-\B(x_0)$, and $\tilde{\R}_{x_0}(\tau,y)=\tilde{\B}_{x_0}(y):(\nabla^2u(\tau)\circ\psi_{x_0})(y)$. We can write
 \begin{equation*}
    \begin{aligned}
         &|D_{y'}|^\eta\tilde{\R}_{x_0}(\tau,y)=\tilde{\B}_{x_0}(y):|D_{y'}|^\eta\left((\nabla^2u(\tau))\circ\psi_{x_0}\right)+C_{d,\eta}\int_{\mathbb{R}^{d-1}}\frac{(\nabla^2u(\tau))\circ\psi_{x_0}(y'-h,y_d)\delta_h^{y'} \tilde{\B}_{x_0}(y)}{|h|^{d-1+\eta}}dh\\
         &\qquad:=\tilde{\R}_{x_0}^1(\tau,y)+\tilde{\R}_{x_0}^2(\tau,y)
    \end{aligned}
 \end{equation*}
 Note that
 \begin{equation*}
     \begin{aligned}
         &\left. |\tilde{\B}_{x_0}(y)|\right|_{x_0=x}\lesssim \|\B\|_{\dot C^\kappa}|z_{x_0,x}-y|^{\kappa},\\
         &|\delta_h^{y'} \tilde{\B}_x(y)|\lesssim(\|\B\|_{ \dot W^{\frac{\eta}{2},\infty}}+\|\B\|_{ \dot W^{\frac{3\eta}{2},\infty}})\min\{1,|h|^{\frac{3}{2}\eta},|h|^{\frac{\eta}{2}}\},\\
         &|\delta_h^y\delta_\beta^{y'}\tilde{\B}_x(y)|\lesssim \min\{|h|,|\beta|\}^{\kappa+\frac{3}{2}\eta}\|\B\|_{\dot W^{\kappa+\frac{3\eta}{2},\infty}},
     \end{aligned}
 \end{equation*}
     when $\kappa+\frac{3}{2}\eta< 1$. This directly gives
     \begin{align*}
        & \left| \left.(\tilde{\R}_{x_0}^2(\tau,y)-\tilde{\R}_{x_0}^2(\tau,z_{x_0,x}))\right|_{x_0=x}\right|\lesssim |z_{x_0,x}-y|^\eta\|\nabla^2u(\tau)\|_{\dot C^{\eta}}\|\B\|_{\dot W^{\frac{\eta}{2},\infty}\cap \dot W^{\frac{3\eta}{2},\infty}}\\
	         &\quad\quad\quad+|z_{x_0,x}-y|^\kappa\|\nabla^2u(\tau)\|_{L^\infty}\|\B\|_{\dot W^{\kappa+\frac{\eta}{2},\infty}\cap \dot W^{\kappa+\frac{3\eta}{2},\infty}}
     \end{align*}
	     For brevity, write
	     \[
	     \mathscr Q_{h,\eta}(s;x,y)
	     :=\delta_h^x\nabla_{x'}\nabla|D_{y'}|^{-\eta}
	     \sH(s,x'-y',x_d,y_d).
	     \]
	     Combining the preceding estimates, we deduce that
 \begin{equation}\label{eq2.22}
     \begin{aligned}
         &\left|\left(\delta_\beta^z\nabla_{z'}\nabla v_{R,x_0}(t,z_{x_0,x})\right)|_{x_0=x}\right|\\
         &\quad\lesssim \left|\int_0^t\int_{\mathbb{R}_+^d}
         \left(\mathscr Q_{\beta,\eta}(t-\tau;z_{x_0,x},y)
         \tilde{\R}_{x_0}^1(\tau,y)\right)|_{x_0=x}dyd\tau\right|  \\
         &\quad\quad+\left|\int_0^t\int_{\mathbb{R}_+^d}
         \left(\mathscr Q_{\beta,\eta}(t-\tau;z_{x_0,x},y)
         (\tilde{\R}_{x_0}^2(\tau,y)-\tilde{\R}_{x_0}^2(\tau,z_{x_0,x}))\right)|_{x_0=x}dyd\tau\right|  \\
         &\quad\lesssim \int_0^t\int_{\mathbb{R}_+^d}
         |\mathscr Q_{\beta,\eta}(t-\tau;z_{x_0,x},y)|\,|z_{x_0,x}-y|^{\kappa}
         \||D_{y'}|^\eta\nabla^2u(\tau)\|_{L^\infty}dyd\tau
         \|\B\|_{\dot C^\kappa}\\
	         &\quad\quad+ \int_0^t\int_{\mathbb{R}_+^d}
		         |\mathscr Q_{\beta,\eta}(t-\tau;z_{x_0,x},y)|\,|z_{x_0,x}-y|^{\eta}
	         \|\nabla^2u(\tau)\|_{\dot C^{\eta}}dyd\tau
	         \|\B\|_{\dot W^{\frac{\eta}{2},\infty}\cap \dot W^{\frac{3}{2}\eta,\infty}}\\
	         &\quad\quad+ \int_0^t\int_{\mathbb{R}_+^d}
		         |\mathscr Q_{\beta,\eta}(t-\tau;z_{x_0,x},y)|\,|z_{x_0,x}-y|^{\kappa}
	         \|\nabla^2u(\tau)\|_{L^\infty}dyd\tau
	         \|\B\|_{\dot W^{\kappa+\frac{\eta}{2},\infty}\cap
	         \dot W^{\kappa+\frac{3\eta}{2},\infty}}\\
         &\quad\lesssim |\beta|^{\kappa}t^{-\frac{1+\kappa-\eta}{2}}
         (\|\B\|_{\dot W^{\frac{\eta}{2},\infty}}+\|\B\|_{\dot W^{1,\infty}})\\
         &\qquad\times\sup_{0<t\leq T}\left(t^{\frac{1}{2}}\|\nabla^2u(t)\|_{L^\infty}
         +t^{\frac{1+\eta}{2}}\|\nabla^2u(t)\|_{\dot C^{\eta}}\right),
     \end{aligned}
 \end{equation}
 Here we use interpolation, \eqref{eq2.32}, and the uniform bound of $\psi_x$. For the $\partial_d^2$ term, by using the equality \eqref{eq2.15} and the estimate \eqref{eq2.22}, we only need to control the $\partial_t$ type terms. By \eqref{eq2.26} and the formula of $\tilde{\R}_{x_0}$, we have
 \begin{equation*}
     \begin{aligned}
         &\left|(\tilde{\R}_{x_0}(\tau,y)-\tilde{\R}_{x_0}(t,y))|_{x_0=x}\right|\\
	    &\quad\lesssim |z_{x_0,x}-y|^{\kappa}(t-\tau)^{\frac{\varkappa}{2}}\|\B\|_{\dot W^{\kappa,\infty}} \tau^{-\frac{1+\varkappa}{2}}\sup_{0<t\leq T}\left(\|\nabla u(t)\|_{L^\infty(\mathbb{R}^d_+)}+t^{\frac{1}{2}}\|(\partial_t,\nabla^2)u(t)\|_{L^\infty(\mathbb{R}^d_+)}\right),
     \end{aligned}
 \end{equation*}
 for some $\varkappa\in(0,\kappa)$. So we deduce that
 \begin{equation}\label{eq2.23}
     \begin{aligned}
         &\left|\int_0^t\int_{\mathbb{R}_+^d}\delta_\beta^z\partial_t\sH(t-\tau,(z_{x_0,x})'-y',(z_{x_0,x})_d,y_d)\left(\tilde{\R}_{x_0}(\tau,y)\right)|_{x_0=x}dyd\tau \right|\\
         &\lesssim \left|\int_0^t\int_{\mathbb{R}_+^d}\delta_\beta^z\partial_t\sH(t-\tau,(z_{x_0,x})'-y',(z_{x_0,x})_d,y_d)\left(\tilde{\R}_{x_0}(\tau,y)-\tilde{\R}_{x_0}(t,y)\right)|_{x_0=x}dyd\tau \right|\\
         &\qquad\qquad+\left|\left(\tilde{\R}_{x_0}(t,z_{x_0,x}-\beta)\right)|_{x_0=x}\right|+\left|\int_{\mathbb{R}_+^d}\delta_\beta^z \sH(t,(z_{x_0,x})'-y',(z_{x_0,x})_d,y_d)\tilde{\R}_{x_0}(t,y)dy\right|\\
         &\lesssim |\beta|^{\kappa}\|\B\|_{\dot C^\kappa}\left(\int_0^t(t-\tau)^{-1+\frac{\varkappa}{2}}\tau^{-\frac{1+\varkappa}{2}}d\tau+t^{-\frac{1}{2}}\right)\sup_{0<t\leq T}\left(\|\nabla u(t)\|_{L^\infty(\mathbb{R}^d_+)}+t^{\frac{1+\kappa}{2}}\|(\partial_t,\nabla^2)u(t)\|_{\dot C^{\kappa}(\mathbb{R}^d_+)}\right)\\
         &\lesssim t^{-\frac{1}{2}}|\beta|^{\kappa}\|\B\|_{\dot C^\kappa}\sup_{0<t\leq T}\left(\|\nabla u(t)\|_{L^\infty(\mathbb{R}^d_+)}+t^{\frac{1+\kappa}{2}}\|(\partial_t,\nabla^2)u(t)\|_{\dot C^{\kappa}(\mathbb{R}^d_+)}\right),
     \end{aligned}
 \end{equation}
 where we use \eqref{eq2.26} in the second inequality for some $\varkappa\in(0,\kappa)$. \eqref{eq2.22} and \eqref{eq2.23} imply that 
 \begin{equation}\label{vrhha}
     \begin{aligned}
      & \sup_{0<t\leq T}\left(|(\nabla v_{R,x_0}(t,z_{x_0,x}))|_{\substack{x_0=x}}|+t^\frac{1+\kappa}{2}\sup_{\alpha}\frac{|( \delta_\beta^z\nabla^2 v_{R,x_0}(t,z_{x_0,x}))|_{\substack{x_0=x}}|}{|\alpha|^\kappa}\right)\\
      &\quad\quad\quad\quad\quad\lesssim T^{\frac{\eta}{2}}(\|\B\|_{\dot C^\kappa}+\|\B\|_{\dot W^{1,\infty}})\sup_{0<t\leq T}\left(\|\nabla u(t)\|_{L^\infty(\mathbb{R}^d_+)}+t^{\frac{1+\kappa}{2}}\|(\partial_t,\nabla^2)u(t)\|_{\dot C^{\kappa}(\mathbb{R}^d_+)}\right).
 \end{aligned}
 \end{equation}
 Finally, combining \eqref{spliv}, \eqref{lllha}, \eqref{unnha}, \eqref{ununun2}, and \eqref{vrhha} with \eqref{mcbdfor}, \eqref{eq2.86} yields \eqref{re5.1}, completing the proof of the lemma.
		\end{proof}
\begin{proof}[Proof of Lemma~\ref{mcbdlemti}]
		Before the proof, we define the following notation. For $g:[0,T]\times\mathbb{R}^d_+\to \mathbb{R}$ and $\eps\in(0,1)$, define the mollification of $g$ on the half-space by
    \begin{equation*}
        g_\eps(t,x)=\int_{\mathbb{R}_+^d}g(t,y)\tilde\rho_\eps(x'-y',x_d,y_d)dy,
    \end{equation*}
    with 
    \begin{equation*}
        \begin{aligned}
        \tilde \rho_\eps(x'-y',x_d,y_d)=\rho_\eps(x'-y',x_d+y_d)+\rho_\eps(x'-y',x_d-y_d),
        \end{aligned}
        \end{equation*}
  where $\rho_\eps$ is the standard mollifier in $\mathbb{R}^d$.

First we prove \eqref{eq2.25}. For $0<\tau<t<T$, split $\nabla f(t)-\nabla f(\tau)$ into the following three terms.
\begin{equation*}
    \begin{aligned}
        |\nabla f(t,x)-\nabla f(\tau,x )|
        &\lesssim |(\nabla f)_\eps(t,x)-(\nabla f )_\eps(\tau,x)|\\
        &\quad+|\nabla f (t,x)-(\nabla f)_\eps(t,x)|
        +|\nabla f(\tau,x )-(\nabla f)_\eps(\tau,x )|\\
        &:=I_1(t,\tau,x)+I_2(t,x)+I_3(\tau,x),
    \end{aligned}
    \end{equation*}
 For $I_1$, using integration by parts we have
\begin{equation*}
    \begin{aligned}
        &I_1(t,\tau,x)\lesssim\left|\int_\tau ^t\partial_{\tau'}(\nabla f)_\eps(\tau',x)d\tau'\right|\\
        &\lesssim\int_\tau ^t\int_{\mathbb{R}_+^d}|\partial_{\tau'}f(\tau',y)||\nabla_{y}(\tilde \rho_\eps(x'-y',x_d,y_d))|dyd\tau'+\int_\tau^t\int_{\mathbb{R}^{d-1}}|\partial_{\tau'}f(\tau',y',0)||\rho_\eps(x'-y',x_d)|dy'd\tau'\\
        &\lesssim \eps^{-1}|t-\tau |\|\partial_t f\|_{L^\infty([\tau ,t], L^\infty)}\lesssim  \eps^{-1}|t-\tau | \tau ^{-\frac{1}{2}}\sup_{0<\tau'\leq T}\tau'^{\frac{1}{2}}\|\partial_{\tau'}f(\tau')\|_{L^\infty}.
    \end{aligned}
\end{equation*}
For $I_2$, we have
\begin{equation*}
    \begin{aligned}
       I_2(t,x)\lesssim& \left|\int_{\mathbb{R}_+^d}
       (\nabla f(t,x)-\nabla f(t,y))\tilde \rho_\eps(x'-y',x_d,y_d)dy \right|\\
       \lesssim& \eps t^{-\frac{1}{2}}
       \sup_{0<s\leq T}s^{\frac{1}{2}}\|\nabla^2f(s)\|_{L^\infty}.
    \end{aligned}
\end{equation*}
Similarly,
\[
I_3(\tau,x)\lesssim \eps \tau^{-\frac{1}{2}}\sup_{0<\tau'\leq T}\tau'^{\frac{1}{2}}\|\nabla^2f(\tau')\|_{L^\infty}.
\]
 Taking $\eps=|t-\tau |^{\frac{1}{2}}$ yields the following estimate
\begin{equation*}
\begin{aligned}
    \|\nabla f(t)-\nabla f(\tau )\|_{L^\infty}
    &\lesssim \|I_1(t,\tau)\|_{L^\infty}
    +\|I_2(t)\|_{L^\infty}+\|I_3(\tau)\|_{L^\infty}\\
    &\lesssim\tau ^{-\frac{1}{2}}(t-\tau )^{\frac{1}{2}}
    \sup_{0<\tau'\leq T}\tau'^{\frac{1}{2}}
    \|(\partial_{\tau'},\nabla^2)f(\tau')\|_{L^\infty}.
\end{aligned}
\end{equation*}
Since obviously $|\nabla f(t)-\nabla f(\tau )|\lesssim\sup_{0<\tau'\leq T}\|\nabla f(\tau')\|_{L^\infty}$, by interpolation we have
\begin{equation*}
\begin{aligned}
    \|\nabla f(t)-\nabla f(\tau )\|_{L^\infty}
    &\lesssim \tau ^{-\frac{\kappa}{2}}(t-\tau )^{\frac{\kappa}{2}}\\
    &\quad\times\sup_{0<t\leq T}\left(\|\nabla f(t)\|_{L^\infty}
    +t^{\frac{1}{2}}\|(\partial_t,\nabla^2)f(t)\|_{L^\infty}\right),
    \qquad \kappa\in[0,1].
\end{aligned}
    \end{equation*}
    This completes the proof of \eqref{eq2.25}.

    For \eqref{eq2.26}, it suffices to consider $t-\tau<\tau/4$.
    Indeed, if $t-\tau\geq\tau/4$, the left-hand side of \eqref{eq2.26}
    is controlled by the following interpolation inequality together with
    $(t-\tau)^{-\varkappa/2}\lesssim\tau^{-\varkappa/2}$:
    \[
    \|\nabla^2f(\tau)\|_{L^\infty}\lesssim
    \|\nabla f(\tau)\|_{L^\infty}^{\frac{\kappa}{1+\kappa}}
    \|\nabla^2f(\tau)\|_{\dot C^\kappa}^{\frac{1}{1+\kappa}}
    \]
    Thus the required estimate holds in this regime. In the remaining regime,
    arguing as in the proof of \eqref{eq2.25}, we write
    \begin{equation*}
    \begin{aligned}
        |\nabla^2 f(t,x)-\nabla^2 f(\tau,x )|&\lesssim |(\nabla^2 f)_\eps(t,x)-(\nabla^2 f )_\eps(\tau,x)|+|\nabla^2 f (t,x)-(\nabla^2 f)_\eps(t,x)|\\
        &\quad\quad+|\nabla^2 f(\tau,x )-(\nabla^2 f)_\eps(\tau,x )|\\
        &:=II_1(t,\tau,x)+II_2(t,x)+II_3(\tau,x),
    \end{aligned}
    \end{equation*}
    For $II_2$, note that
    \begin{equation}\label{eq2.27}
        \begin{aligned}
       II_2(t,x)\lesssim& \left|\int_{\mathbb{R}_+^d} (\nabla^2 f(t,x)-\nabla^2 f(t,y))\tilde \rho_\eps(x'-y',x_d,y_d)dy \right|\lesssim \eps^{\kappa} t^{-\frac{1+\kappa}{2}}\sup_{0<\tau'\leq T}\tau'^{\frac{1+\kappa}{2}}\|\nabla^2f(\tau')\|_{\dot C^{\kappa}},
    \end{aligned}
    \end{equation}
    and similarly 
    \begin{equation}\label{eq2.28}
        II_3\lesssim\eps^{\kappa} \tau^{-\frac{1+\kappa}{2}}\sup_{0<\tau'\leq T}\tau'^{\frac{1+\kappa}{2}}\|\nabla^2f(\tau')\|_{\dot C^{\kappa}}.
    \end{equation}
	    The term $II_1$ is more delicate. Split $\nabla^2$ into the three cases $\nabla_{x'}^2$, $\nabla_{x'}\partial_d$, and $\partial_d^2$, denoted by $II_{1,x'x'}$, $II_{1,x'd}$, and $II_{1,dd}$, respectively. For $II_{1,x'x'}$,
    \begin{equation}\label{eq2.29}
    \begin{aligned}
        &II_{1,x'x'}(t,\tau,x)\lesssim\left|\int_\tau ^t\partial_{\tau'}(\nabla_{x'}^2 f)_\eps(\tau',x)d\tau'\right|\\
        &\lesssim\int_\tau ^t\int_{\mathbb{R}_+^d}|\partial_{\tau'}f(\tau',y)-\partial_{\tau'}f(\tau',x)||\nabla_{y'}^2(\tilde \rho_\eps(x'-y',x_d,y_d))|dyd\tau'\\
        &\lesssim \eps^{-2+\kappa}|t-\tau |\|\partial_t f\|_{L^\infty([\tau ,t], \dot C^{\kappa})}\lesssim  \eps^{-2+\kappa}|t-\tau | \tau ^{-\frac{1+\kappa}{2}}\sup_{0<\tau'\leq T}\tau'^{\frac{1+\kappa}{2}}\|\partial_{\tau'}f(\tau')\|_{\dot C^{\kappa}}.
    \end{aligned}
\end{equation}
For $II_{1,x'd}$, by \eqref{eq2.25} we have
\begin{equation}\label{eq2.30}
    \begin{aligned}
        &II_{1,x'd}(t,\tau,x)\lesssim\left|\int_\tau ^t\int_{\mathbb{R}_+^d}\partial_{\tau'}(\partial_d f)(\tau',y)\nabla_{y'}\tilde{\rho}_\eps(x'-y',x_d,y_d)dyd\tau'\right|\\
        &\lesssim \left|\int_{\mathbb{R}_+^d}\left(\partial_df(\tau,y)-\partial_df(t,y)\right)\nabla_{y'}\tilde{\rho}_\eps(x'-y',x_d,y_d)dy \right|\\
        &\lesssim  \eps^{-1}|t-\tau |^{\frac{1}{2}} \tau ^{-\frac{1}{2}}\sup_{0<t\leq T}\left(\|\nabla f(t)\|_{L^\infty(\mathbb{R}^d_+)}+t^{\frac{1}{2}}\|(\partial_t,\nabla^2)f(t)\|_{L^\infty(\mathbb{R}^d_+)}\right).
    \end{aligned}
\end{equation}
For $II_{1,dd}$, there will be boundary terms, 
\begin{equation}\label{eq2.31}
    \begin{aligned}
        &II_{1,dd}(t,\tau,x)\\
        &\lesssim\left|\int_\tau ^t\int_{\mathbb{R}_+^d}
        \partial_{\tau'}\partial_df(\tau',y)
        \partial_{y_d}\tilde{\rho}_\eps(x'-y',x_d,y_d)dyd\tau'\right|\\
        &\quad+\left|\int_\tau^t\int_{\mathbb{R}^{d-1}}
        \partial_{\tau'}\partial_df(\tau',y',0)
        \rho_\eps(x'-y',x_d)dy'd\tau'\right|\\
        &\lesssim\left|\int_\tau ^t\int_{\mathbb{R}_+^d}
        \partial_{\tau'}f(\tau',y)\partial_{y_d}^2
        \tilde{\rho}_\eps(x'-y',x_d,y_d)dyd\tau'\right|\\
        &\quad+\int_{\mathbb{R}^{d-1}}
        |\partial_df(\tau,y',0)-\partial_df(t,y',0)|
        |\rho_\eps(x'-y',x_d)|dy'\\
        &\lesssim \tau^{-\frac{1}{2}}
        \left(\eps^{-2}|t-\tau |+\eps^{-1}|t-\tau|^{\frac{1}{2}}\right)\\
        &\qquad\times\sup_{0<t\leq T}\left(\|\nabla f(t)\|_{L^\infty(\mathbb{R}^d_+)}
        +t^{\frac{1}{2}}\|(\partial_t,\nabla^2)f(t)\|_{L^\infty(\mathbb{R}^d_+)}\right),
    \end{aligned}
\end{equation}
where we used \eqref{eq2.25} and the fact that $\partial_{y_d}\tilde \rho_\eps|_{y_d=0}=0$. Set $h=t-\tau$ and, since $h<\tau/4$ implies $t\simeq\tau$, take
\begin{equation*}
    \eps=t^{\frac{\kappa}{2(1+\kappa)}}h^{\frac{1}{2(1+\kappa)}},
    \qquad \varkappa=\frac{\kappa}{2(1+\kappa)}.
\end{equation*}
The four powers appearing in \eqref{eq2.27}, \eqref{eq2.28}, \eqref{eq2.29}, \eqref{eq2.30}, and \eqref{eq2.31} then satisfy
\begin{equation*}
\eps^\kappa\tau^{-\frac{1+\kappa}{2}},\quad
\eps^{-2+\kappa}h\tau^{-\frac{1+\kappa}{2}},\quad
\eps^{-1}h^{\frac12}\tau^{-\frac12},\quad
\eps^{-2}h\tau^{-\frac12}
\ \lesssim\ h^{\frac{\varkappa}{2}}\tau^{-\frac{1+\varkappa}{2}},
\end{equation*}
where powers of $h/\tau\leq1/4$ are discarded. This proves \eqref{eq2.26}.
\end{proof}
\begin{remark}
    The estimate \eqref{eq2.26} may not be optimal (the optimal index should be $\varkappa=\kappa$). However, for the critical term $v_{N,x_0}$, we use only \eqref{eq2.25}. Estimate \eqref{eq2.26} is used only for the remainder term $v_{R,x_0}$ to gain an additional factor $(t-\tau)^{\varkappa/2}$, so that $(t-\tau)^{\varkappa/2}\|\delta_\alpha \nabla^2\sH(t-\tau)\|_{L^1}$ is integrable at the endpoint $\tau\to t$.
\end{remark}

The second toy model is the heat equation on the half-space with the following forcing terms.
\begin{equation}\label{mcbd2}
 		\begin{aligned}
 			&\partial_tu(t,x)-\Delta u(t,x)=F^1(t,x)+\partial_dF^2(t,x),\quad \text{in}\ [0,T]\times \mathbb{R}_+^d,\\
 			&u(0,x)=u_0(x),\quad \text{in}\ \mathbb{R}_+^d,\\
 			&u(t,x)=0,\quad \text{on}\ [0,T]\times\partial\mathbb{R}_+^d,
 		\end{aligned}
 	\end{equation}
  Compared with \eqref{mcbd1}, the parabolic operator is simpler, but the forcing need not vanish on the boundary and therefore cannot in general be extended oddly. The divergence form $\partial_dF^2$ leads instead to the anisotropic regularity in the following lemma.
 \begin{lemma}\label{mcbdlemby}
 	Let $\kappa\in(0,1)$. Consider $F^1,F^2\in C^\infty([0,T]\times \mathbb{R}_+^d)$ and $u_0\in W^{1,\infty}(\mathbb{R}_+^d)$. If $u$ is a solution to the system \eqref{mcbd2}, then
    \begin{equation}\label{mcbddlow}
		\sup_{0<t\leq T}\left(\|u(t)\|_{C^1}+t^{\frac{\kappa}{2}}\|\nabla u(t)\|_{\dot C^{\kappa}}+t^{\frac{1+\kappa}{2}}\|\nabla_{x'}\nabla u(t)\|_{\dot C^{\kappa}}\right)\lesssim \|u_0\|_{W^{1,\infty}}+(1+T^{1/2})\mathbf{B}_{0}(F^1,F^2)(T).
 	\end{equation}
    Moreover, for any $\gamma\geq0$, we have
\begin{equation}\label{mcbddpri}
    \begin{aligned}
        \sup_{0<t\leq T}\left(t^{\gamma+\frac{\kappa}{2}}\|\nabla u(t)\|_{\dot C^{\kappa}}+t^{\gamma+\frac{1+\kappa}{2}}\|\nabla_{x'}\nabla u(t)\|_{\dot C^{\kappa}} \right)\lesssim \sup_{0<t\leq T}t^{\gamma}\| u(t)\|_{C^1}+\mathbf{B}_{\gamma}(F^1,F^2)(T),
    \end{aligned}
\end{equation}
where 
\begin{equation*}
   \mathbf{B}_{\gamma}(F^1,F^2)(T)=\sup_{0<t\leq T}\left(t^{\gamma+\frac{1}{2}}\|F^1(t)\|_{L^\infty}+t^\gamma\|F^2(t)\|_{L^\infty}+t^{\gamma+\frac{\kappa}{2}}\|F^2(t)\|_{\dot C^\kappa}+t^{\gamma+\frac{1+\kappa}{2}}\|(F^1,\nabla_{x'}F^2)(t)\|_{\dot C^{\kappa}} \right).
\end{equation*}
    \end{lemma}
 \begin{proof}
	 	We first prove \eqref{mcbddlow}. With the Green function $\sH$ defined in \eqref{defH}, we rewrite \eqref{mcbd2} as the following integral equation:
 	\begin{equation}\label{mcbddlemfml}
 		\begin{aligned}
	 		u(t,x)&=\int_{\mathbb{R}_+^d}
	 		\sH(t-t_0,x'-y',x_d,y_d)u(t_0,y)dy\\
	 		&\quad+\int_{t_0}^t\int_{\mathbb{R}^d_+}
	 		\sH(t-\tau,x'-y',x_d,y_d)
	 		(F^1+\partial_dF^2)(\tau,y)dyd\tau\\
 			&:=u_{L}(t_0,t,x)+\sum_{k=1}^2u^k_{N}(t_0,t,x)\quad \forall t_0<t.
 		\end{aligned}
 	\end{equation}
 	By \eqref{eskerha}, we can prove that
 	\begin{equation}\label{ulhalf}
 		\begin{aligned}
 			&\| u_{L}(t_0,t)\|_{C^1(\mathbb{R}^d_+)}\lesssim \|u(t_{0})\|_{W^{1,\infty}(\mathbb{R}^d_+)},
 			\\
            &\| \delta_\alpha u_{L}(t_0,t)\|_{C^1(\mathbb{R}^d_+)}\lesssim (t-t_0)^{-\frac{\kappa}{2}}|\alpha|^{\kappa}\|u(t_{0})\|_{W^{1,\infty}(\mathbb{R}^d_+)},
 			\\
 			&\| \delta_\alpha\nabla_{x'}\nabla u_{L}(t_0,t)\|_{L^\infty(\mathbb{R}^d_+)}\lesssim \int_{\mathbb{R}^d_+} |\delta_\alpha\nabla\sH(t-t_0,x'-y',x_d,y_d)|dy\|\nabla_{x'} u(t_{0})\|_{L^\infty(\mathbb{R}^d_+)}\\
 			&\quad\quad\quad\quad\quad\quad\quad\quad\quad\quad\quad\quad\lesssim \min\left\{1,\frac{|\alpha|}{(t-t_0)^{\frac{1}{2}}}\right\}(t-t_0)^{-\frac{1}{2}}\|\nabla u(t_0)\|_{L^\infty(\mathbb{R}^d_+)}.
 		\end{aligned}
 	\end{equation}
 	For $u_{N}^i$ with $i=1,2$, it is easy to get the $L^\infty$ bound as 
    \begin{equation}\label{unlinf}
        \begin{aligned}
            \sup_{t_0\leq t\leq T}\|u_N^{i}(t_0,t)\|_{L^\infty}\lesssim (T-t_0)^{\frac{1}{2}}\sup_{0<t\leq T}\left(t^{\frac{1}{2}}\|F^1(t)\|_{L^\infty}+\|F^2(t)\|_{L^\infty} \right).
        \end{aligned}
    \end{equation}
    For $u_N^1$, using integration by parts to get
 	\begin{equation}\label{un1half}
 		\begin{aligned}
 			&\|\nabla u_{N}^1(t_0,t)\|_{L^\infty(\mathbb{R}^d_+)}\lesssim \int_{t_0}^t\int_{\mathbb{R}^d_+} |\nabla_x\sH(t-\tau,x'-y',x_d,y_d)||F^1(\tau,y)|dyd\tau\\
 			&\quad\quad\quad\quad\quad\quad\quad\quad\quad\quad\quad\quad\lesssim \sup_{t\in[t_0,T]}(t-t_0)^{\frac{1}{2}}\|F^1(t)\|_{L^\infty},\\
            &\|\delta_\alpha\nabla u_{N}^1(t_0,t)\|_{L^\infty(\mathbb{R}^d_+)}\lesssim \int_{t_0}^t\int_{\mathbb{R}^d_+} |\delta_\alpha\nabla_x\sH(t-\tau,x'-y',x_d,y_d)||F^1(\tau,y)|dyd\tau\\
 			&\quad\quad\quad\quad\quad\quad\quad\quad\quad\quad\quad\quad\lesssim |\alpha|^{\kappa}(t-t_0)^{-\frac{\kappa}{2}}\sup_{t\in[t_0,T]}(t-t_0)^{\frac{1}{2}}\|F^1(t)\|_{L^\infty}.
 		\end{aligned}
 	\end{equation}
	    For higher-order H\"older estimates, we use the same argument as in \eqref{dxpdx}:
    \begin{equation}\label{un1higher}
        \begin{aligned}
            &\| \delta_\alpha\nabla_{x'}\nabla u_{N}^1(t_0,t)\|_{L^\infty(\mathbb{R}^d_+)}\lesssim \int_{t_0}^t\int_{\mathbb{R}^d_+} |\delta_\alpha\nabla_x\nabla_{x'}|D_{x'}|^{-\varkappa}\sH(t-\tau,x'-y',x_d,y_d)|\\
            &\quad\quad\quad\times\left||D_{y'}|^{\varkappa}F^1(\tau,y)-|D_{y'}|^{\varkappa}F^1(\tau,x)\right|dyd\tau\\
            &\quad\lesssim \int_{t_0}^t\int_{\mathbb{R}^d_+} |\delta_\alpha\nabla_x\nabla_{x'}|D_{x'}|^{-\varkappa}\sH(t-\tau,x'-y',x_d,y_d)||x'-y'|^{\kappa-\varkappa}\||D_{y'}|^{\varkappa}F^1(\tau)\|_{\dot C^{\kappa-\varkappa}_{x'}}dyd\tau\\
 			&\quad\lesssim |\alpha|^{\kappa}(t-t_0)^{-\frac{1+\kappa}{2}}\sup_{t\in[t_0,T]}(t-t_0)^{\frac{1+\kappa}{2}}\|F^1(t)\|_{\dot C^\kappa},
        \end{aligned}
    \end{equation}
 	for some $\varkappa\in(0,\kappa)$, and we use \eqref{eskerha} for the last inequality.  For $u_N^2$, integrate by parts and use the boundary condition of $\sH$ and $\tilde{\sH}$, we have
    \begin{equation*}
        \begin{aligned}
            u_N^2&=\int_{t_0}^t\int_{\mathbb{R}_+^d}\partial_{x_d}\tilde{\sH}(t-\tau,x'-y',x_d,y_d)F^2(\tau,y)dyd\tau\\
            &=\int_{t_0}^t\int_{\mathbb{R}^d}\partial_{x_d}\K(t-\tau,x-y)F^{2,e}(\tau,y)dyd\tau,
        \end{aligned}
    \end{equation*}
    where we denote $F^{2,e}$ to be the even extension of $F^2$,
    \begin{align}\label{evenext}
 F^{2,e}(t,x)=\begin{cases}
        F^2(t,x',x_d),\ \ \ x\in\mathbb{R}^d_+,\\
        F^2(t,x',-x_d),\ \ \ x\in \mathbb{R}^d_-.
    \end{cases}
 \end{align} 
 Obviously, $\nabla_{x'}F^{2,e}$ is the even extension of $\nabla_{x'}F^2$. Furthermore, by the fundamental theorem of calculus, for any $f:\mathbb{R}_+^d\rightarrow\mathbb{R}$, the $\dot C^{\eta}$ norm is retained under even extension $f^{\mathrm{even}}$, say
    \begin{equation*}
        \|f^{\mathrm{even}}\|_{\dot C^{\eta}(\mathbb{R}^d)}\lesssim \|f\|_{\dot C^{\eta}(\mathbb{R}^d_+)},\quad \forall \eta\in(0,1).
    \end{equation*}
    So for $\varkappa\in(0,\kappa)$,we have
 	\begin{align}
	 		&\|\nabla u_{N}^2(t_0,t)\|_{L^\infty(\mathbb{R}^d_+)}\lesssim \int_{t_0}^t\int_{\mathbb{R}^d} |\nabla\partial_{x_d}\K(t-\tau,x-y)||F^{2,e}(\tau,y)-F^{2,e}(\tau,x)|dyd\tau\label{un2half}\\
 			&\quad\quad\quad\quad\quad\quad\quad\quad\quad\quad\quad\quad\lesssim \sup_{t\in[t_0,T]}(t-t_0)^{\frac{\kappa}{2}}\|F^2(t)\|_{\dot C^\kappa},\nonumber\\
            &\|\delta_\alpha^x\nabla u_{N}^2(t_0,t)\|_{L^\infty(\mathbb{R}^d)}\nonumber\\
            &\quad\lesssim \int_{t_0}^t\int_{\mathbb{R}^d} |\delta_\alpha^x\nabla\partial_{x_d}|D_{y_d}|^{-\varkappa}\K(t-\tau,x-y)|\left||D_{y_d}|^{\varkappa}F^{2,e}(\tau,y',y_d)-|D_{x_d}|^{\varkappa}F^{2,e}(\tau,y',x_d)\right|dyd\tau\nonumber\\
            &\quad\lesssim\int_{t_0}^t\int_{\mathbb{R}^d} |\delta_\alpha^x\nabla\partial_{x_d}|D_{y_d}|^{-\varkappa}\K(t-\tau,x-y)||x_d-y_d|^{\kappa-\varkappa}\||D_{x_d}|^{\varkappa}F^{2,e}(\tau)\|_{\dot C^{\kappa-\varkappa}_{x_d}}dyd\tau\nonumber\\
 			&\quad\lesssim |\alpha|^{\kappa}(t-t_0)^{-\frac{\kappa}{2}}\sup_{t\in[t_0,T]}(t-t_0)^{\frac{\kappa}{2}}\|F^2(t)\|_{\dot C^\kappa},\nonumber\\
            &\| \delta_\alpha^x\nabla_{x'}\nabla u_{N}^2(t_0,t)\|_{L^\infty(\mathbb{R}^d_+)}\lesssim \int_{t_0}^t\int_{\mathbb{R}^d} |\delta_\alpha^x\nabla\partial_{x_d}|D_d|^{-\varkappa}\K(t-\tau,x-y)|\nonumber\\
            &\quad\quad\quad\quad\quad\quad\quad\quad\quad\quad\quad\quad\quad\quad\quad\quad\times\left||D_{y_d}|^{\varkappa}\nabla_{x'}F^{2,e}(\tau,y',y_d)-|D_{x_d}|^{\varkappa}\nabla_{x'}F^{2,e}(\tau,y',x_d)\right|dyd\tau\nonumber\\
            &\quad\lesssim\int_{t_0}^t\int_{\mathbb{R}^d} |\delta_\alpha^x\nabla\partial_{x_d}|D_{y_d}|^{-\varkappa}\K(t-\tau,x-y)||x_d-y_d|^{\kappa-\varkappa}\||D_{x_d}|^{\varkappa}\nabla_{x'}F^{2,e}(\tau)\|_{\dot C^{\kappa-\varkappa}_{x_d}}dyd\tau\nonumber\\
 			&\quad\lesssim |\alpha|^{\kappa}(t-t_0)^{-\frac{1+\kappa}{2}}\sup_{t\in[t_0,T]}(t-t_0)^{\frac{1+\kappa}{2}}\|\nabla_{x'}F^2(t)\|_{\dot C^\kappa}\nonumber.
	 	\end{align}
	 Taking $t_0=0$ in \eqref{ulhalf}, \eqref{unlinf}, \eqref{un1half}, \eqref{un1higher}, and \eqref{un2half} gives \eqref{mcbddlow}.
 	
	 	Starting from any $\tau\in(0,t)$, \eqref{mcbddlow} implies, for every $\gamma\geq0$,
 	\begin{equation}\label{mcbddmlemh}
 		\begin{aligned}
 			&\sup_{t\in(\tau,T]}\left((t-\tau)^{\frac{\kappa}{2}}\|\nabla u(t)\|_{C^\kappa}+(t-\tau)^{\frac{1+\kappa}{2}}\|\nabla_{x'}\nabla u(t)\|_{C^\kappa}\right)\\
 			&\quad\lesssim \|\nabla u(\tau)\|_{L^\infty}+ \sup_{t\in[\tau,T]}(t-\tau)^{\frac{1}{2}}\|F^1(t)\|_{L^\infty}+ \sup_{t\in[\tau,T]}(t-\tau)^{\frac{1+\kappa}{2}}\|(F^1,\nabla_{x'}F^2)(t)\|_{C^\kappa}\\
            &\quad\quad\quad+ \sup_{t\in[\tau,T]}(t-\tau)^{\frac{\kappa}{2}}\|F^2(t)\|_{C^\kappa}\\
            &\quad\lesssim \tau^{-\gamma}\left(\tau^{\gamma}\|\nabla u(\tau)\|_{L^\infty}+ \sup_{0<t\leq T}t^{\gamma+\frac{1}{2}}\|F^1(t)\|_{L^\infty}+ \sup_{t\in[\tau,T]}t^{\gamma+\frac{1+\kappa}{2}}\|(F^1,\nabla_{x'}F^2)(t)\|_{C^\kappa}\right.\\
            &\left.\quad\quad\quad+ \sup_{t\in[\tau,T]}t^{\gamma+\frac{\kappa}{2}}\|F^2(t)\|_{C^\kappa}\right).
 		\end{aligned}
 	\end{equation}
 	Then for any $t>0$, take $\tau=\frac{t}{2}$, and multiply $t^{\gamma}$ on both sides of \eqref{mcbddmlemh}, and we obtain \eqref{mcbddpri}. 
    \end{proof}

\section{Well-posedness of the mean curvature flow on the half-space}\label{sec3}
	This section proves local and global well-posedness for graphical mean curvature flow on the half-space with homogeneous Dirichlet boundary data. Throughout this section, $\kappa\in(\frac34,1)$ is fixed. We establish the local and global well-posedness of the graphical mean curvature equation on the half-space. The proof relies on the a priori estimates derived in Section~\ref{sec2} and a perturbative argument around the variable-coefficient operator with coefficient $\A[\nabla\phi]$ and homogeneous Dirichlet boundary condition.
    
	    The mean curvature system on the half-space with homogeneous Dirichlet boundary condition is
	\begin{equation}\label{eqmcbdd}
		\begin{aligned}
			&\partial_t f=\A[\nabla f]:\nabla^2 f,\ \ \ \text{in}\ [0,T]\times \mathbb{R}_+^d,\\
			&f|_{t=0}=f_0,\ \ \ \text{in}\ \mathbb{R}_+^d,\\
			&f=0,\ \ \text{on}\ [0,T]\times\{x_d=0\},
		\end{aligned}
	\end{equation}
 where the matrix $\A$ is defined in \eqref{mcf}. Note that the operator $\A[\nabla f]:\nabla^2$ is uniformly elliptic; more precisely,
    \begin{equation*}
        \frac{|\xi|^2}{1+\|\nabla f\|_{L^\infty}^2}\leq \A[\nabla f]:(\xi\otimes\xi)\leq |\xi|^2.
    \end{equation*}
	
	\begin{remark}
	    For general quasilinear Dirichlet problems, rough initial data need not satisfy the higher-order compatibility conditions required for full boundary regularity at $t=0$. The special structure of \eqref{eqmcbdd} propagates additional flat-boundary identities: Lemma~\ref{lembdc} shows that any solution $f\in L^\infty(0,T;C^{2m+2}(\overline{\mathbb{R}}_+^d))$ satisfies $\Delta^l f|_{(0,T)\times \partial\mathbb{R}_+^d}=0$ for $0\leq l\leq m+1$. On a curved boundary, Lemma~\ref{mcbdlembdy} instead expresses the relevant normal derivatives in terms of lower-order derivatives after flattening. These identities provide the positive-time boundary regularity used below.
\end{remark}
\subsection{Functional spaces and auxiliary lemma}\label{sec3.1}
This subsection introduces the time-weighted solution spaces and establishes the auxiliary linear estimate used in the fixed-point argument.
For $f:[0,T]\times\mathbb{R}^d_+\to \mathbb{R}^N$, define 
\begin{align*}
&\|f\|_{X_T}:=\|\nabla f\|_{L^\infty((0,T)\times\mathbb{R}^d_+)}+\sup_{0<t\leq T}t^{\frac{1+\kappa}{2}}\|\left(\nabla^2f,\partial_tf\right)(t)\|_{\dot C^{\kappa}(\mathbb{R}^d_+)},\\
&\|f\|_{Y_T^m}:=\|\nabla f\|_{L^\infty((0,T)\times\mathbb{R}^d_+)}+\sum_{n=0}^m\sup_{0<t\leq T}t^{n+\frac{\kappa+1}{2}}\|\nabla^{2n}\left(\nabla^2f,\partial_tf\right)(t)\|_{\dot C^{\kappa}(\mathbb{R}^d_+)}.
\end{align*}
Let $\phi\in C_b^{\infty}(\overline{\mathbb{R}}^d_+)$ satisfy $\phi|_{\partial\mathbb{R}^d_+}=0$, and let $\sigma\in(0,1)$. We define
 \begin{equation*}
 \begin{aligned}
     \mathcal{X}_{T,\phi}^\sigma=\left\{f\in L^\infty((0,T);C_b^1(\overline{\mathbb{R}}^d_+))\cap C((0,T];C_b^2(\overline{\mathbb{R}}^d_+)):f|_{(0,T]\times\partial\mathbb{R}^d_+}=0,\ \|f-\phi\|_{X_T}\leq\sigma\right\},
     \end{aligned}
     \end{equation*}
     \begin{equation*}
         \begin{aligned}
     &\mathcal{Y}_{T,\phi}^{m,\sigma}=\left\{f\in L^\infty((0,T);C_b^{1}(\overline{\mathbb{R}}_+^d))\cap C\left((0,T];C_b^{2m+2}(\overline{\mathbb R}_+^d)\right),\right.\\ 
     &\quad\quad\left.f|_{(0,T]\times \partial\mathbb{R}_+^d}=0,\ 
     \Delta^k f|_{(0,T]\times\partial\mathbb{R}_+^d}=0,\ k=1,\ldots,m:\ 
     \|f-\phi\|_{Y_T^m}\leq\sigma\right\}.
     \end{aligned}
 \end{equation*}
The following is a key lemma to prove Theorem~\ref{eqmcls}, which is also a direct corollary of Lemma~\ref{lemheat}.
\begin{lemma}\label{lemmbd}
			   Let $F=F^1+F^2\in C^\infty((0,T]\times \mathbb{R}_+^d;\mathbb R^N)$ with $F^2|_{(0,T]\times \partial\mathbb{R}_+^d}=0$. Let $\A:\mathbb R^{N\times d}\to\mathbb R_{\rm sym}^{d\times d}$ be the graphical mean-curvature coefficient map in \eqref{mcf}, and let $u_0,\phi\in C_b^\infty(\overline{\mathbb{R}}_+^d;\mathbb R^N)$ have zero boundary trace. Assume that the coefficient field $\A[\nabla\phi]$ is uniformly elliptic and has the bounded derivatives used below. If $u:(0,T]\times\mathbb R_+^d\to\mathbb R^N$ solves
		    \begin{equation}\label{mcbd}
     \begin{aligned}
         &\partial_tu(t,x)-\A[\nabla\phi](t,x):\nabla^2u(t,x)=F(t,x)=(F^1+F^2)(t,x),\quad \text{in}\ [0,T]\times\mathbb{R}_+^d,\\
         &u(0,x)=u_0(x),\quad \text{in}\ \mathbb{R}_+^d,\\
     &u(t,x)=0,\quad \text{on}\ [0,T]\times\{x_d=0\}.
     \end{aligned}
 \end{equation}
 Then we have
 \begin{equation}\label{uxha}
 \begin{aligned}
    &\|u\|_{X_T}\lesssim \|u_0\|_{\dot{W}^{1,\infty}}
    +T^{\frac{\eta}{2}}(1+T)\|\nabla\phi\|_{C^2}
    \sup_{0<t\leq T}\left(\|\nabla u(t)\|_{L^\infty}
    +
    t^{\frac{1}{2}}\|(\partial_t,\nabla^2)u(t)\|_{L^\infty}\right)\\
    &\quad+\sup_{0<t\leq T}\left(t^{\frac{1}{2}}\|F^1(t)\|_{L^\infty}
    +t^{\frac{1+\kappa}{2}}\|F^1(t)\|_{\dot C^\kappa}
    +t^{\frac{1+\kappa}{2}}\|F^2(t)\|_{\dot C^\kappa}\right)\\
    &\quad+\sup_{0<\tau<t<T}\tau^{\frac{1+\kappa}{2}}
    \frac{\|F^1(t)-F^1(\tau)\|_{L^\infty}}{(t-\tau)^{\frac{\kappa}{2}}},
 \end{aligned}
 \end{equation}
 where $0<\eta<\frac{2}{3}(1-\kappa)$ is as in Lemma~\ref{lemheat}. Moreover, for every $\gamma\geq0$,
 \begin{equation}\label{uzha}
 \begin{aligned}
     &\sup_{0<t\leq T}t^{\gamma+\frac{1+\kappa}{2}}(\|\nabla^2u(t)\|_{\dot C^{\kappa}}+\|\partial_tu(t)\|_{\dot C^{\kappa}})\\
     &\lesssim \sup_{0<t\leq T}t^{\gamma}\|\nabla u(t)\|_{L^\infty}
     +T^{\frac{\eta}{2}}(1+T)\|\nabla \phi\|_{C^2}
     \sup_{0<t\leq T}t^{\gamma}\left(\|\nabla u(t)\|_{L^\infty}
     +t^{\frac{1}{2}}\|(\partial_t,\nabla^2)u(t)\|_{L^\infty}\right)\\
     &\quad+\sup_{0<t\leq T}\left(t^{\gamma+\frac{1}{2}}\|F^1(t)\|_{L^\infty}
     +t^{\gamma+\frac{1+\kappa}{2}}\|F^1(t)\|_{\dot C^\kappa}
     +t^{\gamma+\frac{1+\kappa}{2}}\|F^2(t)\|_{\dot C^{\kappa}}\right)\\
     &\quad+\sup_{0<\tau<t<T}\tau^{\gamma+\frac{1+\kappa}{2}}
     \frac{\|F^1(t)-F^1(\tau)\|_{L^\infty}}{(t-\tau)^{\frac{\kappa}{2}}}.
 \end{aligned}
 \end{equation}
\end{lemma}
		\begin{proof}
 Lemma~\ref{lemheat}, applied with $\B=\A[\nabla\phi]$, gives for every $T>0$
\begin{equation}\label{wha}
    \begin{aligned}
        &\sup_{0<t\leq T}\left(\|\nabla u(t)\|_{L^\infty(\mathbb{R}_+^d)}
        +t^{\frac{1+\kappa}{2}}\|\nabla^2u(t)\|_{\dot C^{\kappa}(\mathbb{R}_+^d)}\right)\\
        &\quad\lesssim \|u_0\|_{\dot W^{1,\infty}(\mathbb{R}_+^d)}
        +\sup_{0<\tau<t<T}\tau^{\frac{1+\kappa}{2}}
        \frac{\|F^1(t)-F^1(\tau)\|_{L^\infty(\mathbb{R}_+^d)}}{(t-\tau)^{\frac{\kappa}{2}}}\\
        &\quad+\sup_{0<t\leq T}\left(t^{\frac{1}{2}}\|F^1(t)\|_{L^\infty(\mathbb{R}^d_+)}+t^{\frac{1+\kappa}{2}}\|F^1(t)\|_{\dot C^{\kappa}(\mathbb{R}_+^d)}+t^{\frac{1+\kappa}{2}}\|F^2(t)\|_{\dot C^{\kappa}(\mathbb{R}_+^d)}\right)\\
        &\quad+T^\frac{\eta}{2}\|\nabla\phi\|_{C^2}\sup_{0<t\leq T}(\|\nabla u(t)\|_{L^\infty(\mathbb{R}_+^d)}+t^{\frac{1+\kappa}{2}}\|\nabla^2u(t)\|_{\dot C^{\kappa}(\mathbb{R}_+^d)}).
    \end{aligned}
\end{equation}
Furthermore, by \eqref{mcbd}, we can see that
\begin{equation}\label{ux2ha}
    \begin{aligned}
        &\sup_{0<t\leq T}t^{\frac{1+\kappa}{2}}\|\partial_tu(t)\|_{\dot C^{\kappa}}
        \lesssim \sup_{0<t\leq T}t^{\frac{1+\kappa}{2}}
        \left(\|F(t)\|_{\dot C^{\kappa}}+\|\A[\nabla \phi]:\nabla^2u(t)\|_{\dot C^{\kappa}}\right)\\
        &\quad\lesssim \sup_{0<t\leq T}t^{\frac{1+\kappa}{2}}\|F(t)\|_{\dot C^{\kappa}}
        + \sup_{0<t\leq T}t^{\frac{1+\kappa}{2}}\|\nabla^2u(t)\|_{\dot C^\kappa}\\
        &\qquad+T^{\frac{\eta}{2}}\|\nabla \phi\|_{\dot C^\kappa}
        \sup_{0<t\leq T}\left(\|\nabla u(t)\|_{L^\infty(\mathbb{R}_+^d)}
        +t^{\frac{1+\kappa}{2}}\|\nabla^2u(t)\|_{\dot C^{\kappa}(\mathbb{R}_+^d)}\right).
    \end{aligned}
\end{equation}
Then \eqref{uxha} follows from \eqref{wha} and \eqref{ux2ha}.
It remains to prove \eqref{uzha}. For any $\tau\in(0,T)$, taking $\tau$ as a new initial time and applying \eqref{uxha}, we obtain
\begin{equation*}
    \begin{aligned}
        &\sup_{t\in[\tau,T]}(\|\nabla u(t)\|_{L^\infty}+(t-\tau)^{\frac{1+\kappa}{2}}\|(\nabla^2,\partial_t)u(t)\|_{\dot C^{\kappa}})\\
        &\lesssim \|\nabla u(\tau)\|_{L^\infty}
        +T^{\frac{\eta}{2}}(1+T)\|\nabla\phi\|_{C^2}
        \sup_{0<t\leq T}\left(\|\nabla u(t)\|_{L^\infty(\mathbb{R}_+^d)}
        +t^{\frac{1+\kappa}{2}}\|\nabla^2u(t)\|_{\dot C^{\kappa}(\mathbb{R}_+^d)}\right)\\
        &\quad+\sup_{t\in[\tau,T]}\left((t-\tau)^{\frac{1}{2}}\|F^1(t)\|_{L^\infty}
        +(t-\tau)^{\frac{1+\kappa}{2}}\|F^1(t)\|_{\dot C^\kappa}\right)\\
        &\quad+\sup_{\tau<\tau_1<\tau_2<T}(\tau_1-\tau)^{\frac{1+\kappa}{2}}
        \frac{\|F^1(\tau_1)-F^1(\tau_2)\|_{L^\infty}}
        {(\tau_2-\tau_1)^{\frac{\kappa}{2}}}\\
        &\quad+\sup_{t\in[\tau,T]}(t-\tau)^{\frac{1+\kappa}{2}}\|F^2(t)\|_{\dot{C}^\kappa},
        \end{aligned}
\end{equation*}
which implies that
\begin{equation}\label{eq3.4}
    \begin{aligned}
       &\tau^{\gamma}\sup_{t\in[\tau,T]}(\|\nabla u(t)\|_{L^\infty}+(t-\tau)^{\frac{1+\kappa}{2}}\|(\nabla^2,\partial_t)u(t)\|_{\dot C^{\kappa}})\\
       &\quad\lesssim \sup_{0<t\leq T}t^{\gamma}\| \nabla u(t)\|_{L^\infty}+T^{\frac{\eta}{2}}(1+T)\|\nabla\phi\|_{C^2}\sup_{0<t\leq T}(t^{\gamma}\|\nabla u(t)\|_{L^\infty}+t^{\gamma+\frac{1+\kappa}{2}}\|(\nabla^2,\partial_t)u(t)\|_{\dot C^{\kappa}})\\
    &\quad\quad+\sup_{0<t\leq T}\big(t^{\gamma+\frac{1}{2}}\|F^1(t)\|_{L^\infty}+t^{\gamma+\frac{1+\kappa}{2}}\|F^1(t)\|_{ \dot{C}^\kappa}\big)+\sup_{0<\tau_1<\tau_2<T}\tau_1^{\gamma+\frac{1+\kappa}{2}}\frac{\|F^1(\tau_1)-F^1(\tau_2)\|_{L^\infty}}{(\tau_2-\tau_1)^{\frac{\kappa}{2}}}\\
        &\quad\quad+\sup_{0<t\leq T}t^{\gamma+\frac{1+\kappa}{2}}\|F^2(t)\|_{\dot{C}^\kappa},\quad\forall\tau\in(0,T).
    \end{aligned}
\end{equation}
Since 
\begin{equation*}
   \sup_{0<t\leq T}t^{\gamma+\frac{1+\kappa}{2}}(\|\nabla^2u(t)\|_{\dot C^{\kappa}}+\|\partial_tu(t)\|_{\dot C^{\kappa}})\lesssim \sup_{0<t\leq T}(\frac{t}{2})^{\gamma}\sup_{\tau\in[\frac{t}{2},T]}\tau^{\frac{1+\kappa}{2}}\|(\nabla^2,\partial_t)u(\tau)\|_{\dot C^{\kappa}} ,
\end{equation*}
we can take $\tau=\frac{t}{2}$ in \eqref{eq3.4} to get \eqref{uzha}, and we complete the proof of the lemma.
		\end{proof}
	\subsection{Proof of Theorem~\ref{eqmcls}}
    This subsection proves the half-space theorem by constructing a contraction around a smooth reference profile and then deriving the small-data global estimate.
Using the special structure of the nonlinear term in \eqref{eqmcbdd}, we obtain the two boundary-condition lemmas, Lemmas~\ref{bdcon} and~\ref{lembdc}; their proofs are deferred to the appendix. We are now ready to prove the main theorem of this subsection.

		\begin{proof}[Proof of Theorem~\ref{eqmcls}]
 We will construct the solution to the system \eqref{eqmcbdd} by a fixed point argument.
Let $g\in \mathcal{Y}^{m,\sigma}_{T,\phi}$, where this space is defined at the beginning of Section~\ref{sec3.1}; the parameters $\sigma,T$, and $\phi$ will be fixed below. Define $\mathcal{S}g=f$, where $f$ is a strong solution of
    \begin{equation*}
		\begin{aligned}
			&\partial_t f=\A[\nabla g]:\nabla^2 f,\ \ \ \text{in}\ [0,T]\times \mathbb{R}_+^d,\\
			&f|_{t=0}=f_0,\ \ \ \text{in}\ \mathbb{R}_+^d,\\
			&f=0,\ \ \text{on}\ [0,T]\times\partial \mathbb{R}_+^d.
		\end{aligned}
	\end{equation*}
For smooth $g$ satisfying the boundary identities in the definition of $\mathcal Y_{T,\phi}^{m,\sigma}$, its odd extension across $\{x_d=0\}$ has the required regularity through order $2m+1$. The resulting extension of the coefficient field $\A[\nabla g]$ has the parity needed to extend the linear Dirichlet problem to a uniformly parabolic whole-space system. Standard smooth linear theory, followed by restriction to $\mathbb R_+^d$, therefore defines $\mathcal Sg$; the whole-space Schauder estimate used for its positive-time bounds is Theorem~\ref{propb=0} (see also \cite{KHN2025}). For nonsmooth $g$ in the completed fixed-point space, $\mathcal Sg$ is defined by smooth approximation and the uniform estimates proved below.
    
    \begin{remark}\label{rmk3.3}
        Throughout the proof, we estimate the equation for each component $f^\ell$ of $f=(f^\ell)_{\ell=1}^N$ separately and suppress the component index for simplicity.
    \end{remark}
    First consider the equation for $f-\phi$, with $\phi|_{\partial\mathbb{R}_+^d}=0$ and $\|f_0-\phi\|_{\dot W^{1,\infty}}\leq \eps_0$. We write
\begin{equation*}
    \begin{aligned}
			&\partial_t (f-\phi)-\A[\nabla \phi]:\nabla^2 (f-\phi)=F^1+F^2,\ \ \ \text{in}\ [0,T]\times \mathbb{R}_+^d,\\
			&(f-\phi)|_{t=0}=f_0-\phi ,\ \ \ \text{in}\ \mathbb{R}_+^d,\\
			&f-\phi =0,\ \ \text{on}\ [0,T]\times\partial \mathbb{R}_+^d,
		\end{aligned}
\end{equation*}
where 
\begin{align*}
   &F^1=\A[\nabla \phi]:\nabla^2 \phi,\\
         &F^2=(\A[\nabla g]-\A[\nabla\phi]):\nabla^2f.
\end{align*}
By Lemma~\ref{lembdc} with $F=0$, we have
\begin{equation*}
    \Delta f|_{\partial\mathbb{R}_+^d}(t)=0,\quad\forall t\in(0,T].
\end{equation*}
Note that $f|_{\partial\mathbb{R}_+^d}=\Delta f|_{\partial\mathbb{R}_+^d}=g|_{\partial\mathbb{R}_+^d}=\phi|_{\partial\mathbb{R}_+^d}=0$, by Lemma~\ref{lem5.1} we have
\begin{equation*}
    F^2|_{\partial\mathbb{R}_+^d}=0.
\end{equation*}
Since $F^1$ is smooth and independent of $t$, we have
\begin{equation}\label{eq3.6}
    \begin{aligned}
        &\sup_{0<t\leq T}\|F^1(t)\|_{C^{2}}\lesssim \|\nabla\phi\|_{C^{3}}(1+\|\nabla\phi\|_{C^{3}})^{3},\\
          &\sup_{0<t\leq T}t^{\frac{1}{2}}\|F^2(t)\|_{L^\infty}\lesssim \|g-\phi\|_{X_T}(\|f-\phi\|_{X_T}+T^{\frac{1}{2}}\|\nabla\phi\|_{C^2}),\\
          &\sup_{0<t\leq T}t^{\frac{1+\kappa}{2}}\|F^2(t)\|_{\dot C^\kappa}\lesssim (\|g-\phi\|_{X_T}+T^{\frac{\kappa}{2}}(1+T)\|\nabla\phi\|_{C^2})(\|f-\phi\|_{X_T}+T^{\frac{1}{2}}\|\nabla\phi\|_{C^2}),
    \end{aligned}
\end{equation}
Hence, we deduce from \eqref{uxha} and \eqref{eq3.6} that 
\begin{equation}\label{mchsprxt}
    \begin{aligned}
        \|f-\phi\|_{X_T}\leq& C_1\|f_0-\phi\|_{\dot{W}^{1,\infty}}+C_1(\|g-\phi\|_{X_T}+T^\frac{\eta}{2}(1+T)\|\nabla\phi\|_{C^2})(\|f-\phi\|_{X_T}+T^{\frac{1}{2}}\|\nabla\phi\|_{C^2})\\
        &\quad+T^{\frac{1}{2}}(1+T)\|\nabla\phi\|_{C^2}(1+\|\nabla\phi\|_{C^2})^3
    \end{aligned}
\end{equation}
Since $g\in \mathcal{Y}_{T,\phi}^{m,\sigma}$, we take $\eps_0$, $\sigma$, $\phi$ and $T$ such that
\begin{equation}\label{eq3.17}
\sigma\leq \frac{1}{2^{10dm}C_1},\ \ \ \eps_0\leq\frac{\sigma}{2^{10dm}C_1C_0'} \ \ \ T\leq 
\left(\frac{\sigma}{2^{10dm}C_1(2+\|\nabla\phi\|_{C^{2m+2}})^2}\right)^{10},
\end{equation}
for some $C_0'$ to be fixed later. Then for $\|f_0-\phi\|_{\dot W^{1,\infty}}\leq \eps_0$, we obtain the following estimate
\begin{equation}\label{bdmcxts}
    \|f-\phi\|_{X_T}\leq \frac{1}{2^{5dm}C_0'}\sigma.
\end{equation}

Next we estimate the  higher-order norm $\|f-\phi\|_{Y_T^m}$. Denote $f^m=\Delta^mf$, then one can write the equation of $f^m$ as 
\begin{equation*}
    \begin{aligned}
			&\partial_t f^m-\A[\nabla g]:\nabla^2 f^m=F^m,\ \ \ \text{in}\ (0,T]\times \mathbb{R}_+^d,
		\end{aligned}
\end{equation*}
with 
\begin{equation*}
     F^m=\Delta^m(\A[\nabla g]:\nabla^2f)-\A[\nabla g]:\nabla^2\Delta^mf.
\end{equation*}
By induction, we can prove the boundary vanishing condition for $f^m$, say
\begin{equation}\label{eq3.12}
    f^m(t,x)=0,\ \ \text{on}\ (0,T]\times\partial \mathbb{R}_+^d.
\end{equation}
Precisely, obviously \eqref{eq3.12} holds for $m=0$. Assume \eqref{eq3.12} holds for any $m\leq M$. By Lemma~\ref{bdcon}, we have 
\begin{equation}\label{eq3.11}
    F^m|_{\partial\mathbb{R}_+^d}=0,\quad\forall  m\leq M
\end{equation}
Applying Lemma~\ref{lembdc}, we can get $\Delta^{M+1}f|_{\partial\mathbb{R}_+^d}=0$, which implies \eqref{eq3.12} holds for $m=M+1$ and the induction proof is complete. Furthermore, 
\begin{align*}
    \partial_t f^m-\A[\nabla \phi]:\nabla^2 f^m=F^m+\left(\A[\nabla g]-\A[\nabla\phi]\right):\nabla^2f^m:=\tilde{F}^m.
\end{align*}
By \eqref{eq3.11} and Lemma~\ref{lem5.1}, we deduce that 
\begin{equation*}
    \tilde{F}^m|_{\partial\mathbb{R}_+^d}=0.
\end{equation*}
Note that since $f^m(0)$ is not well-defined, for any $t>0$, we take $t_0=\frac{t}{2}$ as the new initial time. Apply \eqref{uzha} to $f^m$ with  $\gamma=m$, $F^1=0$, $F^2=\tilde{F}^m$, for any $t\leq T$, one has
 \begin{align*}
     &t^{m+\frac{1+\kappa}{2}}\|(\nabla^2,\partial_t)f^m(t)\|_{\dot C^{\kappa}}\lesssim(t-t_0)^{m+\frac{1+\kappa}{2}}\|(\nabla^2,\partial_t)f^m(t)\|_{\dot C^{\kappa}}\\
     &\quad\lesssim (t-t_0)^{m}\|\nabla f^m(t_0)\|_{L^\infty}+\sup_{\tau\in[t_0,t]}(\tau-t_0)^{m+\frac{1+\kappa}{2}}\|\tilde{F}^m(\tau)\|_{\dot{C}^\kappa}\\
     &\quad\quad+T^\frac{\eta}{2}(1+T)\|\nabla\phi\|_{C^{2}}\sup_{\tau\in[t_0,t]}\left((\tau-t_0)^{m}\|\nabla f^m(\tau)\|_{L^\infty}+(\tau-t_0)^{m+\frac{1+\kappa}{2}}\|(\nabla^2,\partial_t)f^m(\tau)\|_{\dot C^{\kappa}}\right)\\
     &\quad\lesssim\sup_{\tau\in(0,T]}\tau^{m}\| \nabla f^m(\tau)\|_{L^\infty}+\sup_{0<\tau\leq T}\tau^{m+\frac{1+\kappa}{2}}\|\tilde{F}^m(\tau)\|_{\dot{C}^\kappa}\\
     &\quad\quad+T^\frac{\eta}{2}(1+T)\|\nabla\phi\|_{C^{2}}\sup_{0<\tau\leq T}\left(\tau^{m}\|\nabla f^m(\tau)\|_{L^\infty}+\tau^{m+\frac{1+\kappa}{2}}\|(\nabla^2,\partial_t)f^m(\tau)\|_{\dot C^{\kappa}}\right).
 \end{align*}
Since $\Delta^kf|_{\partial\mathbb{R}_+^d}=0$ for any $k\leq m+1$ (see Lemmas~\ref{lem5.1}--\ref{lembdc}), the odd extension $f^{\mathrm{odd}}$ of $f$ across $\partial\mathbb R_+^d$ is $C^{2m+2}$ for every $t>0$, and the full-space Calder\'on--Zygmund estimate gives
\begin{equation}\label{eq3.15}
    \|\nabla^{2m+2}f\|_{\dot{C}^{\kappa}}\lesssim\|\nabla^{2m+2}f^{\mathrm{odd}}\|_{\dot{C}^{\kappa}}\lesssim\|\nabla^2\Delta^mf^{\mathrm{odd}}\|_{\dot{C}^{\kappa}}\lesssim\|\nabla^2\Delta^mf\|_{\dot{C}^{\kappa}},
\end{equation}
and by interpolation,
\begin{equation}\label{eq3.14}
    \|u\|_{\dot{C}^{\gamma_0}}\lesssim\|u\|_{\dot{C}^{\gamma_1}}^{\frac{\gamma_2-\gamma_0}{\gamma_2-\gamma_1}}\|u\|_{\dot{C}^{\gamma_2}}^{\frac{\gamma_0-\gamma_1}{\gamma_2-\gamma_1}},\quad 0<\gamma_1\leq \gamma_0\leq \gamma_2,\quad \gamma_1<\gamma_2.
\end{equation}
We use the following fundamental lemma; its proof can be found in \cite{KHN2025}.
		\begin{lemma}\label{lemcom}
				Let $m,M,N\in\mathbb{N}$, $\alpha\in(0,1)$, and $\Lambda>0$. Consider $g,g_1,g_2,g_3,g_4:\mathbb{R}^d\to\mathbb{R}^N$ and $f:\mathbb{R}^N\to \mathbb{R}^M$ satisfying
			\begin{align*}
				\sum_{k=1}^{m+2}\|f^{(k)}\|_{L^\infty}\leq\Lambda.
			\end{align*}
			Then the implicit constants below may depend on $\Lambda$.
			\begin{equation}\label{mainint11}
				\|\nabla^m(f\circ g)\|_{L^\infty}\lesssim  \|g\|_{\dot C^1}^m+\| g\|_{\dot C^m},
			\end{equation}
			\begin{equation}\label{mainint12}
				\|\nabla^m(f\circ g)\|_{\dot C^\alpha}\lesssim  \| g\|_{\dot C^\alpha}^\frac{m+\alpha}{\alpha}+\|g\|_{\dot C^{m+\alpha}},
			\end{equation}
			\begin{equation}\label{mainint21}
				\|\nabla^m(f\circ g_1-f\circ g_2)\|_{L^\infty}\lesssim\sum_{n=0}^m \|g_1-g_2\|_{\dot C^n}(\|(g_1,g_2)\|_{\dot C^1}^{m-n}+\|(g_1,g_2)\|_{\dot C^{m-n}}),
			\end{equation}
			\begin{equation}\label{mainint22}
				\begin{aligned}
					\|\nabla^m(f\circ g_1-f\circ g_2)\|_{\dot C^\alpha}\lesssim &
					\sum_{n=0}^m\left\{\|g_1-g_2\|_{\dot C^{n+\alpha}}(\|(g_1,g_2)\|_{\dot C^1}^{m-n}+\|(g_1,g_2)\|_{\dot C^{m-n}})\right.\\
					&\left.\quad\quad+\|g_1-g_2\|_{\dot C^n}(\|(g_1,g_2)\|_{\dot C^\alpha}^\frac{m-n+\alpha}{\alpha}+\|(g_1,g_2)\|_{\dot C^{m-n+\alpha}})\right\}.
				\end{aligned}
			\end{equation}
			\begin{equation*}
				\begin{aligned}
					&\left\|\nabla^m\big((f\circ g_1-f\circ g_2)-(f\circ g_3-f\circ g_4)\big)\right\|_{L^\infty}\\
					&\lesssim \sum_{n=0}^m\|(g_1-g_2)-(g_3-g_4)\|_{\dot C^n}(\|(g_1,g_2)\|_{\dot C^1}^{m-n}+\|(g_1,g_2)\|_{\dot C^{m-n}})\\
					&\quad\quad\quad+\sum_{m_1+m_2+m_3=m}\|g_3-g_4\|_{\dot C^{m_1}}\|(g_1-g_3,g_2-g_4)\|_{\dot C^{m_2}}\left(\sum_{k=1}^4(\|g_k\|_{\dot C^1}^{m_3}+\|g_k\|_{\dot C^{m_3}})\right).
				\end{aligned}
			\end{equation*}
			
		\end{lemma}
The lemma also applies to maps between finite-dimensional matrix spaces after identifying those spaces with Euclidean spaces; $f^{(k)}$ denotes the $k$th Fr\'echet derivative and all norms are taken componentwise. In particular, it applies to $p\mapsto\A[p]$.
Notice that by \eqref{mainint11}, we have
\begin{equation}\label{mcbdhsesf}
    \begin{aligned}
    &\sup_{0<t\leq T}t^{m+\frac{1}{2}}\|\tilde{F}^m(t)\|_{L^\infty}\\
    &\quad\lesssim \sup_{0<t\leq T}t^{m+\frac{1}{2}}\sum_{\ell=0}^{2m-1}\|\nabla^2f(t)\|_{\dot C^{\ell}}\|\A[\nabla g](t)\|_{\dot C^{2m-\ell}}+\sup_{0<t\leq T}t^{m+\frac{1}{2}}\|\A[\nabla g]-\A[\nabla\phi]\|_{L^\infty}\|\nabla^2f^m(t)\|_{L^\infty}\\
    &\quad \lesssim \sup_{0<t\leq T}t^{m+\frac{1}{2}}\left(\sum_{\ell=0}^{2m-1}\|\nabla^2f(t)\|_{\dot C^{\ell}}\left(\|\nabla g(t)\|_{\dot C^1}^{2m-\ell}+\|\nabla g(t)\|_{\dot C^{2m-\ell}} \right)+\|\nabla(g-\phi)\|_{L^\infty}\|\nabla^2f(t)\|_{\dot C^{2m}}\right) \\
    &\quad\lesssim M_{g,\phi,T}^{2m+3}(\|g-\phi\|_{Y_T^m}+T^{\frac{\kappa}{2}}(1+T)^{m+2}\|\nabla\phi\|_{C^{2m+2}})(\|f-\phi\|_{Y_T^m}+T^{\frac{\kappa}{2}}(1+T)^{m+2}\|\nabla\phi\|_{C^{2m+2}}),
    \end{aligned}
\end{equation}
with
\begin{equation*}
    M_{g,\phi,T}=1+\|g-\phi\|_{Y_T^m}+T^{\frac{\kappa}{2}}(1+T)^m\|\nabla\phi\|_{C^{2m+2}}.
\end{equation*}
Furthermore, by \eqref{mainint12} we have
\begin{align}
    &\sup_{0<t\leq T}t^{m+\frac{1+\kappa}{2}}\|\tilde{F}^m(t)\|_{\dot C^\kappa}\label{eq3.22}\\
    &\quad\lesssim \sup_{0<t\leq T}t^{m+\frac{1+\kappa}{2}}\sum_{\ell=0}^{2m-1}\left(\|\nabla^2f(t)\|_{\dot C^{\ell}}\|\A[\nabla g](t)\|_{\dot C^{2m-\ell+\kappa}}+\|\nabla^2f(t)\|_{\dot C^{\ell+\kappa}}\|\A[\nabla g](t)\|_{\dot C^{2m-\ell}}\right)\nonumber\\
    &\quad\quad+\sup_{0<t\leq T}t^{m+\frac{1+\kappa}{2}}\left(\|\A[\nabla g]-\A[\nabla\phi]\|_{L^\infty}\|\nabla^2f^m(t)\|_{\dot C^\kappa}+\|\A[\nabla g]-\A[\nabla\phi]\|_{\dot C^\kappa}\|\nabla^2f^m(t)\|_{L^\infty}\right)\nonumber\\
    &\quad \lesssim \sup_{0<t\leq T}t^{m+\frac{1+\kappa}{2}}\sum_{\ell=0}^{2m-1}\|\nabla^2f(t)\|_{\dot C^{\ell+\kappa}}\left(\|\nabla g(t)\|_{\dot C^1}^{2m-\ell}+\|\nabla g(t)\|_{\dot C^{2m-\ell}} \right)\nonumber \\
    &\quad\quad+\sup_{0<t\leq T}t^{m+\frac{1+\kappa}{2}}\sum_{\ell=0}^{2m-1}\|\nabla^2f(t)\|_{\dot C^{\ell}}\left(\|\nabla g(t)\|_{\dot C^1}^{2m-\ell+\kappa}+\|\nabla g(t)\|_{\dot C^{2m-\ell+\kappa}}\right) \nonumber \\
    &\quad\quad+\sup_{0<t\leq T}t^{m+\frac{1+\kappa}{2}}\left(\|\nabla(g,\phi)(t)\|_{C^\kappa}\|\nabla^2f(t)\|_{\dot C^{2m}}+\|\nabla(g-\phi)(t)\|_{L^\infty}\|\nabla^2f(t)\|_{\dot C^\kappa}\right) \nonumber\\
    &\quad\lesssim M_{g,\phi,T}^{2m+3}(\|g-\phi\|_{Y_T^m}+T^{\frac{\kappa}{2}}(1+T)^{m+2}\|\nabla\phi\|_{C^{2m+2}})(\|f-\phi\|_{Y_T^m}+T^{\frac{\kappa}{2}}(1+T)^{m+2}\|\nabla\phi\|_{C^{2m+2}}),\nonumber
    \end{align}
So by \eqref{uzha}, \eqref{eq3.15}, \eqref{mcbdhsesf} and \eqref{eq3.22}, we can get
\begin{align}
&\sup_{0<t\leq T}t^{m+\frac{1+\kappa}{2}}\|(\nabla^2,\partial_t)\nabla^{2m}(f-\phi)(t)\|_{\dot C^\kappa}\leq C_2\sup_{0<t\leq T}t^{m}\| \nabla f^m(t)\|_{L^\infty}+ C_2 T^{\frac{\kappa}{2}}(1+T)^{m+2}\|\nabla\phi\|_{C^{2m+2}}\label{mchspryt}\\
        &\quad+C_2M_{g,\phi,T}^{2m+3}(\|g-\phi\|_{X_T}+T^{\frac{1}{2}}(1+T)^{m+3}\|\nabla\phi\|_{C^{2m+2}})(\|f-\phi\|_{Y_T^m}+T^{\frac{\kappa}{2}}(1+T)^{m+2}\|\nabla\phi\|_{C^{2m+2}}).\nonumber
\end{align}
By the interpolation inequality \eqref{eq3.14}, the first term on the right-hand side of \eqref{mchspryt} satisfies
\begin{equation*}
   \begin{aligned}
       \sup_{0<t\leq T}t^{m}\| \nabla f^m(t)\|_{L^\infty}\leq & \eta\sup_{0<t\leq T}t^{m+\frac{1+\kappa}{2}}\|(\nabla^2,\partial_t)\nabla^{2m}(f-\phi)(t)\|_{\dot C^\kappa}\\
       &+C_\eta\sup_{0<t\leq T}\|\nabla f(t)\|_{L^\infty}+T^{m}(1+T^2)\|\nabla\phi\|_{C^{m+2}},
   \end{aligned}
\end{equation*}
for any $\eta\in(0,1)$. Combining the estimates above with lower-order estimate \eqref{bdmcxts}, and take 
\begin{equation}\label{eq3.16}
    \sigma\leq \frac{1}{2^{10dm}C_1C_2},\ \ \ C_0'=100C_3, \ \ \ T\leq \left(\frac{ \sigma}{2^{10dm}C_1C_2(2+\|\nabla\phi\|_{C^{2m+2}})^{10dm}}\right)^{10},
\end{equation}
for $C_0'$ in \eqref{eq3.17}, to ensure
\begin{equation}\label{eq3.19}
    \begin{aligned}
        \sup_{0<t\leq T}t^{m+\frac{1+\kappa}{2}}\|(\nabla^2,\partial_t)\nabla^{2m}(f-\phi)(t)\|_{\dot C^{\kappa}}\leq  \frac{\sigma}{2^{3dm}}.
    \end{aligned}
\end{equation}
By \eqref{bdmcxts}, \eqref{eq3.16} and \eqref{eq3.19} (note that \eqref{eq3.16} implies \eqref{eq3.17}), we have
\begin{equation}\label{eq3.20}
    \|f-\phi\|_{Y_T^m}\leq \sigma.
\end{equation}
Hence
\begin{equation*}
    \mathcal{S}:\mathcal{Y}_{T,\phi}^{m,\sigma}\rightarrow\mathcal{Y}_{T,\phi}^{m,\sigma}.
\end{equation*}

It remains to prove that $\mathcal{S}$ is a contraction.
Let $f_1=\mathcal{S}g_1$, $f_2=\mathcal{S}g_2$, $g_1,g_2\in\mathcal{Y}_{T,\phi}^{m,\sigma}$, and set $\mathbf{f}=f_1-f_2$, $\mathbf{g}=g_1-g_2$. For brevity, write $f=(f_1,f_2)$ and $g=(g_1,g_2)$. The equation for $\mathbf{f}$ is
\begin{equation*}
    \begin{aligned}
        &\partial_t\mathbf f-\A[\nabla \phi]:\nabla^2\mathbf f=(\A[\nabla g_1]-\A[\nabla g_2]):\nabla^2f_2+(\A[\nabla g_1]-\A[\nabla\phi]):\nabla^2\mathbf{f}=:\mathbf{F}_1,\quad \text{in}\ [0,T]\times\mathbb{R}_+^d,\\
        &\mathbf{f}(0,x)=0,\quad \text{in}\ \mathbb{R}_+^d,\\
        &\mathbf f(t,x)=0,\quad\text{on}\ [0,T]\times\partial\mathbb{R}_+^d.
    \end{aligned}
\end{equation*}
Then apply Lemma~\ref{lemmbd} for $\mathbf{f}$ with $F^1=0$, $F^2=\mathbf{F}_1$. By the fact that $f_i|_{\partial\mathbb{R}_+^d}=\Delta f_i|_{\partial\mathbb{R}_+^d}=0$ for $i=1,2$, apply Lemma~\ref{lem5.1}, then we have
\begin{equation*}
    \mathbf{F}_1|_{\partial\mathbb{R}_+^d}=0.
\end{equation*}
By \eqref{uxha}, one has
\begin{equation}\label{mchsfpxt}
    \begin{aligned}
\|\mathbf{f}\|_{X_T}\leq C_4 (\|g-\phi\|_{X_T}+T^{\frac{\eta}{2}}(1+T)\|\nabla\phi\|_{C^2})\|\mathbf{f}\|_{X_T}+C_4\sup_{0<t\leq T}t^{\frac{1+\kappa}{2}}\|\mathbf{F}_1(t)\|_{C^\kappa}.
    \end{aligned}
\end{equation}
For $\mathbf{F}_1$, since $f,g\in\mathcal{Y}_{T,\phi}^{m,\sigma}$, one has
\begin{equation}\label{mchsfpf1}
    \begin{aligned}
        &\sup_{0<t\leq T}t^{\frac{1}{2}}\|\mathbf{F}_1(t)\|_{L^\infty}\leq C_5(1+\sigma+T^{\frac{\kappa}{2}}(1+T)^m\|\nabla\phi\|_{C^2})^6\|\mathbf{g}\|_{X_T}(\sigma+T^{\frac{1}{2}}\|\nabla\phi\|_{C^2}),\\
        &\sup_{0<t\leq T}t^{\frac{1+\kappa}{2}}\|\mathbf{F}_1(t)\|_{\dot C^\kappa}\leq C_5(1+\sigma+T^{\frac{\kappa}{2}}(1+T)\|\nabla\phi\|_{C^2})^6\|\mathbf{g}\|_{X_T}(\sigma+T^{\frac{1}{2}}\|\nabla\phi\|_{C^2}).
        \end{aligned}
\end{equation}
Combining \eqref{mchsfpxt} and \eqref{mchsfpf1}, we can take 
\begin{equation*}
    \sigma\leq\frac{1}{2^{10dm}C_4C_5},\quad T\leq (\frac{C_1'}{2^{10dm}C_4C_5(2+\|\nabla\phi\|_{C^2})^7})^{\frac{2}{\kappa}},
\end{equation*}
to get
\begin{equation}\label{eq3.24}
    \|\mathbf{f}\|_{X_T}\leq C_1'\|\mathbf{g}\|_{X_T},
\end{equation}
for some $C_1'<1$ to be decided later. 

For $\mathbf{f}^m=\Delta^m\mathbf{f}=\Delta^m(f_1-f_2)$, we can write the equation of $\mathbf{f}^m$ as
\begin{equation*}
    \begin{aligned}
	        &\partial_t\mathbf{f}^m-\A[\nabla \phi]:\nabla^2\mathbf{f}^m=\mathbf{F}_1^m,\quad \text{in }(0,T]\times\mathbb{R}_+^d,\\
    &\mathbf{f}^m(t,x)=0,\quad\text{on}\ (0,T]\times\partial\mathbb{R}_+^d,
    \end{aligned}
\end{equation*}
with 
\begin{equation*}
    \mathbf{F}_1^m=\Delta^m((\A[\nabla g_1]-\A[\nabla g_2]):\nabla^2f_2)+\Delta^m(\A[\nabla g_1]:\nabla^2\mathbf{f})-\A[\nabla g_1]:\nabla^2\Delta^m\mathbf{f}+(\A[\nabla g_1]-\A[\nabla\phi]):\nabla^2\mathbf{f}^m.
\end{equation*}
Note that by Lemma~\ref{bdcon}, for any $x\in\partial\mathbb{R}_+^d$,
\begin{equation*}
    \Delta^m\big((\A[\nabla g_1]-\A[\nabla g_2]):\nabla^2f_2\big)-(\A[\nabla g_1]-\A[\nabla g_2]):\nabla^2\Delta^mf_2=0,
\end{equation*}
\begin{equation*}
    \Delta^m(\A[\nabla g_1]:\nabla^2\mathbf{f})-\A[\nabla g_1]:\nabla^2\Delta^m\mathbf{f}=0.
\end{equation*}
It follows from Lemma~\ref{bdcon} and \ref{lembdc} that $\Delta^lf|_{\partial\mathbb{R}^{d}_+}=0$ for all $l\leq m+1$, by Lemma~\ref{lem5.1}, for any $x\in\partial\mathbb{R}_+^d$,
\begin{equation*}
    (\A[\nabla g_1]-\A[\nabla g_2]):\nabla^2\Delta^{m}f_2=(\A[\nabla g_1]-\A[\nabla\phi]):\nabla^2\mathbf{f}^m=0.
\end{equation*}
These two boundary-vanishing conditions imply
\begin{equation*}
    \mathbf{F}_1^m|_{\partial\mathbb{R}_+^d}=0.
\end{equation*}
Applying Lemma~\ref{lemmbd} to $\mathbf{f}^m$ with $F^1=0$ and $F^2=\mathbf{F}^m_1$, we obtain \begin{equation}\label{eq3.25}
    \begin{aligned}
        &\sup_{0<t\leq T}t^{m+\frac{1+\kappa}{2}}\|(\nabla^2,\partial_t)\mathbf{f}^m(t)\|_{\dot C^\kappa}\leq C_6\sup_{0<t\leq T}t^{m-\frac{1}{2}}\|\mathbf{f}^m(t)\|_{L^\infty}+C_6\sup_{0<t\leq T}t^{m+\frac{1+\kappa}{2}}\|\mathbf{F}^m_1(t)\|_{\dot C^\kappa}\\
        &\quad+C_6M_{g,\phi,T}^{2m+8}(\|g-\phi\|_{Y^m_T}+T^\frac{\kappa}{2}(1+T)^6\|\nabla\phi\|_{C^2})\\
        &\quad\quad\quad\quad\times(\sup_{0<t\leq T}t^{m}\|\nabla \mathbf f^m(t)\|_{L^\infty}+\sup_{0<t\leq T}t^{m+\frac{1+\kappa}{2}}\|(\partial_t,\nabla^2)\mathbf{f}^m(t)\|_{\dot C^{\kappa}}).
    \end{aligned}
\end{equation}
For $m\geq 1$, interpolation gives the following bound for the first term on the right-hand side:
\begin{equation}\label{mchsfpyi}
    C_6\sup_{0<t\leq T}t^{m-\frac{1}{2}}\|\mathbf{f}^m(t)\|_{L^\infty}\leq \frac{1}{100C_6}\sup_{0<t\leq T}t^{m+\frac{1+\kappa}{2}}\|\nabla^2\mathbf{f}^m(t)\|_{\dot C^\kappa}+C_7\sup_{0<t\leq T}\|\nabla\mathbf{f}(t)\|_{L^\infty}.
\end{equation}
For $\mathbf{F}_1^m$, by \eqref{mainint21} and \eqref{mainint22}, we have
\begin{equation}\label{mchsfpf2}
    \begin{aligned}
        &\sup_{0<t\leq T}t^{m+\frac{1+\kappa}{2}}\|\mathbf{F}_1^m(t)\|_{C^\kappa}\\
        &\quad\lesssim \sup_{0<t\leq T}t^{m+\frac{1+\kappa}{2}}\left(\sum_{\ell=0}^{2m}(\|\nabla\mathbf{g}(t)\|_{C^{\ell+\kappa}}\|\nabla^2f(t)\|_{\dot C^{2m-\ell}}+\|\nabla\mathbf{g}(t)\|_{\dot C^{\ell}}\|\nabla^2f(t)\|_{ C^{2m-\ell+\kappa}})\right.\\
        &\quad\quad+\sum_{\ell=1}^{2m}(\|\nabla g(t)\|_{C^{\ell+\kappa}}\|\nabla^2\mathbf{f}(t)\|_{\dot C^{2m-\ell}}+\|\nabla g(t)\|_{\dot C^{\ell}}\|\nabla^2\mathbf{f}(t)\|_{ C^{2m-\ell+\kappa}})\\
        &\quad\quad\left.+ (\|\nabla(g-\phi)(t)\|_{C^\kappa}\|\nabla^2\mathbf{f}^m(t)\|_{L^\infty}+\|\nabla(g-\phi)(t)\|_{L^\infty}\|\nabla^2\mathbf{f}^m(t)\|_{C^\kappa}) \right)\\
    &\quad\lesssim C_7 (1+\sigma+\|\nabla\phi\|_{C^{2m+2}})^{2m+3}\\
    &\quad\quad\times\Bigl[
    \|\mathbf{g}\|_{Y^m_T}
    \left(\|f-\phi\|_{Y^m_T}
    +T^{\frac{\kappa}{2}}(1+T)^{m+2}\|\nabla\phi\|_{C^{2m+2}}\right)\\
    &\qquad\qquad+
    \|\mathbf{f}\|_{Y_T^m}
    \left(\|g-\phi\|_{Y^m_T}
    +T^{\frac{\kappa}{2}}(1+T)^{m+2}\|\nabla\phi\|_{C^{2m+2}}\right)
    \Bigr].
    \end{aligned}
\end{equation}
Since $g\in\mathcal{Y}_{T,\phi}^{m,\sigma}$, combining \eqref{eq3.15}, \eqref{eq3.20}, \eqref{eq3.24}, \eqref{eq3.25}, \eqref{mchsfpyi} and \eqref{mchsfpf2}, we deduce that
\begin{equation}\label{eq3.29}
    \begin{aligned}
        &\sup_{0<t\leq T}t^{m+\frac{1+\kappa}{2}}\|(\nabla^2,\partial_t)\nabla^{2m}\mathbf{f}(t)\|_{\dot C^\kappa}\leq 2C_1'C_7\|\mathbf{g}\|_{Y^m_T}\\
        &\quad\quad+C_7M_{g,\phi,T}^{2m+8}
        (1+\sigma+\|\nabla\phi\|_{C^{2m+2}})^{2m+3}\\
        &\qquad\qquad\times
        \left(\sigma+T^{\frac{\kappa}{2}}(1+T)^{2m+3}
        \|\nabla\phi\|_{C^{2m+2}}\right)
        (\|\mathbf{g}\|_{Y_T^m}+\|\mathbf{f}\|_{Y_T^m}).
    \end{aligned}
\end{equation}
So if we take
\begin{align*}
    \sigma&<\frac{1}{2^{10dm}C_6C_7},\\
    T&<\min\left\{
    \left(\frac{ \sigma}
    {2^{10dm}C_6C_7(2+\|\nabla\phi\|_{C^{2m+2}})^{2m+9}}\right)^{\frac{2}{\kappa}},
    \left(\frac{ \sigma}
    {2^{10dm}C_1C_7(2+\|\nabla\phi\|_{C^{2m+2}})^{2m+9}}\right)^2
    \right\}^{10dm},
\end{align*}
and $C_1'$ satisfies $C_1'C_7\leq \frac{1}{100}$, then \eqref{eq3.24} and \eqref{eq3.29} give that
\begin{equation*}
    \|\mathbf{f}\|_{Y_T^m}\leq\frac{1}{2}\|\mathbf{g}\|_{Y_T^m},
\end{equation*}
The contraction mapping theorem therefore gives existence and uniqueness, proving part~(i) of Theorem~\ref{eqmcls}.

For part~(ii), take $\phi\equiv0$ and use the same estimates. Write $\mathcal{X}^\sigma=\mathcal{X}_{\infty,0}^\sigma$, $\|\cdot\|_{X}=\|\cdot\|_{X_\infty}$, $\mathcal{Y}^{m,\sigma}=\mathcal{Y}_{\infty,0}^{m,\sigma}$, and $\|\cdot\|_{Y^m}=\|\cdot\|_{Y^m_\infty}$ for brevity. Taking $\phi\equiv0$ in \eqref{mchsprxt} gives
\begin{equation}\label{mchsprxi}
    \|f\|_{X}\leq C_1\|f_0\|_{\dot{W}^{1,\infty}}+C_1(1+\|g\|_{X})\|g\|_{X}\|f\|_{X}.
\end{equation}
By taking $\sigma$ small enough such that $C_1(1+\sigma)\sigma\leq\frac{1}{100C_1}$, $\eps_0$ small enough such that $C_1\eps_0\leq\frac{\sigma}{2^{10dm}}$, one can prove 
\begin{equation*}
    \|f\|_{X}\leq \frac{1}{2^{5dm}}\sigma.
\end{equation*}
Furthermore, by taking $\phi\equiv 0$ in \eqref{mchspryt}, one can prove
\begin{equation}\label{mchspryi}
	    \sup_{t\in(0,T]}t^{m+\frac{1+\kappa}{2}}\|(\nabla^2,\partial_t)\nabla^{2m}f(t)\|_{\dot{C}^\kappa}\leq C_4(1+\|g\|_{Y^m})^{2m+3}\|g\|_{Y^m}\sup_{t\in(0,T]}\|\nabla f(t)\|_{L^\infty}.
\end{equation}
So by taking $\sigma$ small enough such that $C_1C_4\sigma(1+\sigma)^{2m+3}\leq\frac{1}{100}$, then \eqref{mchsprxi} and \eqref{mchspryi} give the following a priori estimate, 
\begin{equation*}
    \mathcal{S}:\mathcal{Y}^{m,\sigma}\rightarrow\mathcal{Y}^{m,\sigma}.
\end{equation*}
For the contraction argument, taking $\phi=0$ in \eqref{mchsfpxt} and \eqref{mchsfpf1} gives
\begin{equation*}
    \|\mathbf{f}\|_{X}\leq C_1(\|g_1\|_X+\|g_2\|_X)\|\mathbf{f}\|_X+C_1C_5(1+\sigma)^6\|\mathbf{g}\|_X\|f_2\|_X.
\end{equation*}
By taking $\sigma\leq \frac{1}{2^{10dm}C_1C_5}$, one has
\begin{equation}\label{mchsfpfx}
    \|\mathbf{f}\|_X\leq \frac{1}{2^{5dm}}\|\mathbf{g}\|_X.
\end{equation}
Finally by taking $\phi=0$ in \eqref{eq3.25} and \eqref{mchsfpf2}, one has 
\begin{equation*}
    \begin{aligned}
    &\sup_{t\in(0,T]}t^{m+\frac{1+\kappa}{2}}\|(\nabla^2,\partial_t)\nabla^{2m}\mathbf{f}(t)\|_{\dot {C}^\kappa}\leq C_5(1+\|g\|_{Y^m})^{2m+3}\\
    &\qquad\qquad\times\left(\|g\|_{Y^m}\sup_{t\in(0,T]}\|\nabla\mathbf{f}(t)\|_{L^\infty}+\|\mathbf{g}\|_{Y^m}(\|f_1\|_{Y^m}+\|f_2\|_{Y^m})+\|\mathbf{f}\|_{Y^m}(\|g_1\|_{Y^m}+\|g_2\|_{Y^m})\right).
    \end{aligned}
\end{equation*}
Apply the above estimate, together with \eqref{mchsfpfx}, and take $\sigma$ small enough, then we can get
\begin{equation*}
    \|\mathbf{f}\|_{Y^m}\leq\frac{1}{2}\|\mathbf{g}\|_{Y^m},
\end{equation*}
which completes the proof by the contraction mapping theorem.
\end{proof}

 \begin{remark}\label{bdydiff}
     Bounded domains differ from the half-space because their boundaries are curved. On the flat boundary, certain higher-order derivatives remain zero (see Lemma~\ref{lembdc}); this property does not persist unchanged after flattening a curved boundary. Only tangential derivatives of the Dirichlet trace vanish directly, so the higher-order boundary terms require a separate analysis.
 \end{remark}
 \begin{remark}
 The flat geometry makes the half-space global-regularity argument simpler. On a general bounded domain, localization produces lower-order remainders; these are handled together with exponential decay obtained from an energy estimate.
 \end{remark}
 \section{Well-posedness of the mean curvature flow on a bounded domain}\label{sec4}
This section transfers the half-space analysis to smooth bounded domains through localization, boundary flattening, and long-time decay estimates. Throughout this section, $\kappa\in(\frac34,1)$ is fixed. We prove local and global existence and uniqueness for the mean curvature system \eqref{meancur1} on a bounded domain $\Omega$. For smooth data, define $\mathcal Sg=f$, where $f$ solves the linear problem
 \begin{equation}\label{eqmcbd}
 		\begin{aligned}
 			&\partial_t f=\A[\nabla g]:\nabla^2 f,\ \ \ \text{in}\ [0,T]\times \Omega,\\
 			&f|_{t=0}=f_0,\ \ \ \text{in}\ \Omega,\\
 			&f=0,\ \ \text{on}\ [0,T]\times\partial \Omega.
 		\end{aligned}
 	\end{equation} 
	    As in Remark~\ref{rmk3.3}, we consider each component $f^\ell$ of $f=(f^\ell)_{\ell=1}^N$ and suppress the index. We decompose the domain into interior and boundary patches. In the interior, the equation is treated as a perturbation of the whole-space problem in Theorem~\ref{propb=0}. Near the boundary, we flatten the boundary and reduce the problem locally to a half-space system with additional lower-order forcing terms. Global-in-time control is obtained by combining local Schauder estimates with energy-based exponential decay. Section~\ref{sec4.1} treats localization; Sections~\ref{sec4.2} and~\ref{sec4.3} prove local and global well-posedness, respectively.
 \subsection{Renormalization of the equation}\label{sec4.1}
This subsection localizes the equation and renormalizes each boundary patch into a constant-coefficient half-space problem with controlled forcing terms. We first localize and renormalize \eqref{eqmcbd}. To regularize the initial data $f_0$, let $e^{t\Delta_\Omega}$ denote the Dirichlet heat semigroup on $\Omega$: $v=e^{t\Delta_\Omega}u$ solves
 \begin{align*}
 	&\partial_tv(t,x)-\Delta v(t,x)=0,\quad\text{in}\ [0,\infty)\times\Omega,\\
 	&v(0,x)=u(x),\quad\text{in}\ \Omega,\\
 	&v(t,x)=0,\quad\text{on}\ [0,\infty)\times\partial\Omega.
 \end{align*}
 We denote  \begin{align*}
 	\phi=e^{\eps_1\Delta_\Omega} f_0,
 \end{align*}	
where $\eps_1$ is small and will be fixed later. 

We now localize the system on a bounded domain $\Omega\subset\mathbb R^d$. Its boundary $\partial\Omega$ is a compact $C^{2m+3}$ hypersurface. We choose the following finite cover:

For any $r\in(0,1)$ and $x\in\partial\Omega$, let $B(x,r)$ be the $r$-neighborhood of $x$ in $\mathbb{R}^d$. Then $\{B(x,r)\}_{x\in\partial\Omega}$ covers $\partial\Omega$. We can take a finite Vitali-type subcover $\{B(x_i,r)\}_{i=1}^K$ such that:
\begin{itemize}
	    \item i) $\{B(x_i,r)\}_{i=1}^K$ covers $\partial\Omega$.
\item ii) There exists a constant $C_d$, depending only on the dimension $d$, such that for any $i\leq K$, at most $C_d$ indices $j\leq K$ satisfy $B(x_i,3r)\cap B(x_j,3r)\neq\varnothing$.
\end{itemize}
The radius $r$ will be fixed later in \eqref{eq4.3}. For fixed $r$, define $\Omega_i=B(x_i,r)$ and $\tilde{\Omega}_i=B(x_i,2r)$. There exists an interior open set $\Omega_0\Subset \Omega$ such that
\[
\Omega\subset\Omega_0\cup \bigcup_{i=1}^K{\Omega}_i.
\]
Then we define two families of smooth and compactly supported cut-off functions $\{\chi_i\}_{i=0}^K$ and $\{\tilde{\chi}_i\}_{i=0}^K$ which map $\Omega$ to $[0,1]$, such that
\[
\Omega_i\subset \chi_i^{-1}(1),\quad \operatorname{supp}(\chi_i)\subset\tilde{\chi}^{-1}_i(1),\quad\operatorname{supp}(\tilde{\chi}_i)\subset \tilde{\Omega}_i ,\quad\operatorname{supp}(\tilde{\chi}_0)\Subset\Omega.
\]
By this definition, there exist $c,C>0$, depending only on $d$, such that
\begin{equation*}
    c\leq \sum_{i}\chi_i\leq \sum_{i}\tilde{\chi}_i\leq C,\quad\forall x\in\Omega.
\end{equation*}
We choose $r$ and the cover sufficiently small so that, for every $i$,
\begin{equation}\label{eq4.3}
	    |\A[\nabla (\tilde{\chi}_i\phi)](x)-\A[\nabla(\tilde{\chi}_i\phi)](y)|\leq \eps,\quad \forall x,y\in\operatorname{supp}(\chi_i),
\end{equation} 
with $\eps$ to be fixed at the end of the a priori estimate. The patch radius is chosen sufficiently small, depending on $\eps$, $\|\phi\|_{C^2}$, and the boundary-chart norms, so that \eqref{eq4.3} holds.
Furthermore, by condition~(ii) of the cover $\{\Omega_i\}_{i=0}^K$, for any $x\in\Omega$, there exist at most $C_d$ sets $\tilde \Omega_i$ containing $x$, which implies
\begin{equation}\label{eq4.2}
    \|f\|_{L^\infty (\Omega)}\leq\sup_{0\leq i\leq K}\|\chi_if\|_{L^\infty},
\end{equation}
which allows the partition to be refined arbitrarily, even though $K$ may then be large.

Finally, for any $i\in\{1,\ldots,K\}$, we may assume without loss of generality, after a translation and rotation, that
\begin{equation*}
	    \partial\Omega\cap\tilde{\Omega}_i=\{(x',x_d):x_d=\varphi_i(x')\}\cap\tilde\Omega_i,
\end{equation*}
in a neighborhood of $0$, for some $C^{2m+3}$ function $\varphi_i$ with $\varphi_i(0)=0$. Let $U_i^+:=\Psi_i(\tilde\Omega_i\cap\Omega)\subset\mathbb R_+^d$. On these local coordinate neighborhoods, define the mutually inverse flattening maps
    \begin{equation*}
        \begin{aligned}
	            &\Phi_i(x)=(x',x_d+\varphi_i(x')):U_i^+\rightarrow \tilde{\Omega}_i\cap\Omega,\\
	            &\Psi_i(x)=(x',x_d-\varphi_i(x')):\tilde{\Omega}_i\cap\Omega\rightarrow U_i^+.
        \end{aligned}
    \end{equation*}
Then $\Phi_i$ maps $U_i^+\cap\partial\mathbb{R}^d_+$ to the corresponding portion of $\partial\Omega$. After extending each localized function by zero outside its coordinate patch, decompose the $\mathbb R^N$-valued function $f$ into
\begin{equation*}
	    f_{in}:=(\chi_0f):\mathbb{R}^d\rightarrow\mathbb{R}^N,\quad\{(\chi_if)\circ\Phi_i\}_{i=1}^K:\mathbb{R}_+^d\rightarrow\mathbb{R}^N.
\end{equation*}
We will renormalize the equation of $f_{in}$ and the boundary case separately. For $f_{in}$, it satisfies
\begin{equation}\label{eqinter}
 		\begin{aligned}
 			&\partial_tf_{in}-\A[\nabla \phi]:\nabla^2f_{in}=F_{in},\quad \text{in}\ [0,T]\times\mathbb{R}^d,\\
 			&f_{in}|_{t=0}=\chi_0f_0,\quad \text{in} \ \mathbb{R}^d,
 		\end{aligned}
 	\end{equation}
 	with
 	\begin{equation*}
 		F_{in} =-2\A[\nabla g]:(\nabla f\otimes\nabla\chi_0)-\A[\nabla g]:(f\nabla^2\chi_0)+(\A[\nabla g]-\A[\nabla\phi]):\nabla^2 f_{in}.
 	\end{equation*}
	For the boundary term, first multiply the equation for $f$ by $\chi_i$. Since the operator $\A[\nabla g]:\nabla^2$ is invariant under translations and rotations, we have
 	\begin{equation}\label{eqbd}
 		\begin{aligned}
 			&\partial_t(\chi_if)-\A[\nabla g_i]:\nabla^2(\chi_if)=G_i(f,g),\\
 			&\chi_if|_{t=0}=\chi_if_0,
 		\end{aligned}
 	\end{equation}
	 	with $g_i:=\tilde{\chi}_ig$ and lower-order forcing term
 	\begin{equation*}
 		\begin{aligned}
 		    G_i(f,g)=-2\A[\nabla g_i]:(\nabla\chi_i\otimes\nabla f)-\A[\nabla g_i]:(f\nabla^2\chi_i).
 		\end{aligned}
 	\end{equation*}
 After applying the flattening map, a matrix field $\alpha:U_i^+\to\mathbb R^{N\times d}$ gives the principal coefficient $\tilde{\A}[\alpha]$; for any scalar component $h$ of $f$,
 	\begin{equation}\label{mcbdnewc}
 		\begin{aligned}
			&\left(\A[\alpha]:\nabla^2h\right)\circ\Phi_i=\tilde{\A}[\alpha]:\nabla^2(h\circ\Phi_i)+R_i[h,\alpha],\\
 			&\tilde{\A}[\alpha]=(\nabla\Phi_i)^{-1}\A[\alpha](\nabla\Phi_i)^{-\top},\\
			&R_i[h,\alpha]=-\tilde\A[\alpha]:\left((\nabla h\circ\Phi_i)\cdot\nabla^2\Phi_i\right).
 		\end{aligned}
 	\end{equation}
	    Here $M^{-\top}=(M^{\top})^{-1}$. Set $\tilde g_i=(\tilde\chi_i g)\circ\Phi_i$ and
	    \[
	      \mathcal R_i[f,g]:=R_i\!\left[\chi_i f,\nabla\tilde g_i(\nabla\Phi_i)^{-1}\right].
	    \]
	    After flattening, $(\chi_if)\circ\Phi_i$ satisfies
    \begin{equation*}
        \begin{aligned}
	            &\partial_t\left((\chi_if)\circ\Phi_i\right)-\tilde{\A}[\nabla\tilde g_i(\nabla\Phi_i)^{-1}]:\nabla^2\left((\chi_if)\circ\Phi_i\right)=G_i(f,g)\circ\Phi_i+\mathcal R_i[f,g],\\
            &\left((\chi_if)\circ\Phi_i\right)|_{\partial\mathbb{R}_+^d}=0,
        \end{aligned}
    \end{equation*}
	 	where $\mathcal R_i[f,g]$ is the lower-order geometric remainder from \eqref{mcbdnewc}.
   This transforms \eqref{eqbd} to the half-space. We next freeze the coefficient at $0$. The matrix $\tilde\A$ is uniformly bounded and positive definite. Set $\tilde{\phi}_i:=(\tilde{\chi}_i\phi)\circ\Phi_i$ and $\tilde{\A}_0:=\tilde{\A}[\nabla\tilde{\phi}_i(\nabla\Phi_i)^{-1}](0)$. Choose an invertible linear map $\Theta_i:\mathbb R_+^d\to\mathbb R_+^d$ satisfying $\tilde\A_0=\Theta_i\Theta_i^\top$; then
    \begin{equation*}
        (\tilde{\A}_0:\nabla^2f)\circ\Theta_i=\Delta(f\circ\Theta_i).
    \end{equation*}
	Such a map is obtained by an RQ factorization of a square root of $\tilde{\A}_0$: multiplying an arbitrary square root $\tilde\Theta_i$ on the right by a suitable orthogonal matrix gives an upper-triangular factor $\Theta_i=\tilde\Theta_i\mathcal O$ with positive diagonal. Thus
	    \[
	    \tilde{\A}_0=\Theta_i\Theta_i^{\top},\qquad(\Theta_i)_{dj}=0\quad(j<d),\qquad(\Theta_i)_{dd}>0.
	    \]
    so $\Theta_i$ maps $\mathbb R^d_+$ onto itself. Uniform ellipticity gives uniform bounds for $\Theta_i$ and $\Theta_i^{-1}$.
	    We choose the patches and cutoffs as in \eqref{eq4.3}. For $\phi=e^{\eps_1\Delta_\Omega}f_0$, this gives
    \begin{equation}\label{eq4.6}
	        |\tilde{\A}[\nabla ((\tilde{\chi}_i\phi)\circ\Phi_i)(\nabla\Phi_i)^{-1}](x)-\tilde{\A}_0|\lesssim C_{\Omega}(1+\|\phi\|_{C^2}) \eps,\quad \forall x\in\operatorname{supp}({\chi}_i\circ\Phi_i).
    \end{equation}
Here we shrink the support of $\chi_i$ with respect to $\eps$ if necessary. Freezing the coefficient in this way replaces $\tilde{\A}$ by the constant matrix $\tilde{\A}_0$ in the principal part. The function $f_i(t,x):=\left((\chi_if)\circ\Phi_i\right)\circ\Theta_i$ then satisfies
 	\begin{equation}\label{eqbdref}
 		\begin{aligned}
 			&\partial_tf_i(t,x)-\Delta f_i(t,x)=F_i[f,g](t,x),\quad \text{in}\ [0,T]\times\mathbb{R}^d_+,\\
 			&f_i|_{t=0}=f_{i,0}=(\chi_i f_0)\circ\Phi_i\circ\Theta_i,\quad \text{in}\ \mathbb{R}^d_+,\\
 			&f_i(t,x)=0,\ \text{on}\ [0,T]\times\partial \mathbb{R}^d_+,
 		\end{aligned}
 	\end{equation}
 	where the forcing term has the form
 	\begin{equation}\label{deftilF}
 		\begin{aligned}
			&F_i[f,g]=\left(G_i[f,g]\circ\Phi_i+\mathcal R_i[f,g]\right)\circ\Theta_i\\
			&\qquad+\sum_{1\leq l,j\leq d}\left(\Theta_i^{-1}\left((\tilde{\A}[\nabla\tilde g_i(\nabla\Phi_i)^{-1}]-\tilde{\A}_0)\circ\Theta_i\right)\Theta_i^{-\top}\right)_{lj}\partial_{lj}f_i.
 		\end{aligned}
 	\end{equation}
    In the following proof, we will apply the following strategy:
    \begin{itemize}
	        \item Interior: no boundary is present, so Theorem~\ref{propb=0} applies directly.
	        \item Boundary: coefficient freezing reduces the localized system to \eqref{eqbdref}; we then apply the estimate for the toy model \eqref{mcbd2}. Lemma~\ref{mcbdlembdy} reorganizes the forcing so that no uncontrolled highest-order normal derivative occurs.
	        \item Global regularity: we combine the local result of Section~\ref{sec4.2} with the energy decay in \eqref{mcbdexpd} to prevent norm inflation.
    \end{itemize}
\subsection{Local well-posedness}\label{sec4.2}
This subsection localizes the equation and renormalizes each boundary patch into a constant-coefficient half-space problem with controlled forcing terms, and establishes the local-in-time a priori estimate in Theorem~\ref{mcbdglo} by combining interior Schauder estimates with boundary half-space estimates. We use $\|\cdot\|_{Z_{T,1}^m}$ on interior patches and the anisotropic norm $\|\cdot\|_{Z_{T,2}^m}$ on boundary patches.
 For $f:[0,T]\times\mathbb{R}^d\to  \mathbb{R}^N$,  define 
 \begin{align*}
 	\|f\|_{Z_{T,1}^m}=\sup_{t\in(0,T]}\left(\| f(t)\|_{C^1}+\sum_{1\leq 2l+k\leq 2m+2}t^{\frac{2l+k-1+\kappa}{2}}\|\partial_t^l\nabla^{k}f(t)\|_{C^{\kappa}}\right). 
 \end{align*}
 For $f:[0,T]\times\mathbb{R}^d_{+}\to  \mathbb{R}^N$,  define 
 \begin{equation}\label{defnormZ}
 	\begin{aligned}
	 	\|f\|_{Z^m_{T,2}}=\sup_{0<t\leq T}\left(\|f(t)\|_{C^1}+\sum_{\substack{1\leq n+k+2l\leq 2m+2\\
        k\leq 2m+1}}t^{\frac{n+k+2l-1+\kappa}{2}}\|\nabla^{n}_{x'}\partial_{d}^{k}\partial_t^lf(t)\|_{C^{\kappa}}\right).
 	\end{aligned}
 \end{equation} 
 The norm $\|f\|_{Z_{T,2}^m}$ is anisotropic: it allows less regularity in the normal direction than in the tangential directions.
 Let $f:[0,T]\times\Omega\to \mathbb{R}^N$. For some fixed $m$, we define 
 \begin{equation*}
 	\begin{aligned}
 		&\|f\|_{Z_T^m}:=\|f_{in}\|_{Z_{T,1}^m}+\sup_{1\leq i\leq K}\|f_i\|_{Z_{T,2}^m},
 	\end{aligned}
 \end{equation*}
where $f_{in}$ and $\{f_i\}_{i=1}^K$ are defined as in \eqref{eqinter} and \eqref{eqbdref}. We use a supremum over boundary patches because the overlap is uniformly finite, as in \eqref{eq4.2}. The uniform $C^{2m+3}$ bounds for the charts, the finite-overlap property, and the standard change-of-variables estimates for H\"older norms give
 \begin{equation}\label{holembed}
	\sup_{0<t\leq T}\left(\|f(t)\|_{C^{1}}+t^{\frac{2m+\kappa}{2}}\|f(t)\|_{C^{2m+1+\kappa}}\right)\leq C_{\Omega,m,\kappa}\|f\|_{Z_T^m}.
\end{equation}
 Let $\mathfrak Z_T^m$ be the completion, in the norm $Z_T^m$, of smooth functions on $(0,T]\times\overline\Omega$ with zero boundary trace. This makes the fixed-point space complete. For $\sigma>0$, define
 \begin{equation}\label{spacedef}
	 	\mathcal{Z}_{T,\phi}^{m,\sigma}=\{f\in\phi+\mathfrak Z_T^m:f|_{(0,T]\times\partial\Omega}=0,\ \|f-\phi\|_{Z_T^m}\leq \sigma\}.
 \end{equation}
 By \eqref{holembed}, every element of this space belongs to $L^\infty((0,T);C^1(\overline\Omega))\cap C((0,T];C^{2m+1+\kappa}(\overline\Omega))$ after choosing its canonical positive-time representative.
The anisotropic structure of $\|\cdot\|_{Z_{T,2}^m}$ reflects the loss of regularity in the normal direction induced by the Dirichlet boundary condition and is consistent with the half-space estimates in Section~\ref{sec2}. Furthermore, we define the following time-weighted norm,
    \begin{equation*}
 	\begin{aligned}
	 	\|f\|_{Z^{m,\gamma}_{T,2}}=\sup_{0<t\leq T}\left(t^{\gamma}\|f(t)\|_{C^1}+\sum_{n+2k+2l\leq 2m+1}t^{\gamma+\frac{n+2k+2l+\kappa}{2}}\|\nabla^{n}_{x'}\partial_{d}^{2k}\nabla\partial_t^lf(t)\|_{C^{\kappa}}\right).
 	\end{aligned}
 \end{equation*}
     We define the following lower-order auxiliary norm, which will be useful for the proof of global existence.
\begin{equation*}
    \begin{aligned}
      &\|f\|_{Z_{T,1}^{m,\mathrm{low}}}=\sup_{t\in(0,T]}\left(t^{\frac{\kappa}{2}}\| f(t)\|_{C^{1+\kappa}}+\sum_{1\leq 2l+k\leq 2m+2}t^{\frac{2l+k-1}{2}}\|\partial_t^l\nabla^{k}f(t)\|_{L^\infty}\right),\\
	       &\|f\|_{Z^{m,\mathrm{low}}_{T,2}}=\sup_{0<t\leq T}\left(t^{\frac{\kappa}{2}}\|f(t)\|_{C^{1+\kappa}}+\sum_{\substack{1\leq n+k+2l\leq 2m+2\\
        k\leq 2m+1}}t^{\frac{n+k+2l-1}{2}}\|\nabla^{n}_{x'}\partial_{d}^{k}\partial_t^lf(t)\|_{L^\infty}\right).
    \end{aligned}
\end{equation*}
Similarly, denote $\|f\|_{Z_T^{m,\mathrm{low}}}:=\|f_{in}\|_{Z_{T,1}^{m,\mathrm{low}}}+\sup_{1\leq i\leq K}\|f_i\|_{Z_{T,2}^{m,\mathrm{low}}}$. By interpolation, for any $\theta\in(0,1)$,
\begin{equation*}
    \begin{aligned}
        &\|f\|_{Z_{T,1}^{m,\mathrm{low}}}\leq \theta \|f\|_{Z_{T,1}^m}+C(\theta)\sup_{0<t\leq T}\|f(t)\|_{C^1},\\
        &\|f\|_{Z_{T,2}^{m,\mathrm{low}}}\leq \theta \|f\|_{Z_{T,2}^m}+C(\theta)\sup_{0<t\leq T}\|f(t)\|_{C^1}.
    \end{aligned}
\end{equation*}
 Analogous to Lemma~\ref{lembdc} in the case of half-space, we have the following lemma considering the behavior of solution at the boundary. As we mentioned in Remark~\ref{bdydiff}, unlike the half-space case, $\partial_d^2 f$ is not zero at boundary when considering general bounded domains. Thanks to Lemma~\ref{mcbdlembdy}, we can prove that higher-order normal derivatives $\partial_d^{2k}f$ can be decomposed into a sum of terms with lower-order normal derivatives on the boundary.

 The following result is a corollary of Lemma~\ref{mcbdlemby} and gives the a priori estimate for higher-order derivatives.
    \begin{lemma}\label{maincor}
        Assume $u$ is the solution to \eqref{mcbd2}, $\gamma\geq 0$, then we have
 	\begin{equation*}
 		\begin{aligned}
 			&\|u\|_{Z_{T,2}^{m,\gamma}}\lesssim   \sup_{0<t\leq T}t^{\gamma}\|u(t)\|_{C^1}+\mathbf{C}_{m}^\gamma(F^1,F^2)(T),
 		\end{aligned}
 	\end{equation*}
 	where we denote
    \begin{equation}\label{defCgam}
 	\begin{aligned}
	    \mathbf{C}_{n}^\gamma(F^1,F^2)(T)
	    &=\sup_{0<t\leq T}\sum_{2i+j\leq 2n}\Bigl(
	    t^{\gamma+i+\frac{1+j}{2}}\|\partial_t^i\nabla^{j}F^{1}(t)\|_{L^\infty}\\
	    &\quad+t^{\gamma+i+\frac{1+j+\kappa}{2}}
	    \|\partial_t^i\nabla^{j}(F^{1},\nabla_{x'}F^2)(t)\|_{C^\kappa}\\
    &\quad+t^{\gamma+i+\frac{j}{2}}\|\partial_t^i\nabla^{j}F^2(t)\|_{L^\infty}
    +t^{\gamma+i+\frac{j+\kappa}{2}}\|\partial_t^i\nabla^{j}F^2(t)\|_{C^\kappa}\Bigr).
 	\end{aligned}
    \end{equation}
    \end{lemma}
    \begin{proof}
 	Since the case $\gamma>0$ can be obtained by a procedure similar to \eqref{mcbddmlemh}, we concentrate on the case $\gamma=0$ and omit the superscript $\gamma$ in the proof.
    It suffices to consider spatial derivatives, since by \eqref{mcbd2}, the time derivatives can be controlled by spatial derivatives. Precisely, we will prove
 	\begin{equation*}
 		\begin{aligned}
		    &\sup_{0<t\leq T}\sum_{n+2k\leq 2m}\Bigl(
		    t^{\frac{n}{2}+k}\|\partial_{d}^{2k}\nabla_{x'}^{n}u(t)\|_{C^1}\\
		    &\qquad+t^{\frac{n}{2}+k+\frac{\kappa}{2}}
		    \|\nabla(\partial_{d}^{2k}\nabla_{x'}^{n}u)(t)\|_{\dot C^{\kappa}}\\
		    &\qquad+t^{\frac{n}{2}+k+\frac{1+\kappa}{2}}
		    \|\nabla_{x'}\nabla(\partial_{d}^{2k}\nabla_{x'}^{n}u)(t)\|_{C^\kappa}\Bigr)\\
        &\quad\quad\lesssim \sup_{0<t\leq T}\|u(t)\|_{C^1}+\mathbf{C}_{m}(F^1,F^2)(T).
 		\end{aligned}
 	\end{equation*}
	   First consider derivatives in the first $d-1$ coordinates. Integrating by parts in the tangential variables in \eqref{mcbddlemfml} produces no boundary terms and gives the linear estimate
   \begin{equation*}
       \|\nabla_{x'}^l\nabla u_L(t)\|_{C^\kappa}\lesssim t^{-\frac{l+\kappa}{2}}\|u_0\|_{C^1},
   \end{equation*}
   the forcing terms 
   \begin{align*}
       \|\nabla_{x'}^l\nabla u_{N}^1(t)\|_{C^\kappa}&\lesssim \int_0^{\frac{t}{2}}\int_{\mathbb{R}_+^d}\nabla_{x'}^l\nabla\sH(t-\tau,x'-y',x_d,y_d)F^1(\tau,y',y_d)dyd\tau\\
       &\quad+\int_{\frac{t}{2}}^t\int_{\mathbb{R}_+^d}\nabla\sH(t-\tau,x'-y',x_d,y_d)\nabla_{y'}^lF^1(\tau,y',y_d)dyd\tau\\
       &\lesssim t^{-\frac{l+\kappa}{2}}\mathbf{C}_m(F^1,0),
   \end{align*}
   and
   \begin{align*}
       \|\nabla_{x'}^l\nabla u_{N}^2(t)\|_{C^\kappa}&\lesssim \int_0^{\frac{t}{2}}\int_{\mathbb{R}^d}\nabla_{x'}^l\partial_{d}\nabla\K(t-\tau,x-y)F^{2,e}(\tau,y)dyd\tau\\
       &\quad+\int_{\frac{t}{2}}^t\int_{\mathbb{R}^d}\nabla\partial_d\K(t-\tau,x-y)\nabla_{y'}^lF^{2,e}(\tau,y)dyd\tau\\
       &\lesssim  t^{-\frac{l+\kappa}{2}}\mathbf{C}_m(0,F^{2}),
   \end{align*}
   where we denote $F^{2,e}$ to be the even extension of $F^2$ defined in \eqref{evenext}, and use the estimate in \eqref{un2half}. So we conclude that
    \begin{equation}\label{mcbddtanest}
        \sum_{l\leq 2m}\sup_{0<t\leq T}\left(t^{\frac{l+\kappa}{2}}\|\nabla_{x'}^{l}\nabla u(t)\|_{C^\kappa}+t^{\frac{l+1+\kappa}{2}}\|\nabla_{x'}^{l+1}\nabla u(t)\|_{C^\kappa} \right)\lesssim \|u_0\|_{W^{1,\infty}}+ \mathbf{C}_{m}(F^1,F^2)(T).
    \end{equation}
 	It remains to estimate $\partial_d^{2k}\nabla_{x'}^nu$ with $n+2k\leq 2m$ and $k\neq 0$. First we denote $u_{n}=\nabla_{x'}^nu$, then we can write the equation of $u_n$ as
 	\begin{equation}\label{equn}
 		\begin{aligned}
 			&\partial_tu_n(t,x)-\Delta u_n(t,x)=F_n^1(t,x)+\partial_dF_n^2(t,x),\quad \text{in}\ (0,T]\times\mathbb{R}^d_+,\\
 			&u_n(t,x)=0,\ \text{on}\ (0,T]\times\partial \mathbb{R}^d_+,
 		\end{aligned}
 	\end{equation}
 	with
 	\begin{equation*}
 		F_n^i=\nabla_{x'}^nF^i,\quad i=1,2.
 	\end{equation*}
 	By Lemma~\ref{mcbdlembdy},  we have the boundary condition for $\partial_d^{2k}u_n$ that
 	\begin{equation*}
 		\partial_{d}^{2k}u_n=F_{n}^{1,k-1}+\partial_dF_n^{2,k-1}+R^{2k}_n,\quad \text{on}\ \partial\mathbb{R}^d_+,
 	\end{equation*}
 	where 
 	\begin{align}
 		&F_{n}^{i,k-1}=\sum_{2j+|\alpha|=2k-2}C_{j,\alpha}^{1,i,n}\partial_t^j\nabla^\alpha F_n^i,\ \ i=1,2,\label{formFR}\\
        &R^{2k}_n=\sum_{\substack{l+|\beta|= 2k\\l<2k-1}}C_{l,\beta}^{2,n}\partial_{d}^{l}\nabla_{x'}^{\beta}u_n.\nonumber
 	\end{align}
 	Furthermore, denote
    \begin{equation*}
        u_{n,2k}=\partial_{d}^{2k}u_n-(F_{n}^{1,k-1}+\partial_dF_{n}^{2,k-1}+R^{2k}_n)
    \end{equation*}
    then we will get the equation of $u_{n,2k}$ with Dirichlet boundary condition as
 	\begin{equation}\label{un2k}
 		\begin{aligned}
 			&\partial_tu_{n,2k}(t,x)-\Delta u_{n,2k}(t,x)=\tilde{F}_{n,2k}^1(t,x)+\partial_d\tilde{F}_{n,2k}^2(t,x)+\tilde{R}_n^{2k},\quad \text{in}\ (0,T]\times\mathbb{R}^d_+,\\
 			&u_{n,2k}(t,x)=0,\ \text{on}\ (0,T]\times\partial \mathbb{R}^d_+,
 		\end{aligned}
 	\end{equation}
 	with
 	\begin{equation*}
 		\tilde{F}_{n,2k}^{i}=-(\partial_t-\Delta)F_n^{i,k-1}+\partial_d^{2k}F_n^i,\quad\tilde{R}^{2k}_n=-(\partial_t-\Delta)R_n^{2k}.
 	\end{equation*}
 	By the form of $R_n^{2k}$ and equation of $u_n$ in \eqref{equn}, we have 
 	\begin{equation*}\label{formR}
 		\begin{aligned}
 			\tilde{R}^{2k}_n&=  -(\partial_t-\Delta) R^{2k}_n=\sum_{\substack{l+|\beta|= 2k\\l<2k-1}}C_{l,\beta}^{2,n}\partial_{d}^{l}\nabla_{x'}^{\beta}(\partial_t-\Delta)u_n\\
 			&=\sum_{\substack{l+|\beta|= 2k\\l<2k-1}}C_{l,\beta}^{2,n}\partial_{d}^{l}\nabla_{x'}^{\beta}F^1_{n}+\partial_d\left(\sum_{\substack{l+|\beta|= 2k\\l<2k-1}}C_{l,\beta}^{2,n}\partial_{d}^{l}\nabla_{x'}^{\beta}F^2_{n}\right):=R^{1,2k}_n+\partial_dR_{n}^{2,2k}.
 		\end{aligned}
 	\end{equation*}
 	Applying the estimate \eqref{mcbddpri} to the equation \eqref{un2k} with $\gamma=\frac{n+2k}{2}$, we will get
 	\begin{equation}\label{mcbddhighest}
 		\begin{aligned}
 			&\sup_{0<t\leq T}\left(t^{\frac{n+2k+\kappa}{2}}\|\nabla u_{n,2k}(t)\|_{C^\kappa}+t^{\frac{n+2k+1+\kappa}{2}}\|\nabla_{x'}\nabla u_{n,2k}(t)\|_{C^\kappa} \right)\lesssim  \sup_{0<t\leq T}t^{\frac{n+2k}{2}}\|\nabla u_{n,2k}(t)\|_{L^\infty}\\
            &\quad\quad+\sup_{0<t\leq T}\left(t^{\frac{n+2k+1}{2}}\left(\|\tilde{F}_{n,2k}^{1}(t)\|_{L^\infty}+\|R_n^{1,2k}(t)\|_{L^\infty}\right)+t^{\frac{n+2k+1+\kappa}{2}}\left(\|\tilde{F}_{n,2k}^{1}(t)\|_{C^\kappa}+\|R_n^{1,2k}(t)\|_{C^\kappa}\right)\right)\\
            &\quad\quad+\sup_{0<t\leq T}\left(t^{\frac{n+2k}{2}}\left(\|\tilde{F}_{n,2k}^{2}(t)\|_{L^\infty}+\|R_n^{2,2k}(t)\|_{L^\infty}\right)+t^{\frac{n+2k+\kappa}{2}}\left(\|\tilde{F}_{n,2k}^{2}(t)\|_{C^\kappa}+\|R_n^{2,2k}(t)\|_{C^\kappa}\right)\right)\\
 			&\quad\quad\quad\quad+\sup_{t\in(0,T]}t^{\frac{n+2k+1+\kappa}{2}}\left(\|\nabla_{x'}\tilde{F}_{n,2k}^{2}(t)\|_{C^\kappa}+\|\nabla_{x'}R_n^{2,2k}(t)\|_{C^\kappa}\right).
 		\end{aligned}
 	\end{equation}
 	By the definition of $\tilde{F}_{n,2k}^{1},\tilde{F}_{n,2k}^{2},R_n^{1,2k},R_n^{2,2k}$, for any $n+2k\leq 2m$, one has 
 	\begin{align}
    &\sup_{0<t\leq T}\left(t^{\frac{n+2k+1}{2}}\left(\|\tilde{F}_{n,2k}^{1}(t)\|_{L^\infty}+\|R_n^{1,2k}(t)\|_{L^\infty}\right)+t^{\frac{n+2k+1+\kappa}{2}}\left(\|\tilde{F}_{n,2k}^{1}(t)\|_{C^\kappa}+\|R_n^{1,2k}(t)\|_{C^\kappa}\right)\right)\nonumber\\
 		&\quad+\sup_{0<t\leq T}\left(t^{\frac{n+2k}{2}}\left(\|\tilde{F}_{n,2k}^{2}(t)\|_{L^\infty}+\|R_n^{2,2k}(t)\|_{L^\infty}\right)+t^{\frac{n+2k+\kappa}{2}}\left(\|\tilde{F}_{n,2k}^{2}(t)\|_{C^\kappa}+\|R_n^{2,2k}(t)\|_{C^\kappa}\right)\right)\nonumber\\
 		&\quad+\sup_{0<t\leq T}t^{\frac{n+2k+1+\kappa}{2}}\left(\|(\tilde{F}_{n,2k}^{1},\nabla_{x'}\tilde{F}_{n,2k}^{2})(t)\|_{C^\kappa}+\|(R_n^{1,2k},\nabla_{x'}R_n^{2,2k})(t)\|_{C^\kappa}\right)\lesssim \mathbf{C}_{m}(F^1,F^2)(T).\label{eq4.24}
 	\end{align}
 	So we conclude from the estimate above and \eqref{mcbddhighest} that
    \begin{equation}\label{eq4.22}
        \begin{aligned}
            &\sup_{0<t\leq T}\sum_{n+2k\leq 2m}\left(t^{\frac{n+2k+\kappa}{2}}\|\nabla u_{n,2k}(t)\|_{C^\kappa}+t^{\frac{n+2k+1+\kappa}{2}}\|\nabla_{x'}\nabla u_{n,2k}(t)\|_{C^\kappa} \right)\\
            &\quad\lesssim \sum_{n+2k\leq2m}\sup_{0<t\leq T}t^{\frac{n+2k}{2}}\|\nabla u_{n,2k}(t)\|_{L^\infty}+\mathbf{C}_{m}(F^1,F^2)(T).
        \end{aligned}
    \end{equation}
    By the definition of $u_{n,2k}$, we still need to estimate the remainder terms other than $\partial_d^{2k}\nabla_x'^nu$, namely $F^{1,k-1}_n$, $\partial_dF^{2,k-1}_n$ and $R^{2k}_n$. By the form of $F_{n}^{i,k-1}$ for $i=1,2$ in \eqref{formFR} and the fact that $n+2k\leq 2m$, we deduce that
    \begin{equation}\label{estforrem}
        \begin{aligned}
            &\sup_{0<t\leq T}\left(t^{\frac{n+2k+\kappa}{2}}\|\nabla(F_{n}^{1,k-1}+\partial_dF_n^{2,k-1})(t)\|_{C^\kappa}+t^{\frac{n+2k+1+\kappa}{2}}\|\nabla_{x'}\nabla (F_{n}^{1,k-1}+\partial_dF_n^{2,k-1})(t)\|_{C^\kappa} \right)\\
            &\quad\quad\lesssim \mathbf{C}_{m}(F^1,F^2)(T).
        \end{aligned}
    \end{equation}
The term $R_n^{2k}$ cannot be estimated only by the
    forcing.  Indeed, by its definition in \eqref{formFR}, it contains
    derivatives of $u$ with strictly fewer than $2k-1$ normal derivatives.
    Introduce the lower-normal-order quantity
    \begin{equation*}\label{deflowernormal}
    \begin{aligned}
    \mathcal L_{n,k}(u;T):={}&
    \sum_{\substack{l+|\beta|=2k\\l<2k-1}}
    \sup_{0<t\leq T}\Bigl(
    t^{\frac{n+2k+\kappa}{2}}
    \|\nabla\partial_d^l\nabla_{x'}^{n+|\beta|}u(t)\|_{C^\kappa}\\
    &\hspace{8em}+t^{\frac{n+2k+1+\kappa}{2}}
    \|\nabla_{x'}\nabla\partial_d^l
    \nabla_{x'}^{n+|\beta|}u(t)\|_{C^\kappa}\Bigr).
    \end{aligned}
    \end{equation*}
    Since the coefficients in \eqref{formFR} are fixed constants, the
    triangle inequality gives the correct estimate
    \begin{equation}\label{esttanrem}
    \begin{aligned}
       &\sup_{0<t\leq T}\left(
       t^{\frac{n+2k+\kappa}{2}}\|\nabla R_{n}^{2k}(t)\|_{C^\kappa}
       +t^{\frac{n+2k+1+\kappa}{2}}
       \|\nabla_{x'}\nabla R_{n}^{2k}(t)\|_{C^\kappa}\right)
       \lesssim \mathcal L_{n,k}(u;T).
    \end{aligned}
    \end{equation}
     So \eqref{eq4.22}, \eqref{estforrem} and \eqref{esttanrem} lead to 
    \begin{equation}\label{eq4.25}
        \begin{aligned}
        &\sum_{n+2k\leq 2m}\sup_{0<t\leq T}\left(t^{\frac{n+2k+\kappa}{2}}\|\nabla \nabla_{x'}^{n}\partial_d^{2k}u(t)\|_{C^\kappa}+t^{\frac{n+2k+1+\kappa}{2}}\|\nabla \nabla_{x'}^{n+1}\partial_d^{2k} u(t)\|_{C^\kappa} \right)\\
            &\quad\quad\lesssim \sum_{n+2k\leq 2m}\sup_{0<t\leq T}t^{\frac{n+2k}{2}}\|\nabla u_{n,2k}(t)\|_{L^\infty}+\mathbf{C}_{m}(F^1,F^2)(T)+\sum_{\substack{n+2k\leq2m\\k\geq1}}
            \mathcal L_{n,k}(u;T)\\
            &\quad\quad\lesssim \sum_{n+2k\leq 2m}\sup_{0<t\leq T}t^{\frac{n+2k}{2}}\|\nabla\nabla_{x'}^{n}\partial_d^{2k} u(t)\|_{L^\infty}+\mathbf{C}_{m}(F^1,F^2)(T)+\sum_{\substack{n+2k\leq2m\\k\geq1}}
            \mathcal L_{n,k}(u;T).
        \end{aligned}
    \end{equation}
    We close this estimate by induction on the number of normal
    derivatives.  For $k=0$ the required bounds are precisely the tangential
    estimates \eqref{mcbddtanest}.  At the step $k\geq1$, every summand in
    $\mathcal L_{n,k}(u;T)$ has $l\leq2k-2$ normal derivatives and at least
    two additional tangential derivatives.  It is therefore among the
    lower-normal-order quantities retained in the induction hypothesis
    (with the same total parabolic weight $n+2k$).  Consequently
    \begin{equation*}\label{lowernormalinduction}
       \sum_{\substack{n+2k\leq2m\\k\geq1}}
       \mathcal L_{n,k}(u;T)
       \lesssim \sum_{n+2k\leq2m}\sup_{0<t\leq T}t^{\frac{n+2k}{2}}\|\nabla \nabla_{x'}^n\partial_d^{2k}u(t)\|_{L^\infty}
       +\mathbf C_m(F^1,F^2)(T),
    \end{equation*}
    once the estimates are proved in increasing normal order.\\
    On the other hand, interpolation together with \eqref{eq4.24} and \eqref{eq4.22} bounds the first term on the right-hand side of \eqref{eq4.25} by
    \begin{equation}\label{estmainest}
 	\begin{aligned}
 		\sup_{0<t\leq T}t^{\frac{n+2k}{2}}\|\nabla\nabla_{x'}^{n}\partial_d^{2k} u(t)\|_{L^\infty}&\leq \eta\sup_{0<t\leq T}\left(t^{\frac{n+2k+\kappa}{2}}\|\nabla \nabla_{x'}^{n}\partial_d^{2k}u(t)\|_{C^\kappa}+t^{\frac{n+2k+1+\kappa}{2}}\|\nabla \nabla_{x'}^{n+1}\partial_d^{2k} u(t)\|_{C^\kappa} \right)\\
        &\quad\quad+C(\eta)\left(\sup_{0<t\leq T}\|u(t)\|_{C^1}+\mathbf{C}_{m}(F^1,F^2)\right),\quad \forall \eta\in(0,1),
 	\end{aligned}
    \end{equation}
so we may take $\eta\ll 1$ sufficiently small to absorb this term into the left-hand side of \eqref{eq4.25}. Combining \eqref{eq4.25} and \eqref{estmainest} gives
    \begin{align*}
       & \sum_{n+2k\leq 2m}\sup_{0<t\leq T}\left(t^{\frac{n+2k+\kappa}{2}}\|\nabla\nabla_{x'}^n\partial_d^{2k} u(t)\|_{C^\kappa}+t^{\frac{n+2k+1+\kappa}{2}}\|\nabla\nabla_{x'}^{n+1}\partial_d^{2k} u(t)\|_{C^\kappa} \right)\\
        &\quad\lesssim \sup_{0<t\leq T}\|u(t)\|_{C^1}+ \mathbf{C}_{m}(F^1,F^2)(T),
    \end{align*}
  which completes the proof.
 \end{proof}
 We are now in a position to prove the local well-posedness result in Theorem~\ref{mcbdglo}.
 \begin{proof}[Proof of Theorem~\ref{mcbdglo}, i)] Let $f=\mathcal{S}g$ be the solution to the system \eqref{eqmcbd} for $g\in\mathcal{Z}_{T,\phi}^{m,\sigma}$ with $\phi=e^{\eps_1\Delta_\Omega}f_0$, where $\mathcal{Z}_{T,\phi}^{m,\sigma}$ is defined in \eqref{spacedef}, and  \begin{align*}
 		\|f_0-\phi\|_{W^{1,\infty}}\leq\eps_0,
 	\end{align*} 
 for some $\eps_0,\eps_1\ll 1$ to be fixed later.	First we consider the solution $f_{in}$ to \eqref{eqinter}. For $\eps_1$ small enough, we define the inner cut-off function $\chi_0$ satisfying 
    \begin{equation*}
        \chi_0(x)=0 \quad\text{for} \quad x\in \{\operatorname{dist}(x,\partial\Omega)\leq \eps_1^{\frac{1}{2}}\},\quad \chi_0(x)=1,\quad\text{for} \quad x\in \{\operatorname{dist}(x,\partial\Omega)\geq 2\eps_1^{\frac{1}{2}}\}. 
    \end{equation*}
    Denote $\tilde{f}_{in}=f_{in}-\chi_0\phi$, then $\tilde{f}_{in}$ satisfies
    \begin{equation*}
 		\begin{aligned}
 			&\partial_t\tilde{f}_{in}-\A[\nabla \phi]:\nabla^2\tilde{f}_{in}=\tilde{F}_{in},\quad \text{in}\ [0,T]\times\mathbb{R}^d,\\
 			&\tilde{f}_{in}|_{t=0}=\chi_0(f_0-\phi),\quad \text{in} \ \mathbb{R}^d,
 		\end{aligned}
 	\end{equation*}
 	with
 	\begin{equation*}
 		\tilde{F}_{in} =-2\A[\nabla g]:(\nabla f\otimes\nabla\chi_0)-\A[\nabla g]:(f\nabla^2\chi_0)+\A[\nabla g]:\nabla^2(\chi_0\phi)+(\A[\nabla g]-\A[\nabla\phi]):\nabla^2 \tilde{f}_{in}.
 	\end{equation*}
    By Theorem~\ref{propb=0}, since $g\in\mathcal{Z}_{T,\phi}^{m,\sigma}$ with small $\sigma$, there exists $T=T(\|\phi\|_{C^{2m+3}})>0$ such that
 	\begin{equation}\label{ee1ho}
 		\begin{aligned}
 			\|\tilde{f}_{in}\|_{Z^m_{T,1}}\lesssim \|\chi_0(f_0-\phi)\|_{ W^{1,\infty}}+\sup_{0<t\leq T}\left(t^\frac{1+\kappa}{2}\|\tilde{F}_{in}(t)\|_{ C^\kappa}+t^\frac{2m+1+\kappa}{2}\|\tilde{F}_{in}(t)\|_{ C^{2m+\kappa}}\right).
 		\end{aligned}
 	\end{equation}
	For the data term $\|\chi_0(f_0-\phi)\|_{W^{1,\infty}}$, we have
	    \[
	    \|\chi_0(f_0-\phi)\|_{W^{1,\infty}}\lesssim \|f_0-\phi\|_{W^{1,\infty}}+\|\nabla\chi_0(f_0-\phi)\|_{L^\infty}.
	    \]
	    Since $(f_0-\phi)|_{\partial\Omega}=0$, choose $\chi_0$ so that
	    $\operatorname{dist}(\operatorname{supp}\nabla\chi_0,\partial\Omega)
	    \sim \eps_1^{1/2}$ and
	    $\|\nabla\chi_0\|_{L^\infty}\lesssim\eps_1^{-1/2}$.
	    The trace condition and the Lipschitz bound imply
	    \[
	    |(f_0-\phi)\nabla\chi_0|
	    \lesssim \eps_1^{1/2}\|\nabla\chi_0\|_{L^\infty}
	    \|f_0-\phi\|_{\dot W^{1,\infty}}.
	    \]
	    Therefore,
 	\begin{equation}\label{mcbdbdysc}
 		\|\chi_0(f_0-\phi)\|_{W^{1,\infty}}\lesssim \|f_0-\phi\|_{W^{1,\infty}}.
 	\end{equation}
	Then interpolation, Lemma~\ref{lemcom}, the definition of $\tilde{F}_{in}$, the uniform ellipticity of $\A$, and \eqref{holembed} give
    \begin{equation}\label{estF0}
    \begin{aligned}
        &\|\tilde{F}_{in}(t)\|_{\dot C^{n+\kappa}}\lesssim \sum_{l\leq n}\sum_{l_1+l_2=l}\|\nabla^2f_{in}(t)\|_{\dot C^{n-l+\kappa}}\|\nabla (g-\phi)(t)\|_{\dot C^{l_1}}\left(\|\nabla (g,\phi)(t)\|_{\dot C^1}^{l_2}+\|\nabla (g,\phi)(t)\|_{\dot C^{l_2}}\right)\\
        &\quad\quad+\sum_{l\leq n}\sum_{l_1+l_2=l}\|\nabla^2f_{in}(t)\|_{\dot C^{n-l}}\|\nabla(g-\phi)(t)\|_{\dot C^{l_1}}\left(\|\nabla (g,\phi)(t)\|_{\dot C^\kappa}^{\frac{l_2+\kappa}{\kappa}}+\|\nabla (g,\phi)(t)\|_{\dot C^{l_2+\kappa}}\right)\\
        &\quad\quad+\sum_{l\leq n}\sum_{l_1+l_2=l}\|\nabla^2f_{in}(t)\|_{\dot C^{n-l}}\|\nabla(g-\phi)(t)\|_{\dot C^{l_1+\kappa}}\left(\|\nabla (g,\phi)(t)\|_{\dot C^1}^{l_2}+\|\nabla (g,\phi)(t)\|_{\dot C^{l_2}}\right)\\
        &\quad\quad+\|\chi_0\|_{C^{n+3}}\sum_{l_1+l_2\leq n}\left(\| f(t)\|_{\dot C^{l_1+1+\kappa}}\|\A[\nabla g](t)\|_{C^{l_2}}+\| f(t)\|_{\dot C^{l_1+1}}\|\A[\nabla g](t)\|_{C^{l_2+\kappa}}\right)\\
        &\quad\quad+\|\phi\|_{C^{n+3}}(1+\|\chi_0\|_{C^{n+3}})\|\nabla g(t)\|_{C^{n+\kappa}}\\
        &\quad\lesssim t^{-\frac{n+1+\kappa}{2}}\Bigl[
        \left(\|g-\phi\|_{Z_T^m}
        +T^{\frac{\kappa}{2}}(1+T)^{m}\|\phi\|_{C^{n+3}}\right)\\
        &\qquad\times\left(1+\|g-\phi\|_{Z_T^m}
        +T^{\frac{\kappa}{2}}(1+T)^{m}\|\phi\|_{C^{n+3}}\right)^m
        \|\tilde{f}_{in}\|_{Z_{T,1}^m}\\
        &\quad\quad+T^{\frac{\kappa}{2}}(1+T)^m
        \left(1+\|g-\phi\|_{Z_T^m}
        +T^{\frac{\kappa}{2}}\|\phi\|_{C^{n+3}}\right)^m\\
        &\qquad\times\left(\|f-\phi\|_{Z^{m,\mathrm{low}}_T}
        +T^{\frac{\kappa}{2}}\|\phi\|_{C^{n+3}}\right)\Bigr],
        \qquad n\leq 2m.
        \end{aligned}
        \end{equation}
In case of $g\in \mathcal{Z}_{T,\phi}^{m,\sigma}$, we derive from \eqref{ee1ho}, \eqref{mcbdbdysc} and \eqref{estF0} that
\begin{equation}\label{bddmciorh}
    \begin{aligned}
        &\|\tilde{f}_{in}\|_{Z_{T,1}^m}\leq C_1{\|\chi_0\|_{C^{2m+3}}}\left(\|f_0-\phi\|_{W^{1,\infty}}+\left(\sigma+ T^{\frac{\kappa}{2}}(1+T)^m\|\phi\|_{C^{2m+3}}\right)\|\tilde{f}_{in}\|_{Z_{T,1}^m}\right.\\
        &\quad\quad\left.+T^{\frac{\kappa}{2}}(1+T)^{m}(1+\sigma+T^{\frac{\kappa}{2}}\|\phi\|_{C^{2m+3}})^m(\|f-\phi\|_{Z_T^{m,\mathrm{low}}}+T^{\frac{\kappa}{2}}\|\phi\|_{C^{2m+3}}) \right) ,
    \end{aligned}
\end{equation}
  for some $\sigma, T$ small such that $\sigma+T^{\frac{\kappa}{2}}\|\phi\|_{C^{2m+3}}\ll 1$. So since $\|f-\phi\|_{Z_T^{m,\mathrm{low}}}\lesssim \|f-\phi\|_{Z_T^m}$ and $\|f_0-\phi\|_{W^{1,\infty}}\leq \eps_0$, by taking 
 	\begin{equation*}
	 	\sigma\leq\frac{1}{2^{10dm}C_1\|\chi_0\|_{C^{2m+3}}},\quad \eps_0\leq\frac{\sigma}{2^{10dm}C_1},\quad T\leq (\frac{\sigma}{2^{10dm}C_1\|\phi\|_{C^{2m+3}}})^{\frac{200}{\kappa}},
	 \end{equation*}
 	then \eqref{bddmciorh} leads to
\begin{equation}\label{bdmczt1}
 		\|\tilde f_{in}\|_{Z_{T,1}^m}\leq \frac{1}{10}(\sigma+\|f-\phi\|_{Z_T^m}).
 	\end{equation}

For the boundary estimate, unlike in the half-space case, higher-order normal derivatives do not vanish on the boundary in general bounded domains. To avoid uncontrollable boundary terms, we reorganize the forcing terms using Lemma~\ref{mcbdlembdy}, which expresses even normal derivatives through derivatives of the forcing and lower-normal-order derivatives of the solution.
First we derive the lower-order estimate in the $Z_{T,2}^0$ norm. We use the notation of Section~\ref{sec4.1}. Define $(f_i,{\phi}_i)=(\chi_if,\chi _i\phi)\circ\Phi_i\circ\Theta_i$ and $\tilde{f}_i=f_i-{\phi_i}$. By definition, $\tilde{f}_i$ satisfies
\begin{equation*}
    \begin{aligned}
        &\partial_t\tilde f_i(t,x)-\Delta \tilde f_i(t,x)=\tilde{F}_i[f,g](t,x),\quad \text{in}\ [0,T]\times\mathbb{R}^d_+,\\
 			&\tilde f_i|_{t=0}=f_{i,0}-{\phi}_i,\quad \text{in}\ \mathbb{R}^d_+,\\
 			&\tilde f_i(t,x)=0,\ \text{on}\ [0,T]\times\partial \mathbb{R}^d_+,
 		\end{aligned}
 	\end{equation*}
 	where $\tilde{F}_i$ has the form
 	\begin{equation}\label{eq4.34}
 		\begin{aligned}
 			&\tilde{F}_i[f,g]=F_i+\Delta{\phi}_i,
    \end{aligned}
\end{equation}
with $F_i$ defined in \eqref{deftilF}. Then by taking $F_1=\tilde{F}_i$, $F_2=0$ in Lemma~\ref{mcbdlemby}, we can prove that
    \begin{equation}\label{estz0}
        \|\tilde{f}_i\|_{Z_{T,2}^0}\leq C_2\|\tilde{f}_{i,0}\|_{\dot W^{1,\infty}}+C_2\sup_{0<t\leq T}\left(t^{\frac{1}{2}}\|\tilde{F}_i(t)\|_{L^\infty}+t^{\frac{1+\kappa}{2}}\|\tilde{F}_i(t)\|_{\dot C^{\kappa}}\right).
    \end{equation}
	    For brevity, set
	    \[
	    \mathcal H_i[f,g]
	    :=\bigl(G_i[f,g]\circ\Phi_i+\mathcal R_i[f,g]\bigr)\circ\Theta_i.
	    \]
	    By \eqref{eq4.34}, the uniform bounds for $\Phi_i,\Theta_i$, and the
	    smallness condition \eqref{eq4.6},
    \begin{equation*}
        \begin{aligned}
	           & \sup_{0<t\leq T}\left(t^{\frac{1}{2}}\|\tilde{F}_i(t)\|_{L^\infty}
	           +t^{\frac{1+\kappa}{2}}\|\tilde{F}_i(t)\|_{\dot C^{\kappa}}\right)\\
	           &\quad\lesssim\sup_{0<t\leq T}\left(t^{\frac{1}{2}}\|F_i(t)\|_{L^\infty}
	           +t^{\frac{1+\kappa}{2}}\|F_i(t)\|_{\dot C^{\kappa}}\right)
	           +T^{\frac{1}{2}}(1+T)\|\tilde{\phi}_i\|_{C^{3}}\\
		           &\quad\lesssim\sup_{0<t\leq T}\left(t^{\frac{1}{2}}\|\mathcal H_i[f,g](t)\|_{L^\infty}
		           +t^{\frac{1+\kappa}{2}}\|\mathcal H_i[f,g](t)\|_{\dot C^{\kappa}}\right)
		           +T^{\frac{1}{2}}(1+T)\|\tilde{\phi}_i\|_{C^{3}}\\
           &\quad\quad+\sup_{0<t\leq T}\left(\|\tilde{\A}[\nabla\tilde g_i(\nabla\Phi_i)^{-1}](t)-\tilde{\A}_0\|_{L^\infty}\left(t^{\frac{1}{2}}\|\nabla^2f_i(t)\|_{L^\infty}+t^{\frac{1+\kappa}{2}}\|\nabla^2f_i(t)\|_{C^{\kappa}}\right)\right)\\
           &\quad\quad+\sup_{0<t\leq T}t^{\frac{1+\kappa}{2}}\left(\|\tilde{\A}[\nabla\tilde g_i(\nabla\Phi_i)^{-1}](t)\|_{\dot C^{\kappa}}\|\nabla^2f_i(t)\|_{L^\infty}\right).
        \end{aligned}
    \end{equation*}
  This is not a closed estimate in $Z_{T,2}^0$ alone: its right-hand side contains the higher norm $Z_{T,2}^1$. It is therefore used only as one component of the coupled full-$Z_{T,2}^m$ estimate below.
     The terms $G_i$ and $\mathcal R_i$ are subcritical. The contributions containing $\nabla^2f_i$ involve two normal derivatives and therefore are not controlled by $Z_T^0$ alone; they are controlled by $Z_T^1$, which includes normal regularity through order $3+\kappa$. Consequently,
     \begin{equation*}
         \begin{aligned}
             &\sup_{0<t\leq T}\left(t^{\frac{1}{2}}\|\tilde{F}_i(t)\|_{L^\infty}
             +t^{\frac{1+\kappa}{2}}\|\tilde{F}_i(t)\|_{\dot C^{\kappa}}\right)\\
             &\quad\lesssim T^{\frac{1}{2}}(1+T)
             (1+\|g-\phi\|_{Z_T^0}+T^{\frac{\kappa}{2}}\|\phi\|_{C^3})\\
             &\qquad\quad\times
             (\|f-\phi\|_{Z_T^{0,\mathrm{low}}}+T^{\frac{\kappa}{2}}\|\phi\|_{C^3})\\
            &\qquad+(\eps+\|g-\phi\|_{Z_T^0}+T^{\frac{\kappa}{2}}\|\phi\|_{C^3})
            (\|\tilde{f}_i\|_{Z_T^1}+T^{\frac{\kappa}{2}}\|\phi\|_{C^3}),
         \end{aligned}
     \end{equation*}
     where we use \eqref{eq4.6} to claim the smallness of $\|\tilde{\A}[\nabla\tilde g_i(\nabla\Phi_i)^{-1}](t)-\tilde{\A}_0\|_{L^\infty}$.
    \begin{remark}\label{rmk4.2}
    The forcing estimate contains $\partial_d^2\tilde f_i$, so the estimate at the $Z_{T,2}^0$ level is not closed by itself. At higher order this obstruction is removed: for $l\geq1$, we rewrite
    \begin{equation*}
      \tilde{\B}[\nabla\tilde{ g}_i]_{dd}\partial_{d}^{2l+2}f_i=\partial_d\left(\tilde{\B}[\nabla\tilde{ g}_i]_{dd}\partial_{d}^{2l+1}f_i\right)-\partial_d\left(\tilde{\B}[\nabla\tilde{ g}_i]_{dd}\right)\partial_{d}^{2l+1}f_i,
    \end{equation*}
    but for $l=0$ we cannot do like this, since then $\partial_d\left(\tilde{\B}[\nabla\tilde{ g}_i]_{dd}\right)$ becomes the highest order term.
    \end{remark}
    
    Now we are going to estimate the derivatives. First, for the tangent derivatives, by definition of $\tilde{F}$, \eqref{defnormZ}, \eqref{mcbddtanest} and $g\in \mathcal{Z}_{T,\phi}^{m,\sigma}$, we have
    \begin{equation}\label{esttg1.3}
        \begin{aligned}
        &\sum_{l\leq 2m+1}\sup_{0<t\leq T}\left(t^{\frac{l+\kappa}{2}}\|\nabla_{x'}^l\nabla \tilde{f}_i (t)\|_{C^{\kappa}}\right)\lesssim \|\tilde{f}_{i,0}\|_{W^{1,\infty}}+\sum_{l\leq 2m+1}\sup_{0<t\leq T}t^{\frac{1+l+\kappa}{2}}\|\nabla_{x'}^l\tilde{F}_i(t)\|_{C^{\kappa}}\\
        &\quad\lesssim\|\tilde{f}_{i,0}\|_{W^{1,\infty}}+\left(\eps+\sigma+T^{\frac{\kappa}{2}}\|\phi\|_{C^{2m+3}}\right)(1+\sigma+T^{\frac{\kappa}{2}}\|\phi\|_{C^{2m+3}})^{2m+3}(\|f-\phi\|_{Z_T^m}+T^{\frac{\kappa}{2}}\|\phi\|_{C^{2m+3}})\\
        &\quad\quad+\|\tilde{\chi}_i\|_{C^{m}}T^{\frac{\kappa}{2}}(1+T)^{m}(1+\sigma+T^{\frac{\kappa}{2}}\|\phi\|_{C^{2m+3}})^{2m+3}(\|f-\phi\|_{Z_T^{m,\mathrm{low}}}+T^{\frac{\kappa}{2}}\|\phi\|_{C^{2m+3}}).
            \end{aligned}
    \end{equation}
For higher-order normal derivatives, due to Remark~\ref{rmk4.2}, we will study the equation of $\partial_d^2\tilde{f}_i$ instead. Note that $\tilde{f}_i(t)|_{\partial\mathbb{R}_+^d}=\partial_t\tilde{f}_i(t)|_{\partial\mathbb{R}^d_+}=0$ for any $t\in[0,T]$, then \eqref{eqbdref} implies
\begin{equation*}
    \partial_{dd}\tilde{f}_i+\Delta_{x'}\tilde{f}_i+\tilde{F}_i[f,g]=0,\quad \forall (t,x)\in [0,T]\times \{x_d=0\}.
\end{equation*}
	Thus define
	\[
	\tilde{f}_i'=\partial_{dd}\tilde{f}_i+\Delta_{x'}\tilde{f}_i+\tilde{F}_i[f,g].
	\]
and we can write the equation of $\tilde{f}_i'$ as
\begin{equation}\label{eq4.72}
    \begin{aligned}
        &\partial_t\tilde{f}_i'-\Delta\tilde{f}_i'=\partial_{dd}\tilde{F}_i+(\partial_{t}-\Delta)(\Delta_{x'}\tilde{f}_i+\tilde{F}_i)=\partial_t\tilde{F}_i,\\
&\tilde{f}_i'|_{\partial\mathbb{R}_+^d}=0.
    \end{aligned}
\end{equation}
We need to make further decomposition for the forcing term $\partial_t\tilde{F}_i$, which can be written as
\begin{equation}\label{defFG}
\begin{aligned}
   &\partial_t\tilde{F}_i=G^1_i+\partial_{d}G^2_i ,\\
   &G^1_i=\partial_t\left(\tilde{F}_{1,i}+\tilde{F}_{2,i}^1\right)+\partial_t(\tilde{\B}[\nabla \tilde{g}_i]_{dd})\partial_d^2
{f}_i-\partial_d(\tilde{\B}[\nabla \tilde{g}_i]_{dd})\partial_d\partial_t{f}_i,\\
   &G^2_i=\tilde{\B}[\nabla \tilde{g}_i]_{dd}\partial_d\partial_t{f}_i,\\
	   &\tilde{F}_{1,i}=\left(G_i[f,g]\circ\Phi_i+\mathcal R_i[f,g]\right)\circ\Theta_i+\Delta\tilde{\phi}_i,\\
	   &\tilde{F}_{2,i}^1=\sum_{l\neq d\ \text{or}\ j\neq d}\tilde{\B}[\nabla \tilde g_i]_{lj}\partial_{lj} f_i,\\
   &\tilde{\B}[\nabla \tilde{g}_i]=\Theta_i^{-1}\left((\tilde{\A}[\nabla\tilde g_i(\nabla\Phi_i)^{-1}]-\tilde{\A}_0)\circ\Theta_i\right)\Theta_i^{-\top}.
\end{aligned}
\end{equation}
Since $\tilde{f}_i'(0)$ is not defined, for any $t>0$ take $t_0=\frac{t}{2}$ and use $\tilde{f}_i'(t_0)$ as the initial datum. Then
\begin{align*}
	    &t\|\tilde{f}_i'(t)\|_{C^1}+\sum_{n+2k+2l\leq 2m-1}t^{1+\frac{n+2k+2l+\kappa}{2}}\|\nabla_{x'}^n\partial_d^{2k}\nabla\partial_t^l\tilde{f}_i'(t)\|_{C^\kappa}\\
	    &\qquad\approx(t-t_0)\|\tilde{f}_i'(t)\|_{C^1}+\sum_{n+2k+2l\leq 2m-1}(t-t_0)^{1+\frac{n+2k+2l+\kappa}{2}}\|\nabla_{x'}^n\partial_d^{2k}\nabla\partial_t^l\tilde{f}_i'(t)\|_{C^\kappa}.
\end{align*}
Furthermore,
\begin{align*}
	    &\sup_{\tau\in[t_0,T]}\sum_{2i+j\leq 2m-2}\left((\tau-t_0)^{\gamma+i+\frac{1+j}{2}}\|\partial_\tau^i\nabla^{j}G_i^{1}(\tau)\|_{L^\infty}+(\tau-t_0)^{\gamma+i+\frac{1+j+\kappa}{2}}\|\partial_\tau^i\nabla^{j}(G_i^{1},\nabla_{x'}G_i^2)(\tau)\|_{C^\kappa}\right.\\
	    &\quad\quad\left.+(\tau-t_0)^{\gamma+i+\frac{j}{2}}\|\partial_\tau^i\nabla^{j}G_i^2(\tau)\|_{L^\infty}+(\tau-t_0)^{\gamma+i+\frac{j+\kappa}{2}}\|\partial_\tau^i\nabla^{j}G_i^2(\tau)\|_{C^\kappa}\right)\lesssim \mathbf{C}_{m-1}^1(G_i^1,G_i^2)(T).
\end{align*}
By applying $F^1=G^1$, $F^2=G^2$, $\gamma=1$ in Lemma~\ref{maincor}, for any $t\in(0,T]$, we get
\begin{equation}\label{estmix1.3}
    \begin{aligned}
	        &t\|\tilde{f}_i'(t)\|_{C^1}+\sum_{n+2k+2l\leq 2m-1}t^{1+\frac{n+2k+2l+\kappa}{2}}\|\nabla_{x'}^n\partial_d^{2k}\nabla\partial_t^l\tilde{f}_i'(t)\|_{C^\kappa}\\
        &\qquad\lesssim \sup_{\tau\in(0,T]}\tau\|\tilde{f}_i'(\tau)\|_{C^1}+\mathbf{C}_{m-1}^1(G^1_i,G^2_i)(T),\quad\forall m\in\mathbb{N}.
    \end{aligned}
\end{equation}
Here we use $Z_{T}^{m-1}$ because $\|\tilde{f}_i'\|_{Z_{T,2}^{m-1,1}}$ is of the same differential order as $\|f-\phi\|_{Z_T^m}$.
Since $\tilde f_i'=\partial_t\tilde f_i$, interpolation between the
$t^{(1+\kappa)/2}C^\kappa$ and $t^{1+\kappa/2}C^{1+\kappa}$
components of the $Z_{T,2}^m$ norm gives, for every $\delta\in(0,1)$,
\begin{equation}\label{eq4.fprime-interpolation}
 \sup_{0<t\leq T}t\|\tilde f_i'(t)\|_{C^1}
 \leq \delta\|\tilde f_i\|_{Z_{T,2}^m}
 +C_\delta\|\tilde f_i\|_{Z_{T,2}^0}.
\end{equation}
Indeed, the only nontrivial term is $t\|\nabla\partial_t\tilde f_i\|_{L^\infty}$;
the interpolation $C^1=(C^\kappa,C^{1+\kappa})_{1-\kappa,\infty}$
has exactly the time weight
$t^{\kappa(1+\kappa)/2+(1-\kappa)(1+\kappa/2)}=t$.

We have the following lemma for the quantity $\mathbf{C}_{m-1}^1(G^1_i,G^2_i)$.
\begin{lemma}\label{lem4.3}
    For $G^1_i$, $G^2_i$ defined in \eqref{defFG}, and for $\mathbf{C}_{m-1}^\gamma(G^1_i,G^2_i)$ defined in \eqref{defCgam}, we have
\begin{equation*}
    \begin{aligned}
        &\mathbf{C}_{m-1}^1(G^1_i,G^2_i)
        \lesssim T^{\frac{1}{2}}(1+T)^m
        (1+\|g-\phi\|_{Z_T^{m,\mathrm{low}}}
        +T^{\frac{\kappa}{2}}\|\phi\|_{C^{2m+3}})^{2m+3}\\
        &\qquad\times(\|f-\phi\|_{Z_T^{m,\mathrm{low}}}
        +T^{\frac{\kappa}{2}}\|\phi\|_{C^{2m+3}})\\
        &\quad+\left(\eps+\|g-\phi\|_{Z_T^m}
        +T^{\frac{\kappa}{2}}\|\phi\|_{C^{2m+3}}\right)
        (1+\|g-\phi\|_{Z_T^m}
        +T^{\frac{\kappa}{2}}\|\phi\|_{C^{2m+3}})^{2m+3}\\
        &\qquad\times(\|f-\phi\|_{Z_T^m}
        +T^{\frac{\kappa}{2}}\|\phi\|_{C^{2m+3}})
    \end{aligned}
\end{equation*}
\end{lemma}
\begin{proof}
First we notice that $\tilde{F}_{1,i}$ is subcritical, which enables us to derive the following estimate
\begin{equation*}
    \begin{aligned}
        \mathbf{C}_{m-1}^1(\tilde{F}_{1,i},0)\lesssim T^{\frac{1}{2}}(1+T)^m(1+\|g-\phi\|_{Z_T^{m,\mathrm{low}}}+T^{\frac{\kappa}{2}}\|\phi\|_{C^{2m+3}})^{2m+3}(\|f-\phi\|_{Z_T^{m,\mathrm{low}}}+T^{\frac{\kappa}{2}}\|\phi\|_{C^{2m+3}})
    \end{aligned}
\end{equation*}
The remaining terms are critical. First we consider $\tilde{F}_{2,i}^1$. By Lemma~\ref{lemcom}, for any $2i+j\leq 2m-2$, we have
\begin{align*}
        &\|\partial_t^{l+1}\nabla^j\tilde{F}_{2,i}^1(t)\|_{L^\infty}\lesssim \sum_{l_1+l_2=l+1}\sum_{j_1+j_2=j}\|\partial_t^{l_1}\nabla^{j_1}\tilde{\B}[\nabla\tilde{g}_i](t)\|_{L^\infty}\|\partial_t^{l_2}\nabla_{x'}\nabla^{1+j_2} {f}_i(t)\|_{L^\infty}\\
        &\quad\lesssim t^{-1-l-\frac{j+1}{2}}\left(\|\tilde{f}_i\|_{Z_{T,2}^{m}}+T^{\frac{\kappa}{2}}(1+T)^m\|\phi\|_{C^{m+3}}\right)\left(\eps+\|g-\phi\|_{Z_T^m}+T^{\frac{\kappa}{2}}(1+T)^m\|\phi\|_{C^{m+3}} \right),\\
        &\|\partial_t^{l+1}\nabla^j\tilde{F}_{2,i}^1(t)\|_{\dot C^\kappa}\lesssim \sum_{l_1+l_2=l+1}\sum_{j_1+j_2=j}\left(\|\partial_t^{l_1}\nabla^{j_1}\tilde{\B}[\nabla\tilde{g}_i](t)\|_{\dot C^\kappa}\|\partial_t^{l_2}\nabla_{x'}\nabla ^{1+j_2}{f}_i(t)\|_{L^\infty}\right.\\
        &\quad\quad\left.+\|\partial_t^{l_1}\nabla^{j_1}\tilde{\B}[\nabla\tilde{g}_i](t)\|_{L^\infty}\|\partial_t^{l_2}\nabla_{x'}\nabla^{1+j_2} {f}_i(t)\|_{\dot C^\kappa}\right)\\
        &\quad\lesssim t^{-1-l-\frac{j+1+\kappa}{2}}\left(\|\tilde{f}_i\|_{Z_{T,2}^{m}}+T^{\frac{\kappa}{2}}(1+T)^m\|\phi\|_{C^{m+3}}\right)\left(\eps+\|g-\phi\|_{Z_T^m}+T^{\frac{\kappa}{2}}(1+T)^m\|\phi\|_{C^{m+3}} \right).
    \end{align*}
Here we use the smallness condition \eqref{eq4.6} and the fact that $\tilde{F}_{2,i}^1$ contains no highest-order normal derivative of ${f}_i$.
For the other terms in $G_i^1$, since some derivatives are moved to coefficients $\tilde{\B}[\nabla\tilde{g}_i]$, there is no highest order normal derivatives on $\tilde{f}_i$ either, and for $2l+j\leq 2m-2$, and by Lemma~\ref{lemcom}, we get the $L^\infty$ estimates
\begin{align*}
        &\|\partial_t^l\nabla^j\left(\partial_t(\tilde{\B}[\nabla \tilde{g}_i]_{dd})\partial_d^2{f}_i\right)\|_{L^\infty}\lesssim \sum_{l_1+l_2=l}\sum_{j_1+j_2=j}\|\partial_t^{l_1+1}\nabla^{j_1}\tilde{\B}[\nabla \tilde{ g}_i](t)\|_{L^\infty}\|\partial_t^{l_2}\nabla^{j_2+2}f_i(t)\|_{L^\infty}\\
        &\quad\lesssim t^{-1-l-\frac{j+1}{2}}\left(\|\tilde{f}_i\|_{Z_{T,2}^{m}}+T^{\frac{\kappa}{2}}(1+T)^m\|\phi\|_{C^{m+3}}\right)\left(\eps+\|g-\phi\|_{Z_T^m}+T^{\frac{\kappa}{2}}(1+T)^m\|\phi\|_{C^{m+3}} \right),\\
        &\|\partial_t^l\nabla^j\left(\partial_d(\tilde{\B}[\nabla \tilde{g}_i]_{dd})\partial_d\partial_t{f}_i\right)\|_{L^\infty}\lesssim \sum_{l_1+l_2=l}\sum_{j_1+j_2=j}\|\partial_t^{l_1}\nabla^{j_1+1}\tilde{\B}[\nabla \tilde{ g}_i](t)\|_{L^\infty}\|\partial_t^{l_2+1}\nabla^{j_2+1}f_i(t)\|_{L^\infty}\\
        &\quad\lesssim t^{-1-l-\frac{j+1}{2}}\left(\|\tilde{f}_i\|_{Z_{T,2}^{m}}+T^{\frac{\kappa}{2}}(1+T)^m\|\phi\|_{C^{m+3}}\right)\left(\eps+\|g-\phi\|_{Z_T^m}+T^{\frac{\kappa}{2}}(1+T)^m\|\phi\|_{C^{m+3}} \right),
        \end{align*}
        and the the $\dot C^\kappa$ estimates
        \begin{align*}
        &\|\partial_t^l\nabla^j\left(\partial_t(\tilde{\B}[\nabla \tilde{g}_i]_{dd})\partial_d^2{f}_i\right)\|_{\dot C^\kappa}\lesssim \sum_{l_1+l_2=l}\sum_{j_1+j_2=j}\|\partial_t^{l_1+1}\nabla^{j_1}\tilde{\B}[\nabla \tilde{ g}_i](t)\|_{\dot C^\kappa}\|\partial_t^{l_2}\nabla^{j_2+2}f_i(t)\|_{L^\infty}\\
        &\quad\quad+\sum_{l_1+l_2=l}\sum_{j_1+j_2=j}\|\partial_t^{l_1+1}\nabla^{j_1}\tilde{\B}[\nabla \tilde{ g}_i](t)\|_{L^\infty}\|\partial_t^{l_2}\nabla^{j_2+2}f_i(t)\|_{\dot C^\kappa}\\
        &\quad\lesssim t^{-1-l-\frac{j+1+\kappa}{2}}\left(\|\tilde{f}_i\|_{Z_{T,2}^{m}}+T^{\frac{\kappa}{2}}(1+T)^m\|\phi\|_{C^{m+3}}\right)\left(\eps+\|g-\phi\|_{Z_T^m}+T^{\frac{\kappa}{2}}(1+T)^m\|\phi\|_{C^{m+3}} \right),\\
        &\|\partial_t^l\nabla^j\left(\partial_d(\tilde{\B}[\nabla \tilde{g}_i]_{dd})\partial_d\partial_t{f}_i\right)\|_{\dot C^\kappa}\lesssim \sum_{l_1+l_2=l}\sum_{j_1+j_2=j}\|\partial_t^{l_1}\nabla^{j_1+1}\tilde{\B}[\nabla \tilde{ g}_i](t)\|_{\dot C^\kappa}\|\partial_t^{l_2+1}\nabla^{j_2+1}f_i(t)\|_{L^\infty}\\
        &\quad\quad+\sum_{l_1+l_2=l}\sum_{j_1+j_2=j}\|\partial_t^{l_1}\nabla^{j_1+1}\tilde{\B}[\nabla \tilde{ g}_i](t)\|_{L^\infty}\|\partial_t^{l_2+1}\nabla^{j_2+1}f_i(t)\|_{\dot C^\kappa}\\
        &\quad\lesssim t^{-1-l-\frac{j+1+\kappa}{2}}\left(\|\tilde{f}_i\|_{Z_{T,2}^{m}}+T^{\frac{\kappa}{2}}(1+T)^m\|\phi\|_{C^{m+3}}\right)\left(\eps+\|g-\phi\|_{Z_T^m}+T^{\frac{\kappa}{2}}(1+T)^m\|\phi\|_{C^{m+3}} \right).
    \end{align*}
The estimates above imply that
\begin{equation}\label{eq4.60}
    \begin{aligned}
        \mathbf{C}_{m-1}^1(G_i^1,0)&\lesssim
        \left(\|\tilde{f}_i\|_{Z_{T,2}^{m}}
        +T^{\frac{\kappa}{2}}(1+T)^m\|\phi\|_{C^{m+3}}\right)\\
        &\quad\times\left(\eps+\|g-\phi\|_{Z_T^m}
        +T^{\frac{\kappa}{2}}(1+T)^m\|\phi\|_{C^{m+3}} \right)\\
        &\quad+T^{\frac{1}{2}}(1+T)^m
        (1+\|g-\phi\|_{Z_T^{m,\mathrm{low}}}
        +T^{\frac{\kappa}{2}}\|\phi\|_{C^{2m+3}})^{2m+3}\\
        &\qquad\times(\|f-\phi\|_{Z_T^{m,\mathrm{low}}}
        +T^{\frac{\kappa}{2}}\|\phi\|_{C^{2m+3}}).
    \end{aligned}
\end{equation}
Finally we consider $G_i^2$. For the non-endpoint norms with $2l+j\leq 2m-2$, we have
\begin{align*}
    &\|\partial_t^l\nabla^jG_i^2(t)\|_{L^\infty}\lesssim \sum_{l_1+l_2=l}\sum_{j_1+j_2=j}\|\partial_t^{l_1}\nabla^{j_1}\tilde{\B}[\nabla\tilde{g}_i](t)\|_{L^\infty}\|\partial_t^{l_2+1}\nabla^{j_2+1}f_i(t)\|_{L^\infty}\\
    &\quad\lesssim t^{-1-l-\frac{j}{2}}\left(\|\tilde{f}_i\|_{Z_{T,2}^{m}}+T^{\frac{\kappa}{2}}(1+T)^m\|\phi\|_{C^{m+3}}\right)\left(\eps+\|g-\phi\|_{Z_T^m}+T^{\frac{\kappa}{2}}(1+T)^m\|\phi\|_{C^{m+3}} \right),\\
    &\|\partial_t^l\nabla^jG_i^2(t)\|_{\dot C^\kappa}\lesssim \sum_{l_1+l_2=l}\sum_{j_1+j_2=j}\|\partial_t^{l_1}\nabla^{j_1}\tilde{\B}[\nabla\tilde{g}_i](t)\|_{\dot C^\kappa}\|\partial_t^{l_2+1}\nabla^{j_2+1}f_i(t)\|_{L^\infty}\\
    &\quad\quad+\sum_{l_1+l_2=l}\sum_{j_1+j_2=j}\|\partial_t^{l_1}\nabla^{j_1}\tilde{\B}[\nabla\tilde{g}_i](t)\|_{L^\infty}\|\partial_t^{l_2+1}\partial_d^{j_2+1}f_i(t)\|_{\dot C^\kappa}\\
    &\quad\lesssim t^{-1-l-\frac{j+\kappa}{2}}\left(\|\tilde{f}_i\|_{Z_{T,2}^{m}}+T^{\frac{\kappa}{2}}(1+T)^m\|\phi\|_{C^{m+3}}\right)\left(\eps+\|g-\phi\|_{Z_T^m}+T^{\frac{\kappa}{2}}(1+T)^m\|\phi\|_{C^{m+3}} \right).
\end{align*}
The preceding formula contains derivatives of $f_i$ of order at most $2m+1$. In the highest-order quantity $\|\partial_t^l\nabla^j\nabla_{x'}G_i^2(t)\|_{\dot C^\kappa}$, with $2l+j\leq2m-2$, no highest-order normal derivative occurs; all normal derivatives of $f_i$ still have order at most $2m+1$. Hence
\begin{align*}
    &\|\partial_t^l\nabla^j\nabla_{x'}G_i^2(t)\|_{\dot C^\kappa}\\
    &\quad\lesssim \sum_{l_1+l_2=l}\sum_{j_1+j_2\leq j}\sum_{k_1+k_2=1}\|\partial_t^{l_1}\nabla^{j_1}\nabla_{x'}^{k_1}\tilde{\B}[\nabla\tilde{g}_i](t)\|_{\dot C^\kappa}\|\partial_t^{l_2+1}\nabla^{j_2+1}\nabla^{k_2}_{x'}f_i(t)\|_{L^\infty}\\
    &\quad\quad+\sum_{l_1+l_2=l}\sum_{j_1+j_2\leq j}\sum_{k_1+k_2=1}\|\partial_t^{l_1}\nabla^{j_1}\nabla_{x'}^{k_1}\tilde{\B}[\nabla\tilde{g}_i](t)\|_{L^\infty}\|\partial_t^{l_2+1}\nabla^{j_2+1}\nabla^{k_2}_{x'}f_i(t)\|_{\dot C^\kappa}\\
    &\quad\lesssim t^{-1-l-\frac{j+1+\kappa}{2}}\left(\|\tilde{f}_i\|_{Z_{T,2}^{m}}+T^{\frac{\kappa}{2}}(1+T)^m\|\phi\|_{C^{m+3}}\right)\left(\eps+\|g-\phi\|_{Z_T^m}+T^{\frac{\kappa}{2}}(1+T)^m\|\phi\|_{C^{m+3}} \right).
\end{align*}
So the estimates above lead to
\begin{equation}\label{eq4.61}
    \begin{aligned}
        \mathbf{C}_{m-1}(0,G_i^2)\lesssim \left(\|\tilde{f}_i\|_{Z_{T,2}^{m}}+T^{\frac{\kappa}{2}}(1+T)^m\|\phi\|_{C^{m+3}}\right)\left(\eps+\|g-\phi\|_{Z_T^m}+T^{\frac{\kappa}{2}}(1+T)^m\|\phi\|_{C^{m+3}} \right)
    \end{aligned}
\end{equation}
Finally \eqref{eq4.60} and \eqref{eq4.61} concludes the proof of this lemma.
\end{proof}
\begin{remark}\label{rmk4.4}
    Due to the fact that
    \begin{equation*}
        \|fg\|_{\dot C^\kappa}\lesssim \|f\|_{\dot C^\kappa}\|g\|_{L^\infty}+\|f\|_{L^\infty}\|g\|_{\dot C^\kappa},
    \end{equation*}
    one can see that there is at most one component in the multi-linear estimates of the highest order. As a result, we can prove a stronger version
    \begin{equation*}
    \begin{aligned}
        &\mathbf{C}_{m-1}^1(G^1_i,G^2_i)
        \lesssim T^{\frac{1}{2}}(1+T)^m
        (1+\|g-\phi\|_{Z_T^{m,\mathrm{low}}}
        +T^{\frac{\kappa}{2}}\|\phi\|_{C^{2m+3}})^{2m+3}\\
        &\qquad\times(\|f-\phi\|_{Z_T^{m,\mathrm{low}}}
        +T^{\frac{\kappa}{2}}\|\phi\|_{C^{2m+3}})\\
        &\quad+\left(\eps+\|g-\phi\|_{Z_T^{m,\mathrm{low}}}
        +T^{\frac{\kappa}{2}}\|\phi\|_{C^{2m+3}}\right)
        (1+\|g-\phi\|_{Z_T^{m,\mathrm{low}}}
        +T^{\frac{\kappa}{2}}\|\phi\|_{C^{2m+3}})^{2m+3}\\
        &\qquad\times(\|f-\phi\|_{Z_T^m}
        +T^{\frac{\kappa}{2}}\|\phi\|_{C^{2m+3}})\\
        &\quad+\left(\eps+\|g-\phi\|_{Z_T^m}
        +T^{\frac{\kappa}{2}}\|\phi\|_{C^{2m+3}}\right)
        (1+\|g-\phi\|_{Z_T^{m,\mathrm{low}}}
        +T^{\frac{\kappa}{2}}\|\phi\|_{C^{2m+3}})^{2m+3}\\
        &\qquad\times(\|f-\phi\|_{Z_T^{m,\mathrm{low}}}
        +T^{\frac{\kappa}{2}}\|\phi\|_{C^{2m+3}}).
    \end{aligned}
\end{equation*}
We emphasize that in local regularity, Lemma~\ref{lem4.3} is enough. This remark only contributes to the global regularity, see Section~\ref{sec4.3}.
\end{remark}
In the following, we will prove the a priori estimates for $\|\nabla\nabla_{x'}^{l+1}\partial_d^{2k}\tilde{f}_i(t)\|_{C^\kappa}$ with $g\in \mathcal{Z}_{T,\phi}^{m,\sigma}$ and $l+2k\leq 2m-2$.  Precisely, we want to prove that
\begin{equation}\label{eq4.43}
        \begin{aligned}
        &\sum_{\substack{l+2k\leq 2m-2}}\sup_{0<t\leq T}
        t^{\frac{l+2k+1+\kappa}{2}}
        \|\nabla_{x'}^{l+1}\partial_d^{2k}\nabla \tilde{f}_i (t)\|_{C^{\kappa}}
        \leq C \sup_{0<t\leq T}\|\nabla\tilde{f}_i(t)\|_{L^\infty}\\
        &\quad+C\left(\eps+\sigma+T^{\frac{\kappa}{2}}\|\phi\|_{C^{2m+3}}\right)
        (1+\sigma+T^{\frac{\kappa}{2}}\|\phi\|_{C^{2m+3}})^{2m+3}\\
        &\qquad\times(\|f-\phi\|_{Z_T^m}
        +T^{\frac{\kappa}{2}}\|\phi\|_{C^{2m+3}})\\
        &\quad+CT^{\frac{\kappa}{2}}(1+T)^{m}\|\tilde{\chi}_i\|_{C^m}
        (1+\sigma+T^{\frac{\kappa}{2}}\|\phi\|_{C^{2m+3}})^{2m+3}\\
        &\qquad\times(\|f-\phi\|_{Z_T^{m,\mathrm{low}}}
        +T^{\frac{\kappa}{2}}\|\phi\|_{C^{2m+3}}).
            \end{aligned}
    \end{equation}
We argue by induction on $k$. The case $k=0$ follows from \eqref{esttg1.3}. Assume that \eqref{eq4.43} holds for $k\leq K-1$; it remains to prove the case $k=K$.
First we consider $\|\nabla\nabla^{l+1}_{x'}\partial_d^{2K-2}\tilde{f}_i'(t)\|_{C^\kappa}$. By applying Lemma~\ref{maincor} and Lemma~\ref{lem4.3} to the equation \eqref{eq4.72} and the fact that $g\in\mathcal{Z}_{T,\phi}^{m,\sigma}$, we have
\begin{equation}\label{eq4.36}
    \begin{aligned}
        &\sum_{\substack{l+2k\leq 2m-2\\ 1\leq k\leq K-1}}
        \sup_{0<t\leq T}t^{\frac{1+l+2k}{2}}
        \|\nabla^{l}_{x'}\partial_d^{2k}\tilde{f}_i'(t)\|_{L^\infty}\\
        &\quad+\sum_{\substack{l+2k\leq 2m-2\\ 1\leq k\leq K-1}}
        \sup_{0<t\leq T}t^{1+\frac{1+l+2k+\kappa}{2}}
        \|\nabla^{l+1}_{x'}\partial_d^{2k}\nabla\tilde{f}_i'(t)\|_{C^\kappa}\\
        &\quad\lesssim \sup_{0<t\leq T}t\|\tilde{f}_i'(t)\|_{C^1}+\mathbf{C}_{m-1}^1(G_i^1,G_i^2) \\
        &\quad\lesssim \sup_{0<t\leq T}t\|\tilde{f}_i'(t)\|_{C^1}
        +(\eps+\sigma+T^{\frac{\kappa}{2}}\|\phi\|_{C^{2m+3}})
        (1+\sigma+T^{\frac{\kappa}{2}}\|\phi\|_{C^{2m+3}})^{2m+3}\\
        &\qquad\times(\|f-\phi\|_{Z_T^m}
        +T^{\frac{\kappa}{2}}\|\phi\|_{C^{2m+3}})\\
        &\quad\quad+T^{\frac{1}{2}}(1+T)^m\|\tilde{\chi}_i\|_{C^m}
        (1+\sigma+T^{\frac{\kappa}{2}}\|\phi\|_{C^{2m+3}})^{2m+3}\\
        &\qquad\times(\|f-\phi\|_{Z_T^{m,\mathrm{low}}}
        +T^{\frac{\kappa}{2}}\|\phi\|_{C^{2m+3}}).
    \end{aligned}
\end{equation}
By the definition of $\tilde{f}'$, estimating $\partial_{d}^{2K+1}\nabla_{x'}^{l+1}\tilde{f}_i$ requires subtracting the contribution of $\Delta_{x'}\tilde{f}_i+\tilde{F}_i[f,g]$. The induction hypothesis \eqref{eq4.43} gives
\begin{equation}\label{eq4.39}
    \begin{aligned}
&\sum_{\substack{l+2k\leq 2m-2\\ 1\leq k\leq K-1}}
\sup_{0<t\leq T}t^{\frac{1+l+2k}{2}}
\|\nabla_{x'}^{l}\partial_d^{2k}\Delta_{x'}\tilde{f}_i(t)\|_{L^\infty}\\
&\quad+\sum_{\substack{l+2k\leq 2m-2\\ 1\leq k\leq K-1}}
\sup_{0<t\leq T}t^{1+\frac{1+l+2k+\kappa}{2}}
\|\nabla_{x'}^{l+1}\partial_d^{2k}\nabla\Delta_{x'}\tilde{f}_i(t)\|_{C^\kappa}\\
&\quad\lesssim \|\tilde{f}_{i,0}\|_{W^{1,\infty}}
+\left(\eps+\sigma+T^{\frac{\kappa}{2}}\|\phi\|_{C^{2m+3}}\right)
(1+\sigma+T^{\frac{\kappa}{2}}\|\phi\|_{C^{2m+3}})^{2m+3}\\
&\qquad\times(\|f-\phi\|_{Z_T^m}
+T^{\frac{\kappa}{2}}\|\phi\|_{C^{2m+3}})\\
&\quad\quad+T^{\frac{1}{2}}(1+T)^m\|\tilde{\chi}_i\|_{C^m}
(1+\sigma+T^{\frac{\kappa}{2}}\|\phi\|_{C^{2m+3}})^{2m+4}\\
&\qquad\times(\|f-\phi\|_{Z_T^{m,\mathrm{low}}}
+T^{\frac{\kappa}{2}}\|\phi\|_{C^{2m+3}}).
    \end{aligned}
\end{equation}
Furthermore, by \eqref{eq4.34} and \eqref{deftilF}, $\tilde{F}_i[f,g]$ consists of the subcritical term $(G_i[f,g]\circ\Phi_i+\mathcal R_i[f,g])\circ\Theta_i+\Delta\phi_i$ and the critical coefficient-oscillation term $\Theta_i^{-1}\left(\left(\tilde{\A}[\nabla\tilde{g}_i(\nabla\Phi_i)^{-1}]-\tilde{\A}_0\right)\circ\Theta_i\right)\Theta_i^{-\top}:\nabla^2f_i$. Hence
\begin{equation}\label{estrem1.3}
    \begin{aligned}
    &\sum_{\substack{l+2k\leq 2m-2\\ 1\leq k\leq K-1}}
    \sup_{0<t\leq T}t^{\frac{l+2k+1}{2}}
    \|\nabla_{x'}^{l}\partial_d^{2k}\tilde{F}_i[f,g](t)\|_{L^\infty}\\
    &\quad+\sum_{\substack{l+2k\leq 2m-2\\ 1\leq k\leq K-1}}
    \sup_{0<t\leq T}t^{1+\frac{1+l+2k+\kappa}{2}}
    \|\nabla_{x'}^{l+1}\partial_d^{2k}\nabla\tilde{F}_i[f,g](t)\|_{C^\kappa}\\
    &\quad\lesssim T^{\frac{1}{2}}(1+T)^m\|\tilde{\chi}_i\|_{C^m}(1+\sigma+T^{\frac{\kappa}{2}}\|\phi\|_{C^{2m+3}})(\|f-\phi\|_{Z_T^{m,\mathrm{low}}}+T^{\frac{\kappa}{2}}\|\phi\|_{C^{2m+3}})\\
        &\quad\quad+\left(\eps+\sigma+T^{\frac{\kappa}{2}}(1+T)^m\|\phi\|_{C^{2m+3}}\right)(1+\sigma+T^{\frac{\kappa}{2}}\|\phi\|_{C^{2m+3}})^{2m+3}(\|f-\phi\|_{Z_T^m}+T^{\frac{\kappa}{2}}\|\phi\|_{C^{2m+3}})
    \end{aligned}
\end{equation}
By \eqref{estz0}, \eqref{estmix1.3}, \eqref{eq4.39} and \eqref{estrem1.3}, we conclude that
\begin{equation*}
    \begin{aligned}
        &\sum_{\substack{l+2k\leq 2m\\ k\leq K}}\sup_{0<t\leq T}t^{\frac{1+l+2k+\kappa}{2}}\|\nabla^{l+1}_{x'}\partial_d^{2k}\nabla\tilde{f}_i(t)\|_{C^\kappa}\lesssim \sup_{0<t\leq T}t\|\tilde{f}_i'(t)\|_{C^1}\\
        &\quad\quad+(\eps+\sigma+T^{\frac{\kappa}{2}}\|\phi\|_{C^{2m+3}})(1+\sigma+T^{\frac{\kappa}{2}}\|\phi\|_{C^{2m+3}})^{2m+3}(\|f-\phi\|_{Z_T^m}+T^{\frac{\kappa}{2}}\|\phi\|_{C^{2m+3}})\\
        &\quad\quad+T^{\frac{1}{2}}(1+T)^m\|\tilde{\chi}_i\|_{C^m}(1+\sigma+T^{\frac{\kappa}{2}}\|\phi\|_{C^{2m+3}})^{2m+3}(\|f-\phi\|_{Z_T^{m,\mathrm{low}}}+T^{\frac{\kappa}{2}}\|\phi\|_{C^{2m+3}}).
    \end{aligned}
\end{equation*}
Applying \eqref{eq4.fprime-interpolation}, choosing $\delta$ sufficiently
small, and using the lower-order estimate \eqref{estz0}, we absorb the
$\delta\|\tilde f_i\|_{Z_{T,2}^m}$ term into the left-hand side. This proves
the case $k=K$ in \eqref{eq4.43} and completes the a priori estimate.
Using \eqref{estz0} for the remaining $Z_{T,2}^0$ term, and taking
    \begin{equation*}
        \eps\leq \frac{1}{500C},\quad \sigma\leq \frac{1}{2^{10dm}C},\quad\eps_0\leq\frac{\sigma}{2^{10dm}C},\quad T\leq \left(\frac{\sigma}{2^{10dm}C\|\phi\|_{C^{2m+3}}\|\tilde{\chi}_i\|_{C^m}}\right)^{\frac{10dm}{\kappa}},
    \end{equation*}
    for $C$ in \eqref{eq4.43}, we obtain
 	\begin{equation*}
 		\|\tilde{f}_i\|_{Z_{T,2}^m}\leq \frac{1}{10}(\sigma+\|f-\phi\|_{Z_T^m}).
 	\end{equation*}
 	Combining the result with \eqref{bdmczt1}, we obtain that
	 \begin{equation*}
	 	\|f-\phi\|_{Z_T^m}\leq \frac{1}{2}\sigma.
	 \end{equation*}  

	The construction of $\mathcal S$ for rough data can now be justified without circularity. Choose $\delta_q\downarrow0$, set $f_{0,q}=e^{\delta_q\Delta_\Omega}f_0$, and choose smooth boundary-compatible $g_q\to g$ in $\mathfrak Z_T^m$. Classical linear parabolic theory \cite{L1996} gives a smooth solution $f_q$ of \eqref{eqmcbd} with coefficient $\A[\nabla g_q]$ and initial value $f_{0,q}$ on a common interval $[0,T]$: the preceding a priori estimate is uniform in $q$. For every $\tau>0$, compactness of bounded subsets of the relevant H\"older spaces on $[\tau,T]$ yields a subsequence converging to a solution $f$ of \eqref{eqmcbd}; the uniform $Z_T^m$ bound passes to the limit by lower semicontinuity. Moreover, $f(t)\to f_0$ in $L^\infty(\Omega)$ as $t\downarrow0$, while $\nabla f(t)\rightharpoonup^*\nabla f_0$ in $L^\infty(\Omega)$. Thus the construction gives the asserted initial trace but does not use the maximum principle to claim strong $W^{1,\infty}$ continuity at $t=0$.
	    
	    Now we prove the contraction property. Let $g_1,g_2\in\mathcal{Z}_{T,\phi}^{m,\sigma}$, and $f_1=\mathcal{S}g_1$, $f_2=\mathcal{S}g_2$, and denote $\mathbf{f}=f_1-f_2$, $\mathbf{g}=g_1-g_2$, $\mathbf{A}=\A[\nabla g_1]-\A[\nabla g_2]$, then we can write the equation of $\mathbf{f}$ as
 	\begin{equation*}
 		\begin{aligned}
 			&\partial_t\mathbf{f}-\A[\nabla g_1]:\nabla^2\mathbf{f}=\mathbf{A}:\nabla^2f_2,\ \ \ \text{in}\ [0,T]\times \Omega,\\
 			&\mathbf{f}|_{t=0}=0,\ \ \ \text{in}\ \Omega,\\
 			&\mathbf{f}=0,\ \ \text{on}\ [0,T]\times\partial \Omega.
 		\end{aligned}
 	\end{equation*}
 	For interior estimates we can just denote $\mathbf{f}_{in}=\chi_0\mathbf{f}$ and we can write the equation of $\mathbf{f}_{in}$ as
 	\begin{equation*}
 		\begin{aligned}
 			&\partial_t\mathbf{f}_{in}-\A[\nabla g_1]:\nabla^2\mathbf{f}_{in}=\mathbf{A}:\nabla^2f_{2,in}+\chi_0\A[\nabla g_1]:\nabla^2\mathbf{f}-\A[\nabla g_1]:\nabla^2\mathbf{f}_{in}:=\mathbf{F}_{in},\ \ \ \text{in}\ [0,T]\times \mathbb{R}^d,\\
 			&\mathbf{f}_{in}|_{t=0}=0,\ \ \ \text{in}\ \mathbb{R}^d.
 		\end{aligned}
 	\end{equation*}
 	Apply Theorem~\ref{propb=0} to obtain the interior estimate
 	\begin{equation*}
 		\begin{aligned}
	 		&\|\mathbf{f}_{in}\|_{Z_{T,1}^m}
	 		\lesssim \sup_{0<t\leq T}\left(t^{\frac{1+\kappa}{2}}
	 		\|\mathbf{F}_{in}(t)\|_{\dot C^{\kappa}}
	 		+t^{\frac{2m+1+\kappa}{2}}
	 		\|\mathbf{F}_{in}(t)\|_{\dot C^{2m+\kappa}}\right) \\
            &\quad\lesssim (1+T)^m(\|\mathbf{f}\|_{Z_T^m}+\|\mathbf{g}\|_{Z_T^m})\\
            &\qquad\times(\|f_1-\phi\|_{Z_T^m}+\|f_2-\phi\|_{Z_T^m}
            +\|g_1-\phi\|_{Z_T^m}+\|g_2-\phi\|_{Z_T^m}
            +T^{\frac{\kappa}{2}}\|\phi\|_{C^{2m+3}})\\
				&\qquad\times(1+\|f_1-\phi\|_{Z_T^m}+\|f_2-\phi\|_{Z_T^m}
				+\|g_1-\phi\|_{Z_T^m}+\|g_2-\phi\|_{Z_T^m}
				+T^{\frac{\kappa}{2}}\|\phi\|_{C^{2m+3}}
				\|\chi_0\|_{C^{2m+3}})^{2m+4}\\
	 		&\quad\lesssim (1+T)^m
	 		(\sigma+T^{\frac{\kappa}{2}}\|\phi\|_{C^{2m+3}})
	 		(1+\sigma+T^{\frac{\kappa}{2}}\|\phi\|_{C^{2m+3}}
	 		\|\chi_0\|_{C^{2m+3}})^{2m+4}\\
	 		&\qquad\times(\|\mathbf{f}\|_{Z_T^m}+\|\mathbf{g}\|_{Z_T^m}).
 		\end{aligned}
 	\end{equation*}
	 	Since $\|\mathbf{f}_{in}\|_{Z_{T,1}^m}\leq \|\mathbf{f}\|_{Z_T^m}$,
	 	we can take $\sigma$ and $T$ small enough to ensure
 	\begin{equation*}
		\|\mathbf{f}_{in}\|_{Z_{T,1}^m}\leq\frac{1}{2^{10dm}}\|\mathbf{g}\|_{Z_T^m}.
	\end{equation*}
 	Now we only need to consider the boundary case. Consider the system \eqref{eqbdref} and \eqref{deftilF}, and denote 
    \begin{equation*}
        \begin{aligned}
           &f_{i,l}=(\chi_if_l)\circ\Phi_i\circ\Theta_i,\quad l=1,2\\
	           &\tilde g_{i,l}=(\tilde \chi_ig_l)\circ\Phi_i,\quad l=1,2\\
           &\mathbf{f}_i=f_{i,1}-f_{i,2},\quad \tilde{\mathbf{g}}_i=\tilde g_{i,1}-\tilde g_{i,2}\\
	           &\tilde{\A}^l=\tilde{\A}[\nabla\tilde{g}_{i,l}(\nabla\Phi_i)^{-1}],\quad\tilde{\A}_0=\tilde{\A}[\nabla\tilde{\phi}_i(\nabla\Phi_i)^{-1}](0),\quad G^l_i=G_i(f_l,g_l), \quad R^l_i=\mathcal R_i[f_l,g_l],\quad l=1,2, \\
	           &\tilde{\mathbf{A}}=\tilde{\A}^1-\tilde{\A}^2,\quad\mathbf{G}_i=G^1_i-G^2_i,\quad\mathbf{R}_i=R^1_i-R^2_i,\\
	           &\tilde{\B}^l=\Theta_i^{-1}
	           \bigl((\tilde{\A}^l-\tilde{\A}_0)\circ\Theta_i\bigr)
	           \Theta_i^{-\top},\quad
	           \tilde{\mathbf{B}}=\tilde{\B}^1-\tilde{\B}^2.
        \end{aligned}
    \end{equation*} 
    Then we can write the equation of $\mathbf{f}_i$ as 
 	\begin{equation}\label{eqcaph}
 		\begin{aligned}
 			&\partial_t\mathbf{f}_i-\Delta\mathbf{f}_i=\tilde{\mathbf{F}}_i,\ \ \ \text{in}\ [0,T]\times \mathbb{R}_+^d,\\
 			&\mathbf{f}_i|_{t=0}=0,\ \ \ \text{in}\ \mathbb{R}_+^d,\\
 			&\mathbf{f}_i=0,\ \ \text{on}\ [0,T]\times\partial \mathbb{R}_+^d,
 		\end{aligned}
 	\end{equation}
 	with 
 	\begin{equation*}
 		\begin{aligned}
 			&\tilde{\mathbf{F}}_i=(\mathbf{G}_i\circ\Phi_i+\mathbf{R}_i)\circ\Theta_i+\sum_{1\leq l,j\leq d}\tilde{\mathbf{B}}_{lj}\partial_{lj}f_{i,2}+\sum_{1\leq l,j\leq d}\tilde{\B}^1_{lj}\partial_{lj}\mathbf{f}_i,
 			\end{aligned}
 	\end{equation*}
    where we will simply denote $f=(f_1,f_2)$, $g=(g_1,g_2)$. Then by direct calculation as in Lemma~\ref{lem4.3}, one has
    \begin{equation}\label{eq4.45}
        \begin{aligned}
            &\sup_{0<t\leq T}\left(t^{\frac{1}{2}}\|\tilde{\mathbf{F}}_i(t)\|_{L^\infty}
            +t^{\frac{1+\kappa}{2}}\|\tilde{\mathbf{F}}_i(t)\|_{C^\kappa}\right)
            \lesssim (1+T)^3(\|\mathbf{g}\|_{Z_T^m}+\|\mathbf{f}\|_{Z_T^m})\\
            &\quad\quad\times\left(\eps+\|g-\phi\|_{Z_T^m}+\|f-\phi\|_{Z_T^m}
            +T^{\frac{\kappa}{2}}\|\phi\|_{C^{3}}\right)\\
            &\quad\quad\times\left(1+\|g-\phi\|_{Z_T^m}+\|f-\phi\|_{Z_T^m}
            +T^{\frac{\kappa}{2}}\|\phi\|_{C^{3}}\right)^3,\\
            &\sum_{n\leq 2m+1}\sup_{0<t\leq T}
            t^{\frac{n+\kappa}{2}}\|\nabla_{x'}\nabla^{n}\tilde{\mathbf{F}}_i(t)\|_{C^{\kappa}}
            \lesssim (1+T)^{m}(\|\mathbf{g}\|_{Z_T^m}+\|\mathbf{f}\|_{Z_T^m})\\
            &\quad\quad\times\left(\eps+\|g-\phi\|_{Z_T^m}+\|f-\phi\|_{Z_T^m}
            +T^{\frac{\kappa}{2}}\|\phi\|_{C^{2m+3}}\right)\\
            &\quad\quad\times\left(1+\|g-\phi\|_{Z_T^m}+\|f-\phi\|_{Z_T^m}
            +T^{\frac{\kappa}{2}}\|\phi\|_{C^{2m+3}}\right)^{2m+3}.
        \end{aligned}
    \end{equation}
    We only need to estimate the spatial derivatives, since time derivatives can be obtained from spatial derivatives and \eqref{eqcaph}. Apply Lemma~\ref{maincor} to \eqref{eqcaph} with $F^1=\tilde{\mathbf{F}}_i$, $F^2=0$, and use \eqref{eq4.45} and the fact that $\|(f-\phi,g-\phi)\|_{Z_T^m}\leq\sigma$, we get the lower-order estimates
    \begin{equation}\label{eq4.82}
        \|\mathbf{f}_i\|_{Z_{T,2}^{0}}\leq \frac{1}{2^{10dm}}\left(\|\mathbf{g}\|_{Z_T^m}+\|\mathbf{f}\|_{Z_T^m}\right),
    \end{equation}
	   provided $\eps,\sigma,T$ are sufficiently small. For higher-order tangential derivatives, the argument used for \eqref{esttg1.3}, together with \eqref{eq4.43} and \eqref{eq4.45}, gives
    \begin{equation}\label{eq4.46}
        \begin{aligned}
        &\sum_{l\leq 2m+1}\sup_{0<t\leq T}
        t^{\frac{l+\kappa}{2}}\|\nabla_{x'}^l\nabla \mathbf{f}_i(t)\|_{C^{\kappa}}\\
        &\quad\lesssim \sum_{l\leq 2m+1}\sup_{0<t\leq T}
        \left(t^{\frac{1+\kappa}{2}}\|\tilde{\mathbf{F}}_i(t)\|_{C^{\kappa}}
        +t^{\frac{1+l+\kappa}{2}}\|\nabla_{x'}^l\tilde{\mathbf{F}}_i(t)\|_{C^{\kappa}}\right) \\
        &\quad\lesssim (1+T)^m
        \left(\eps+\|g-\phi\|_{Z_T^m}
        +T^{\frac{\kappa}{2}}(1+\|\phi\|_{C^{2m+3}})\right)\\
        &\qquad\times(1+\|g-\phi\|_{Z_T^m}
        +T^{\frac{\kappa}{2}}\|\phi\|_{C^{2m+3}})^{2m+3}
        (\|\mathbf{f}\|_{Z_T^m}+\|\mathbf{g}\|_{Z_T^m}).
            \end{aligned}
    \end{equation}
	For higher-order normal derivatives, as in the a priori estimate, set
	\[
	\tilde{\mathbf{f}}'_i:=\partial_d^2\mathbf{f}_i+\Delta_{x'}\mathbf{f}_i+\tilde{\mathbf{F}}_i.
	\]
which satisfies the equation
\begin{equation*}
    \begin{aligned}
	        &\partial_t\tilde{\mathbf{f}}_i'-\Delta\tilde{\mathbf{f}}_i'=\partial_t\tilde{\mathbf{F}}_i, \quad\text{in }(0,T]\times\mathbb{R}_+^d,\\
	        &\tilde{\mathbf{f}}_i'=0, \quad\text{on }(0,T]\times\partial\mathbb{R}_+^d.
    \end{aligned}
\end{equation*}
Decompose the new forcing term $\partial_t\tilde{\mathbf{F}}_i$ to the form like $\mathbf{G}^1_i+\partial_d\mathbf{G}^2_i$, with
\begin{align*}
	    &\mathbf{G}^1_i=\partial_t\left((\mathbf{G}_i\circ\Phi_i+\mathbf{R}_i)\circ\Theta_i+\sum_{l\neq d\ \text{or}\ j\neq d}\tilde{\mathbf{B}}_{lj}\partial_{lj}f_{i,2}+\sum_{l\neq d\ \text{or}\ j\neq d}\tilde{\B}^1_{lj}\partial_{lj}\mathbf{f}_i\right)\\
	    &\quad\quad\quad+\partial_t\tilde{\mathbf{B}}_{dd}\partial_{dd}f_{i,2}+\partial_t\tilde{\B}_{dd}^1\partial_{dd}\mathbf{f}_i-\partial_d\tilde{\mathbf{B}}_{dd}\partial_{d}\partial_tf_{i,2}-\partial_d\tilde{\B}_{dd}^1\partial_d\partial_t\mathbf{f}_i, \\
    &\mathbf{G}_i^2=\tilde{\mathbf{B}}_{dd}\partial_{d}\partial_tf_{i,2}+\tilde{\B}_{dd}^1\partial_d\partial_t\mathbf{f}_i
\end{align*}
as in \eqref{defFG}. Then for $f,g\in\mathcal{Z}_{T,\phi}^{m,\sigma}$, we can use the same methods as Lemma~\ref{lem4.3} to get the estimates
\begin{equation*}
    \mathbf{C}_{m-1}^1(\mathbf{G}^1_i,\mathbf{G}^2_i)\lesssim (\eps+\sigma+T^{\frac{\kappa}{2}}\|\phi\|_{C^{2m+3}})(1+\sigma+T^{\frac{\kappa}{2}}\|\phi\|_{C^{2m+3}})^{2m+3}(1+T)^m\left(\|\mathbf{f}\|_{Z_T^m}+\|\mathbf{g}\|_{Z_T^m}\right).
\end{equation*}
Since $\tilde{\mathbf{f}}_i'(0) $ is not well-defined, to be precise, for any $t>0$, we take $t_0=\frac{t}{2}$, and apply $\tilde{\mathbf{f}}_i'(t_0)$ as the initial time, and we can see that 
\begin{align*}
	    &t\|\tilde{\mathbf{f}}_i'(t)\|_{C^1}+\sum_{n+2k+2l\leq 2m-1}t^{1+\frac{n+2k+2l+\kappa}{2}}\|\nabla_{x'}^n\partial_d^{2k}\nabla\partial_t^l\tilde{\mathbf{f}}_i'(t)\|_{C^\kappa}\\
    &\qquad\approx(t-t_0)\|\tilde{\mathbf{f}}_i'(t)\|_{C^1}+\sum_{n+2k+2l\leq 2m-1}(t-t_0)^{1+\frac{n+2k+2l+\kappa}{2}}\|\nabla_{x'}^n\partial_d^{2k}\nabla\partial_t^l\tilde{\mathbf{f}}_i'(t)\|_{C^\kappa}
\end{align*}
Additionally,
\begin{align*}
	    &\sup_{\tau\in[t_0,T]}\sum_{2i+j\leq 2m-2}\left((\tau-t_0)^{\gamma+i+\frac{1+j}{2}}\|\partial_\tau^i\nabla^{j}\mathbf{G}_i^{1}(\tau)\|_{L^\infty}+(\tau-t_0)^{\gamma+i+\frac{1+j+\kappa}{2}}\|\partial_\tau^i\nabla^{j}(\mathbf{G}_i^{1},\nabla_{x'}\mathbf{G}_i^2)(\tau)\|_{C^\kappa}\right.\\
	    &\quad\quad\left.+(\tau-t_0)^{\gamma+i+\frac{j}{2}}\|\partial_\tau^i\nabla^{j}\mathbf{G}_i^2(\tau)\|_{L^\infty}+(\tau-t_0)^{\gamma+i+\frac{j+\kappa}{2}}\|\partial_\tau^i\nabla^{j}\mathbf{G}_i^2(\tau)\|_{C^\kappa}\right)\lesssim \mathbf{C}_{m-1}^1(\mathbf{G}_i^1,\mathbf{G}_i^2)(T).
\end{align*}
Then we apply \eqref{eq4.45}, \eqref{eq4.82} and Lemma~\ref{maincor} to get 
\begin{equation*}
    \begin{aligned}
		        &\|\tilde{\mathbf{f}}_i'\|_{Z_{T,2}^{m-1,1}}\lesssim \sup_{\tau\in(0,T]}\tau\|\tilde{\mathbf{f}}_i'(\tau)\|_{C^1}+ \mathbf{C}_{m-1}^1(\mathbf{G}^1_i,\mathbf{G}^2_i)(T)\\
        &\quad\lesssim \delta\|\mathbf f_i\|_{Z_{T,2}^m}
        +C_\delta\|\mathbf f_i\|_{Z_{T,2}^0}
        +(\eps+\sigma+T^{\frac{\kappa}{2}}\|\phi\|_{C^{2m+3}})(1+\sigma+T^{\frac{\kappa}{2}}\|\phi\|_{C^{2m+3}})^{2m+3}(1+T)^m\left(\|\mathbf{f}\|_{Z_T^m}+\|\mathbf{g}\|_{Z_T^m}\right).
    \end{aligned}
\end{equation*}
Choose $\delta$ sufficiently small and use \eqref{eq4.82} to control
the $Z_{T,2}^0$ term. Finally, we use \eqref{eq4.46} to exclude the effect of $\Delta_{x'}\mathbf{f}_i+\tilde{\mathbf{F}}_i$, and obtain
 	\begin{equation*}
 		\begin{aligned}
			\|\mathbf{f}_i\|_{Z_{T,2}^m}\leq C_7(\eps+\sigma+T^{\frac{\kappa}{2}}\|\phi\|_{C^{2m+3}})(1+\sigma+T^{\frac{\kappa}{2}}\|\phi\|_{C^{2m+3}})^{2m+3}(1+T)^m\left(\|\mathbf{f}\|_{Z_T^m}+\|\mathbf{g}\|_{Z_T^m}\right).
 		\end{aligned}
 	\end{equation*}
 	So by taking $\eps$, $\sigma$ and $T$ small enough, we can prove
 	\begin{equation*}
		\|\mathbf{f}\|_{Z_T^m}\leq \frac{1}{2} \|\mathbf{g}\|_{Z_T^m},
	\end{equation*}
	This completes the proof of Theorem~\ref{mcbdglo}(i).
 \end{proof}
 \begin{remark}
 	The compatibility condition $u_0|_{\partial \Omega}=0$ is necessary. For a counterexample,
 	\begin{equation*}
 		\begin{aligned}
 			&    \partial_tu-\partial_{x}^2 u=0,\quad \text{in} \ [0,T]\times[0,\infty),\\
 			&u(0,x)=1, \quad \text{on}\ [0,\infty),\\
 			&  u(t,0)=0,\quad\forall t\geq 0,
 		\end{aligned}
 	\end{equation*}
	although the initial datum is smooth. The Dirichlet heat-kernel formula gives
 	\begin{equation*}
 		u(t,x)=\int_{0}^\infty (K(t,x-y)-K(t,x+y))dy=2\int_0^xK(t,z)dz,
 	\end{equation*}
 	with
 	\begin{equation*}
 		K(t,x)=\frac{1}{\sqrt{4\pi t}}e^{-\frac{|x|^2}{4t}}.
 	\end{equation*}
 	So we take derivative to obtain
 	\begin{equation*}
 		u'(t,x)=2\frac{1}{\sqrt{4\pi t}}e^{-\frac{x^2}{4t}},
 	\end{equation*}
	which implies
 	\begin{equation*}
	    \|u\|_{L^\infty(0,T;\dot C^1)}\sim \left\|\frac{u(0,x)}{x}\right\|_{L^\infty}.
 	\end{equation*}
 	In fact one can see that $\partial_x^{2n}u(t,x)|_{x=0}=0$ and $\partial_x^{2n}u(t,x)\sim cx$ near boundary for any $n\in\mathbb{N}_+$ and $t>0$. But $\|\frac{u_0(x)}{x}\|_{L^\infty}$ cannot be controlled by $\|u_0\|_{\dot C^1}$ unless $u(0,0)=0$.
 \end{remark}
 \begin{remark}\label{mcbdrmklt}
	As a consequence of Theorem~\ref{mcbdglo}(i), there are structural constants $\varepsilon_{*,0},T_0,C>0$ such that, if $\|f_0\|_{W^{1,\infty}(\Omega)}\leq\varepsilon_*\leq\varepsilon_{*,0}$, then one may take $\phi\equiv0$ and obtain a unique solution of \eqref{eqmcbd} on $[0,T_0]$ satisfying $\|f\|_{Z_{T_0}^m}\leq C\varepsilon_*$.
 \end{remark}
 \subsection{Global well-posedness}\label{sec4.3}
 This subsection extends the local solution globally for small initial data and proves its exponential decay by coupling Schauder control with energy estimates.
 \begin{proof}[Proof of Theorem~\ref{mcbdglo}, ii)]
	\medskip\noindent\textbf{Step 1: Large-time existence.}
We first iterate the local result with the fixed profile $\phi=0$. The local constants $T_0$ and $C$ are structural, while the time weights are reset on each overlapping interval of length $T_0$. This gives existence up to any prescribed finite time, provided the initial norm is chosen sufficiently small depending on that time. We will use this only to reach the fixed structural time $T_*$.
 	In this step, we assume $T\gg 1$, and we will give the global estimates of solution $f$ in $[0,T]$ for sufficiently small initial data. By Remark~\ref{mcbdrmklt}, if $\|f_0\|_{ W^{1,\infty}}=\varepsilon_*\leq C^{-1}\sigma$, we can take $\phi=0$ to prove the local existence with initial data $f_0$ in the time interval $[0,T_0]$ for some uniform constant $T_0\leq 1$, and the solution $f$ satisfying 
 	\begin{equation*}
 		\|f\|_{Z_{T_0}^m}\leq C\varepsilon_*\leq\sigma.
 	\end{equation*}
	Next, if we take $\varepsilon_*$ small enough that $2^{m+\frac{\kappa}{2}}C\varepsilon_*\leq C^{-1}\sigma$, then by Remark~\ref{mcbdrmklt} and uniqueness, we can take $f(\frac{T_0}{2})$ as new initial data and extend the solution to $[0,\frac{3T_0}{2}]$. Then we can prove that
    \begin{equation*}
        \sup_{t\in(\frac{T_0}{2},\frac{3T_0}{2}]}\|f(t)\|_{C^1}+\sup_{t\in(\frac{T_0}{2},\frac{3T_0}{2}]}
        (t-\frac{T_0}{2})^{m+\frac{\kappa}{2}}\|f(t)\|_{\dot C^{2m+1+\kappa}}\leq C\|f(\frac{T_0}{2})\|_{C^1}\leq C^2\varepsilon_*,
    \end{equation*}
	    Hence, on the extension interval $[T_0,\frac{3T_0}{2}]$, where $t-\frac{T_0}{2}\geq \frac{t}{2}$, we have
    \begin{equation*}
        \sup_{t\in(T_0,\frac{3T_0}{2}]}t^{m+\frac{\kappa}{2}}\|f(t)\|_{\dot C^{2m+1+\kappa}}\leq 2^{m+\frac{\kappa}{2}}C\|f(\frac{T_0}{2})\|_{C^1}\leq 2^{m+\frac{\kappa}{2}}C^2\varepsilon_*,
    \end{equation*}
    and
 	\begin{equation*}
		\|f\|_{Z^m_{\frac{3T_0}{2}}}\leq 2^{m+\frac{\kappa}{2}}C^{2}\varepsilon_*\leq\sigma.
 	\end{equation*}
	Similarly, if $2^{m+\frac{\kappa}{2}}3^{m+\frac{\kappa}{2}}C^2\varepsilon_*\leq C^{-1}\sigma$, then we can take $f(T_0)$ as the new initial data and get the a priori estimates of $f$ on $[0,2T_0]$,
    \begin{equation*}
		\|f\|_{Z^m_{2T_0}}\leq (2\times 3)^{m+\frac{\kappa}{2}}C^{3}\varepsilon_*\leq\sigma.
 	\end{equation*}
	    Here $t-T_0\geq \frac{t}{3}$ on $[\frac{3T_0}{2},2T_0]$. Inductively, set $n=\lfloor T/T_0\rfloor+1$. If $\varepsilon_*$ is small enough that $((2n-1)!)^{m+\frac{\kappa}{2}}C^{2n-1}\varepsilon_*\leq C^{-1}\sigma$, then Remark~\ref{mcbdrmklt} permits $2n-1$ repetitions and gives
 	\begin{equation*}\label{eq4.76}
		\|f\|_{Z_{nT_0}^m}\leq ((2n-1)!)^{m+\frac{\kappa}{2}}C^{2n-1}\varepsilon_*.
	\end{equation*}
	By Stirling's formula, the condition $((2n-1)!)^{m+\frac{\kappa}{2}}C^{2n-1}\varepsilon_*\leq C^{-1}\sigma$ leads to
 	\begin{equation}\label{z1}
 		e^{C_0T\log T}\varepsilon_*\leq \sigma,
 	\end{equation}
 	for some $C_0$ depending on $C,T_0,m,\kappa$.
	Then \eqref{eqmcbd} admits a unique solution on $[0,T]$ such that
	\begin{equation}\label{eq4.80}
		\|f\|_{Z_T^m}\lesssim e^{C_0T\log T}\varepsilon_*\lesssim\sigma
	\end{equation}
	provided that \eqref{z1} holds.
 	Let $T_*\geq 2$ be a uniform large constant satisfying
 	\begin{equation}\label{mcbdgi}
		C_1e^{-\frac{1}{d+5}\lambda t}(1+t)^{\beta_{m,\kappa}}\leq\frac{\iota}{10},\quad \forall t\geq T_*, 
 	\end{equation}
 	where $\iota\leq\sigma$, $\beta_{m,\kappa}$, $C_1$ and $\lambda$ will be decided later. Take $\varepsilon_*$ small enough such that
 	\begin{equation}\label{mcbddefs}
		e^{C_0(T_*+1)\log (T_*+1)}\varepsilon_*\leq \frac{\iota}{100}.
 	\end{equation}
	Moreover,  \eqref{eqmcbd} admits a unique solution on $[0,T_*+1]$ satisfying $\|f(T_*)\|_{C^2}\leq \frac{\iota}{50}$. So, we can  define
 	\begin{equation}\label{bddmct*}
 		T^*=\sup\left\{t>T_*:\sup_{t'\in[T_*,t]}\|f(t')\|_{C^2}\leq\iota\right\}.
 	\end{equation}
 	We want to prove by contradiction that $T^*=\infty$.
 	\\
 	\medskip\noindent\textbf{Step 2: Exponential decay.}
 	We want to prove the exponential decay of $L^2$ norm, which is
 	\begin{equation}\label{mcbdexpd}
 		\|f(t)\|_{L^2}\leq e^{-\lambda(t-1)}\|f(1)\|_{L^2}, \quad T^*\geq t\geq 1,
 	\end{equation}
	 	for some $\lambda>0$. Test \eqref{eqmcbd} against $f$. For a smooth solution, integration by parts gives
	 	\[
	 	\int_\Omega f\,\A_{ij}(\nabla f)\partial_{ij}f\,dx
	 	=\int_{\partial\Omega}f\,\A_{ij}(\nabla f)\partial_jf\,n_i\,dS
	 	-\int_\Omega \A_{ij}(\nabla f)\partial_i f\,\partial_j f\,dx
	 	-\int_\Omega f\,\partial_i\A_{ij}(\nabla f)\partial_jf\,dx.
	 	\]
	 	The boundary term vanishes because $f|_{\partial\Omega}=0$. Hence
 	\begin{equation*}
		\begin{aligned}
			\frac{1}{2}\frac{d}{dt}\|f(t)\|_{L^2}^2+\int_{\Omega} \A_{ij}(\nabla f)\partial_if\partial_jf\,dx
            &=-\int_\Omega f\,\partial_i\A_{ij}(\nabla f)\,\partial_jf\,dx\\
			&\leq C\|f(t)\|_{L^\infty}\|\nabla^2f(t)\|_{L^\infty}\|\nabla f(t)\|_{L^2}^2\\
			&\leq C_2(\Omega) \iota^2\|\nabla f(t)\|_{L^2}^2,\quad T^*>t>1,
		\end{aligned}
	\end{equation*}
	 	where $C_2(\Omega)$ depends only on $\Omega$, and we used
	 	\eqref{bddmct*}. Uniform ellipticity in the small-gradient regime gives
	 	a constant $C_3>0$, independent of $t$, such that
	 	\[
	 	\int_{\Omega} \A_{ij}(\nabla f)\partial_if\partial_jf\,dx
	 	\geq C_3\|\nabla f\|_{L^2}^2.
	 	\]
	 	If $\iota$ is sufficiently small that
 	\begin{equation}\label{eq4.57}
 		C_2(\Omega)\iota^2\leq\frac{C_3}{2},
 	\end{equation}
	 	then
 	\begin{equation*}
 		\partial_t\|f(t)\|_{L^2}^2+C_3\|\nabla f(t)\|_{L^2}^2\leq 0.
 	\end{equation*}
	On $[1,T_*]$, the smallness needed above follows from
	\eqref{eq4.80} and \eqref{mcbddefs}; on $[T_*,T^*)$ it follows from
	\eqref{bddmct*}. Thus this energy inequality is valid for
	$1\leq t<T^*$. Poincar\'e's inequality
	$\|f\|_{L^2}\lesssim_{\Omega}\|\nabla f\|_{L^2}$ and Gronwall's
	inequality yield \eqref{mcbdexpd}, with $\lambda$ depending only on
	$\Omega$. To control its value at $t=1$, observe that
	\[
	 (\partial_t-\A_{ij}(\nabla f)\partial_{ij})|f|^2
	 =-2\A_{ij}(\nabla f)\,\partial_i f\cdot\partial_j f\leq0.
	\]
	The parabolic maximum principle and the homogeneous boundary condition
	therefore give
	$\|f(1)\|_{L^2}\leq |\Omega|^{1/2}\|f_0\|_{L^\infty}$.
	No unweighted $C^2$ bound at $t=0$ is required.\\
 	\medskip\noindent\textbf{Step 3: Long-time behavior.}
We prove a polynomial bound at the $Z_T^1$ level.  This is
slightly stronger than the estimate needed in Step~4 and, more importantly,
allows the lower anisotropic norm to be treated without a circular argument.
Fix $K\geq1$ and $\beta_\kappa>0$, to be chosen below, and define
\begin{equation}\label{eq4.z1bootstrap}
 \widetilde T=\sup\left\{T\in[2,T^*):
 \|f\|_{Z_T^1}\leq K(1+T)^{\beta_\kappa}\right\}.
\end{equation}
The finite-time estimate \eqref{eq4.80} and \eqref{mcbddefs} initialize
this bootstrap after increasing $K$ if necessary.  We show that the bound in
\eqref{eq4.z1bootstrap} improves by a factor $1/2$ on every
$2\leq T<\widetilde T$, which implies that $\widetilde T=T^*$.

We have
  	\begin{equation*}
 		\begin{aligned}
 			&\partial_tf(t,x) -\Delta f(t,x) =F(t,x) ,\quad\text{in}\ [0,T^*]\times\Omega,\\
 			&f(t,x)=0,\quad\text{on}\ [0,T^*]\times\partial\Omega,\\
 			&f(0,x)=f_0(x),
 		\end{aligned}
 	\end{equation*}
 	with
 	\begin{equation*}
 		F = (\A[\nabla f]:\nabla^2f-\Delta f)= ((\A[\nabla f]-\A[0]):\nabla^2f).
 	\end{equation*}
    Apply Remark~\ref{rmk4.4}, \eqref{bddmciorh}, \eqref{eq4.36}, \eqref{eq4.39} and \eqref{estrem1.3} with $\phi=0$ and $g=f$, and with $\eps$ small enough, to obtain
    \begin{equation}\label{eq4.74}
        \begin{aligned}
            &\|f\|_{Z_T^1}\leq C_4\Bigl[
            \|f_0\|_{W^{1,\infty}}
            +(\|f\|_{Z_T^{1,\mathrm{low}}}+\eps)
            (1+\|f\|_{Z_T^{1,\mathrm{low}}})^2\|f\|_{Z_T^1}\\
            &\qquad+\|f\|_{Z_T^{1,\mathrm{low}}}
            (\eps+\|f\|_{Z_T^1})
            (1+\|f\|_{Z_T^{1,\mathrm{low}}})^2\Bigr]\\
            &\quad+C_4T^{\frac{\kappa}{2}}(1+T)^2
            (1+\|f\|_{Z_T^{1,\mathrm{low}}})^2
            \|\tilde{\chi}\|_{C^3}\|f\|_{Z_T^{1,\mathrm{low}}}.
        \end{aligned}
    \end{equation}
The cutoff term in \eqref{eq4.74} is estimated, as in \eqref{eq4.75}, by
an interpolation with an absorption parameter $\delta\in(0,1)$:
    \begin{equation}\label{eq4.75}
		\begin{aligned}
		C_4T^{\frac{\kappa}{2}}(1+T)^2\|f\|_{Z_T^{1,\mathrm{low}}}
		&\leq C_4' T^{\frac{\kappa}{2}}(1+T)^2
		\|f\|_{Z_T^1}^{\alpha_\kappa}
		\|f\|_{L^\infty(0,T;W^{1,\infty})}^{1-\alpha_\kappa}\\
		&\leq \delta\|f\|_{Z_T^1}
		+C_{\delta}(1+T)^{\beta_{\kappa}}
		\|f\|_{L^\infty(0,T;W^{1,\infty})},
		\qquad \delta\in(0,1).
		\end{aligned}
    \end{equation}
for some $\alpha_\kappa\in(0,1)$ and a polynomial dependence of
$C_\delta$ on $\delta^{-1}$.

It remains to prove that $Z_T^{1,\mathrm{low}}$ is uniformly small on the
bootstrap interval.  Write $h_1(t)$ and $\ell_1(t)$ for the time-$t$
integrands in $Z_T^1$ and $Z_T^{1,\mathrm{low}}$.  The anisotropic
interpolation following the definition of the latter norm gives
\begin{equation}\label{eq4.pointlow}
 \ell_1(t)\leq \theta h_1(t)+C\theta^{-p}\|f(t)\|_{C^1},
 \qquad 0<\theta<1,
\end{equation}
for some $p>0$.  The mixed boundary derivatives cause no loss here because
the high and low norms contain the same tangential--normal multi-indices.
Terms containing time derivatives are first rewritten as spatial
derivatives using the equation.

For $T_*\leq t<\widetilde T$, the bootstrap assumption gives
$h_1(t)\leq K(1+t)^{\beta_\kappa}$.  Choose
\[
 \theta(t)=\frac{\eta}{2K}(1+t)^{-2\beta_\kappa},
\]
for  $\eta\ll C_4^{-1}$ is a small enough constant. The first term in \eqref{eq4.pointlow} is at most $\eta/2$.  Moreover,
interpolating between \eqref{mcbdexpd} and the high norm supplied by the
bootstrap gives constants $a,b>0$ such that
\[
 \|f(t)\|_{C^1}\leq C\|f(t)\|_{L^2}^{\gamma}\|f(t)\|_{C^2}^{1-\gamma}\leq C e^{-a\lambda t}(1+t)^b,
 \qquad T_*\leq t<\widetilde T,
\]
for some $a,b>0,\gamma\in(0,1)$. Since $C\theta(t)^{-p}$ grows only polynomially, we may enlarge $T_*$ so
that the second term in \eqref{eq4.pointlow} is at most $\eta/2$.  On
$(0,T_*]$, \eqref{eq4.80} and \eqref{mcbddefs} give the same smallness after
decreasing the initial norm.  Hence
\begin{equation}\label{eq4.z1lowsmall}
 \|f\|_{Z_T^{1,\mathrm{low}}}\leq\eta,
 \qquad 2\leq T<\widetilde T.
\end{equation}

Insert \eqref{eq4.z1lowsmall} into \eqref{eq4.74}, and use
\eqref{eq4.75} with a fixed sufficiently small absorption parameter.
After choosing $\eta,\eps,$ and $\sigma$ small, we obtain
\[
 \|f\|_{Z_T^1}\leq C\left[\|f_0\|_{W^{1,\infty}}
 +(1+T)^{M_\kappa}\right],
 \qquad 2\leq T<\widetilde T,
\]
for some $M_\kappa>0$.  Taking $\beta_\kappa>M_\kappa$ and then $K$
sufficiently large, together with the smallness of data $\|f_0\|_{W^{1,\infty}}$, yields the strict improvement
\begin{equation*}\label{eq4.59}
 \|f\|_{Z_T^1}\leq \frac K2(1+T)^{\beta_\kappa},
 \qquad 2\leq T<\widetilde T.
\end{equation*}
Thus $\widetilde T=T^*$.  Finally, \eqref{holembed}, the time weights in
$Z_T^1$, and $\kappa>3/4$ imply, after changing the polynomial exponent,
\begin{equation}\label{bdmcltu}
 \|\nabla^2f(T)\|_{\dot C^{3/4}}
 \leq C(1+T)^{\beta_\kappa},
 \qquad 2\leq T<T^*.
\end{equation}

 	\medskip\noindent\textbf{Step 4: Global existence.}
 	We prove further that $T^*=\infty$. If this is not true, \textit{i.e.} $T^*<\infty$, we denote 
 	\begin{align*}
 		&\|f(T^*)\|_{C^2}=\|f(T^*)\|_{L^\infty}+\|\nabla f(T^*)\|_{L^\infty}+\|\nabla^2 f(T^*)\|_{L^\infty}:=X_0+X_1+X_2.
 	\end{align*}
 	For $\iota$ satisfying \eqref{eq4.57}, using the exponential decay \eqref{mcbdexpd}, long-time estimate \eqref{bdmcltu}, the choice of $T_*$ in \eqref{mcbdgi}, and the Gagliardo--Nirenberg interpolation inequality
 	\begin{equation*}
 		\|\nabla^l f\|_{L^\infty} \leq C_7\|f\|_{L^2}^{a_l} \|\nabla^2f\|_{\dot C^{\frac{3}{4}}}^{1-a_l}, \quad l=0,1,2,\quad a_l=\frac{11-4l}{2d+11},
 	\end{equation*}
 	we get
 	\begin{align*}
		&X_0\leq C_7\|f(T^*)\|_{L^2}^{\frac{11}{2d+11}}\|\nabla^2f(T^*)\|_{\dot C^{\frac{3}{4}}}^{\frac{2d}{2d+11}}\leq C_{7,\Omega}e^{-\frac{11}{2d+11}\lambda (T^*-1)}(1+T^*)^{\beta_{\kappa}},\\
		&X_1\leq C_7\|f(T^*)\|_{L^2}^{\frac{7}{2d+11}}\|\nabla^2f(T^*)\|_{\dot C^{\frac{3}{4}}}^{\frac{2d+4}{2d+11}}\leq C_{7,\Omega}e^{-\frac{7}{2d+11}\lambda (T^*-1)}(1+T^*)^{\beta_{\kappa}},\\
		&X_2\leq  C_7\|f(T^*)\|_{L^2}^{\frac{3}{2d+11}}\|\nabla^2f(T^*)\|_{\dot C^{\frac{3}{4}}}^{\frac{2d+8}{2d+11}}\leq C_{7,\Omega}e^{-\frac{3}{2d+11}\lambda (T^*-1)}(1+T^*)^{\beta_{\kappa}}.
 	\end{align*}
	The displayed Gagliardo--Nirenberg inequality is the bounded-domain
	inequality on the smooth domain $\Omega$.  Equivalently, one may first apply a
	bounded extension operator for $C^{2,3/4}(\overline\Omega)$ and $L^2(\Omega)$;
	the zero extension is not used, since the Dirichlet condition alone does not
	make its second derivatives regular across $\partial\Omega$. We take $\iota$ satisfying \eqref{mcbdgi} and enlarge $C_1$ so that it also dominates $e^{\lambda}C_{7,\Omega}$; then
 	\begin{equation*}
 		X_0+X_1+X_2\leq \frac{\iota}{2},
 	\end{equation*}
By continuity and the local continuation criterion, this strict improvement contradicts the definition of $T^*$; hence $T^*=\infty$.\\
	\medskip\noindent\textbf{Step 5: Higher regularity.}
We only give a sketch, since the argument is the same
bootstrap as in Steps~1--4, now carried out at the $Z_T^m$ level.  No new
boundary issue occurs: $Z_{T,2}^m$ and $Z_{T,2}^{m,\mathrm{low}}$ contain
the same tangential--normal multi-indices, while the terms involving time
derivatives are reduced to spatial derivatives by the equation, exactly as
in Section~\ref{sec4.2}.

Applying the localized estimates of Section~\ref{sec4.2} with $\phi=0$ and
$g=f$, and using the tame product estimates in Remark~\ref{rmk4.4}, we
obtain schematically
\begin{equation}\label{eq4.highpoly}
 \|f\|_{Z_T^m}
 \leq C_m\left[
 \|f_0\|_{W^{1,\infty}}
 +(\eps+\|f\|_{Z_T^{m,\mathrm{low}}})\|f\|_{Z_T^m}
 +(1+T)^{M_m}\|f\|_{Z_T^{m,\mathrm{low}}}
 \right],
\end{equation}
where the highest-order derivative occurs in at most one factor.

To close \eqref{eq4.highpoly}, repeat the first-failure argument of
Step~3.  Assume on a bootstrap interval that
\[
 \|f\|_{Z_S^m}\leq K_m(1+S)^{\beta_{m,\kappa}}
 \quad\text{for every }1\leq S\leq T.
\]
Writing $h_m(t)$ and $\ell_m(t)$ for the time-$t$ integrands in $Z_T^m$
and $Z_T^{m,\mathrm{low}}$, respectively, the same pointwise anisotropic
interpolation as in \eqref{eq4.pointlow} gives
\[
 \ell_m(t)\leq \theta h_m(t)
 +C_m\theta^{-p_m}\|f(t)\|_{C^1}.
\]
Choose
\[
 \theta(t)=\frac{\eta}{2K_m}(1+t)^{-2\beta_{m,\kappa}}.
\]
The first term is at most $\eta/2$.  For the second term, interpolate the
$L^2$ decay \eqref{mcbdexpd} with the polynomial high norm supplied by the
bootstrap.  As in Step~3, the resulting exponential decay of
$\|f(t)\|_{C^1}$ dominates the polynomial growth of $\theta(t)^{-p_m}$.
After enlarging $T_*$ and decreasing the initial norm to handle
$0<t\leq T_*$, we obtain
\[
 \|f\|_{Z_T^{m,\mathrm{low}}}\leq\eta
\]
throughout the bootstrap interval.  Taking $\eta$ and $\eps$ sufficiently
small, the highest-order term in \eqref{eq4.highpoly} is absorbed.  Choosing
$\beta_{m,\kappa}>M_m$ and then $K_m$ sufficiently large gives a strict
improvement of the bootstrap bound.  Hence the first-failure time is
infinite and
\begin{equation*}\label{eq4.highgrowth}
 \|f\|_{Z_T^m}
 \leq C_m(1+T)^{\beta_{m,\kappa}},
 \qquad T\geq1.
\end{equation*}
This is precisely the higher-order analogue of Step~3; the tame estimate in
Remark~\ref{rmk4.4} is what prevents the appearance of two highest-order
factors.
Consequently, by \eqref{holembed},
\begin{equation*}\label{eq4.highholder}
 \|f(T)\|_{C^{2m+1+\kappa}(\overline\Omega)}
 \leq C_m(1+T)^{M_{m,\kappa}},
 \qquad T\geq1,
\end{equation*}
after changing the polynomial exponent.

On the other hand, \eqref{mcbdexpd} gives
\[
 \|f(T)\|_{L^2}
 \leq C_\Omega e^{-\lambda T}
 \|f_0\|_{W^{1,\infty}}.
\]
The bounded-domain Gagliardo--Nirenberg inequality yields
\begin{equation}\label{eq4.highinterpolation}
 \|\nabla^{2m}f(T)\|_{L^\infty}
 \leq C
 \|f(T)\|_{L^2}^{\gamma_m}
 \|f(T)\|_{C^{2m+1+\kappa}}^{1-\gamma_m},
\end{equation}
where
\[
 \gamma_m
 =
 \frac{1+\kappa}{2m+1+\kappa+d/2}\in(0,1).
\]
Combining the preceding estimates, we obtain
\[
 \|\nabla^{2m}f(T)\|_{L^\infty}
 \leq
 C_m e^{-\gamma_m\lambda T}
 (1+T)^{(1-\gamma_m)M_{m,\kappa}},
 \qquad T\geq1.
\]
Thus, for some $c_m>0$,
\[
 \sup_{T\geq1}
 e^{c_mT}\|\nabla^{2m}f(T)\|_{L^\infty}
 \leq C\bigl(\|f_0\|_{W^{1,\infty}},m,\Omega\bigr).
\]
Combining this estimate with the finite-time $Z^m$ estimate proves
\eqref{bddmcgloe}.

 \end{proof}
 \section{Appendix: Boundary identities}\label{sec5}
This section proves the boundary identities and compatibility formulas used to recover normal derivatives in the half-space estimates.
 \begin{lemma}\label{lem5.1}
	     Assume $f,g\in C^{2}(\overline{\mathbb{R}}^d_+)$ and $f|_{\partial\mathbb{R}_+^d}=\Delta f|_{\partial\mathbb{R}_+^d}=g|_{\partial\mathbb{R}_+^d}=0$. Then
     \begin{equation}\label{eq5.1}
         \A[\nabla g]:\nabla^2f=0,\quad\forall x\in\partial\mathbb{R}_+^d.
     \end{equation}
     Moreover, for any $u\in C^2(\overline{\mathbb R^d_+})$ and $x_0\in\partial\mathbb R^d_+$,
	 \begin{equation}\label{mcbdfma}
	 (\A[\nabla g]:\nabla^2 u)(x_0)
	 =\left(\Delta_{x'}u+\frac{\partial_{x_d}^2u}{1+|\partial_{x_d}g|^2}\right)(x_0).
	 \end{equation}
 \end{lemma}
 \begin{proof}
     Since $f|_{\partial\mathbb{R}_+^d}=\Delta f|_{\partial\mathbb{R}_+^d}=0$, we can deduce that
     \begin{equation}\label{eq5.2}
         \partial_{d}^2f|_{\partial\mathbb{R}_+^d}=0.
     \end{equation}
     Furthermore, since $g|_{\partial\mathbb{R}_+^d}=0$, $\nabla g=(0,\ldots,0,\partial_dg)$ on the boundary. Thus, for each component $g^k$ and every $x_0\in\partial\mathbb{R}_+^d$, $\nabla g^k(x_0)$ is parallel to $e_d$. Consequently,
	 $\A[\nabla g](x_0)=\operatorname{diag}(1,\ldots,1,(1+|\partial_dg(x_0)|^2)^{-1})$, which proves \eqref{mcbdfma}. Together with \eqref{eq5.2}, this implies \eqref{eq5.1}.
 \end{proof}
 \begin{lemma}\label{bdcon}
 	Assume $f,g\in C^{2k+2}(\bar{\mathbb{R}}^d_+)$, and 
 	\begin{align*}
 	\Delta^j g=\Delta^j f=0,\ \ \ \ j=0,1,\cdots,k, \ \ \ x\in\partial\mathbb{R}^d_+.
 	\end{align*}
 	Then it holds
 	\begin{align}\label{Acommu}
 	\Delta^k\left(\A[\nabla g]:\nabla^2 f\right)- \A[\nabla g]:\nabla^2\Delta^k f=0,\ \ \ x\in\partial\mathbb{R}^d_+.
 	\end{align}
 \end{lemma}
 \begin{proof}[Proof of Lemma~\ref{bdcon}]
 	By boundary condition, for any $x_0\in\partial\mathbb{R}_+^d$, one has 
 	\begin{align}\label{dex0}
	( \partial_1^{n_1}\cdots\partial_{d-1}^{n_{d-1}} \partial_d^{2j})f(x_0)=0,\quad\quad \forall n_1,\cdots,n_{d-1}\in\mathbb{N},\quad 2j+\sum_{i=1}^{d-1}n_i\leq 2k,\quad 0\leq j\leq k, 
 	\end{align}
 	and 
 	\begin{align}\label{dex0g}
	( \partial_1^{n_1}\cdots\partial_{d-1}^{n_{d-1}} \partial_d^{2j})g(x_0)=0,\quad\quad \forall n_1,\cdots,n_{d-1}\in\mathbb{N},\quad 2j+\sum_{i=1}^{d-1}n_i\leq 2k,\quad 0\leq j\leq k. 
	\end{align}
 	We can write 
 	\begin{equation}\label{deldif}
 	\begin{aligned}
 	&|\Delta^k\left(\A[\nabla g]:\nabla^2 f\right)- \A[\nabla g]:\nabla^2\Delta^k f|\\
 	&\quad\quad\quad\quad\leq \sum_{j,l} |\partial_1^{j_1}\partial_2^{j_2}\cdots\partial_d^{j_d}\A[\nabla g]:\nabla^2\partial_1^{l_1}\partial_2^{l_2}\cdots\partial_d^{l_d}f|,
 	\end{aligned}
 	\end{equation}
 	here the summation is taken for some $j=(j_1,\cdots,j_d)$, $l=(l_1,\cdots,l_d)$ satisfying
 	\begin{align}\label{indjl}
 	\sum_{i=1}^d(j_i+l_i)=2k,\ \ \sum_{i=1}^dj_i> 0,\ \text{and}\  \{j_i+l_i\}_{i=1}^d\subset 2\mathbb{N},
 	\end{align}
 	where $2\mathbb{N}=\{2n:n\in\mathbb{N}\}$ denotes the set of even integers.
	 	We prove that the Frobenius inner product on the right-hand side of \eqref{deldif} is zero for every $j,l$ satisfying \eqref{indjl}. No cancellation is used: it suffices to determine the support of the matrices $\partial_1^{j_1}\partial_2^{j_2}\cdots\partial_d^{j_d}\A[\nabla g](x_0)$ and $\nabla^2\partial_1^{l_1}\partial_2^{l_2}\cdots\partial_d^{l_d}f(x_0)$.
 	
 	For simplicity, for any $M\in\mathbb{R}^{d\times d}$, we define its indicator matrix $\mathbf{I}(M)$ as follows.
 	\begin{align*}
 	\mathbf{I}_{ij}(M)=\begin{cases}
 	1,\ \ \ \ \text{if}\ M_{ij}\neq 0,\\
 	0,\ \ \ \ \text{if}\ M_{ij}= 0.\\
 	\end{cases}
 	\end{align*}
		The boundary identities \eqref{dex0}--\eqref{dex0g} imply that the odd reflections of $f$ and $g$ have the required jets through order $2k$. Consequently, the tangential--tangential and normal--normal entries of $\A[\nabla g]$ are even in $x_d$, whereas its mixed tangential--normal entries are odd. For $\nabla^2f$ the parity is reversed. Tangential differentiation preserves parity in $x_d$ and therefore cannot create a new support pattern at $x_d=0$; only the numbers $j_d$ and $l_d$ of normal derivatives determine which block may be nonzero. Thus the support analysis reduces to $\mathbf{I}(\partial_d^{j_d}\A[\nabla g](x_0))$ and $\mathbf{I}(\nabla^2\partial_d^{l_d}f(x_0))$, with all tangential derivatives regarded as parity-preserving prefactors. For notational simplicity, write $j=j_d$ and $l=l_d$ below.
 	We claim that:
 	(i). If $l,j\in 2\mathbb{N}$, $j\leq 2k, l<2k$, then 
 	\begin{align}
 	&\mathbf{I}(\partial_d^{j}\A[\nabla g](x_0))\leq\left(\begin{array}{cc}
 	1_{(d-1)\times(d-1) }  & 0_{(d-1)\times 1}  \\
 	0_{1\times(d-1)}  & 1
 	\end{array}\right),\label{eveA}\\    &\mathbf{I}(\nabla^2\partial_d^{l}f(x_0))\leq\left(\begin{array}{cc}
 	0_{(d-1)\times(d-1) }  & 1_{(d-1)\times 1}  \\
 	1_{1\times (d-1)}  & 0
 	\end{array}\right).\label{evef}
 	\end{align}
 	Here $1$ denotes a matrix all of whose entries equal $1$. For matrices $A,B$, we write $A\leq B$ if $A_{ij}\leq B_{ij}$ for every $i,j$.
 	(ii). If $l,j\in \mathbb{N}\backslash 2\mathbb{N}$, $j\leq 2k, l<2k$, then 
 	\begin{align}
 	&\mathbf{I}(\partial_d^{j}\A[\nabla g](x_0))\leq\left(\begin{array}{cc}
 	0_{(d-1)\times(d-1) }  & 1_{(d-1)\times 1}  \\
 	1_{1\times (d-1)}  & 0
 	\end{array}\right),\label{oddA}\\    &\mathbf{I}(\nabla^2\partial_d^{l}f(x_0))\leq\left(\begin{array}{cc}
 	1_{(d-1)\times(d-1) }  & 0_{(d-1)\times 1}  \\
 	0_{1\times (d-1)}  & 1
 	\end{array}\right).\label{oddf}
 	\end{align}
 	Then it follows from (i), (ii) and \eqref{deldif} that 
 	\begin{align*}
 	\Big(\Delta^k\left(\A[\nabla g]:\nabla^2 f\right)- \A[\nabla g]:\nabla^2\Delta^k f\Big)(x_0)=0.
 	\end{align*}
 	This holds for any $x_0\in\partial\mathbb{R}_+^d$. This completes the proof of \eqref{Acommu}.
 	
 	Now we prove (i) and (ii). By \eqref{dex0} we obtain \eqref{evef} and \eqref{oddf}. It remains to verify \eqref{eveA} and \eqref{oddA}. 
 	
	 	We first write $\A[\nabla g]$ as a Neumann series
 	\begin{align}\label{series}
	\A[\nabla g]=\left[\mathrm{Id}+\sum_{i=1}^{N}\nabla g^i\otimes \nabla g^i\right]^{-1}=\sum_{m=0}^\infty(-1)^m(\mathsf{P}^{-1}\B)^m\mathsf{P}^{-1},
 	\end{align}
 	where\begin{align*}
	 	\mathsf{P}=\mathrm{Id}+|\partial_d g|^2 e_d\otimes e_d,\ \ \ \B=\sum_{i=1}^N\nabla g^i\otimes \nabla g^i-|\partial_d g|^2 e_d\otimes e_d.
 	\end{align*}
	 	The series on the right-hand side of \eqref{series} is well-defined in a sufficiently small neighborhood of $x_0$ because $\|\mathsf{P}^{-1}(x)\|\leq1$ and
	 	\[
	 	\|\B(x)\|\leq C\|g\|_{C^2}^2|x-x_0|,\qquad |x-x_0|\ll1.
	 	\]
 	Moreover,
 	it suffices to consider finite terms in the series \eqref{series}. In fact, one can check that all the elements in the matrix
 	$
 	(\mathsf{P}^{-1}\B)^2$ contain a factor in  the form $\partial_{i_1}g\partial_{i_2}g$, where $i_1,i_2<d$. Hence for $m\geq 2$, all the elements in $(\mathsf{P}^{-1}\B)^m\mathsf{P}^{-1}$ contain a factor in the form $\prod_{k=1}^{m-1}\partial_{i_k}g$, with $i_k<d$. Then by \eqref{dex0g}, we have 
 	\begin{align*}
 	\left(\partial_d^{j}\left((\mathsf{P}^{-1}\B)^m\mathsf{P}^{-1}\right)\right)(x_0)=0,\ \ \ \forall m>j+1, \quad x_0\in\partial\mathbb{R}_+^d,
 	\end{align*}
	 	since $\partial_d^j$ supplies at most $j$ normal derivatives, at most $j$ tangential factors can be prevented from vanishing at $x_0$; hence only $m\leq j+1$ contributes. We therefore consider $\mathsf{D}_m=(\mathsf{P}^{-1}\B)^m\mathsf{P}^{-1}$ with $0\leq m\leq j+1$. To prove \eqref{eveA} and \eqref{oddA}, it suffices to prove, for any $0\leq m\leq j+1$,
 	\begin{equation}\label{indum}
 	\begin{aligned}
 	&\mathbf{I}(\partial_d^{j}\mathsf{D}_m(x_0))\leq
 	\left(\begin{array}{cc}
 	1_{(d-1)\times(d-1) }  & 0_{(d-1)\times 1}  \\
 	0_{1\times(d-1)}  & 1
 	\end{array}\right),\ \forall j\in 2\mathbb{N},j\leq 2k,\\
 	&	    \mathbf{I}(\partial_d^{j}\mathsf{D}_m(x_0))\leq \left(\begin{array}{cc}
 	0_{(d-1)\times(d-1) }  & 1_{(d-1)\times 1}  \\
 	1_{1\times(d-1)}  & 0
 	\end{array}\right),\ \forall j\in \mathbb{N}\backslash 2\mathbb{N},j\leq 2k.
 	\end{aligned}
 	\end{equation}
	 	When $m=0$, $\mathsf{D}_0=\mathsf{P}^{-1}=\operatorname{diag}(1,\ldots,1,(1+|\partial_dg|^2)^{-1})$. The odd reflection of $g$ makes $\partial_dg$ even in $x_d$, so $|\partial_dg|^2$ is even and every odd normal derivative of $\mathsf P^{-1}$ vanishes at $x_0$. Thus \eqref{indum} holds for $m=0$.
 	By induction, assume \eqref{indum} holds for $m\leq M$, we prove that it also holds for $M+1$. We can write 
 	\begin{align*}
 	\mathsf{D}_{M+1}=(\mathsf{P}^{-1}\B)^{M+1}\mathsf{P}^{-1}=\mathsf{D}_M \B\mathsf{P}^{-1}.
 	\end{align*}
 	And 
 	\begin{align*}
 	\partial_d^j\mathsf{D}_{M+1}=\sum_{j'=0}^j\partial_d^{j'}\mathsf{D}_M\ \partial_d^{j-j'}(\B\mathsf{P}^{-1}).
 	\end{align*}
 	Note that 
 	\begin{align*}
 	\B\mathsf{P}^{-1}=\left(\begin{array}{cc}
 	\nabla'g\otimes \nabla'g & \frac{\partial_d g (\nabla'g)^\top  }{1+|\partial_d g|^2} \\
 	\partial_d g \nabla'g & 0
 	\end{array}\right).
 	\end{align*}
 	By \eqref{dex0g}, we have 
 	\begin{equation}\label{indi}
 	\begin{aligned}
 	&\mathbf{I}(\partial_d^{j-j'}(\B\mathsf{P}^{-1})(x_0))\leq\left(\begin{array}{cc}
 	0_{(d-1)\times (d-1)} & 1_{(d-1)\times 1}  \\
	1_{1\times(d-1)} & 0
 	\end{array}\right),\ \ \ \ j-j'\in \mathbb{N}\backslash 2\mathbb{N},\\
 	&\mathbf{I}(\partial_d^{j-j'}(\B\mathsf{P}^{-1})(x_0))\leq\left(\begin{array}{cc}
 	1_{(d-1)\times (d-1)} & 0_{(d-1)\times 1}  \\
 	0_{1\times(d-1)} & 0
 	\end{array}\right),\ \ \ \ j-j'\in 2\mathbb{N}.
 	\end{aligned}
 	\end{equation}
 	If $j\in 2\mathbb{N}$, then we have $j',j-j'\in 2\mathbb{N}$ or $j',j-j'\in \mathbb{N}\backslash 2\mathbb{N}$, then by \eqref{indum} and \eqref{indi} we have
 	\begin{align*}
	 \mathbf{I}\left(\partial_d^{j'}\mathsf{D}_M\ \partial_d^{j-j'}(\B\mathsf{P}^{-1})\right)&\leq \mathbf{I}\left(  \left(\begin{array}{cc}
 	1_{(d-1)\times(d-1) }  & 0_{(d-1)\times 1}  \\
 	0_{1\times(d-1)}  & 1
 	\end{array}\right)\left(\begin{array}{cc}
 	1_{(d-1)\times (d-1)} & 0_{(d-1)\times 1}  \\
 	0_{1\times(d-1)} & 0
 	\end{array}\right)\right)\\
 	&\leq\left(\begin{array}{cc}
 	1_{(d-1)\times (d-1)} & 0_{(d-1)\times 1}  \\
 	0_{1\times(d-1)} & 0
 	\end{array}\right),\ \ \ j',j-j'\in 2\mathbb{N},
 	\end{align*}
 	and 
 	\begin{align*}
	 \mathbf{I}\left(\partial_d^{j'}\mathsf{D}_M\ \partial_d^{j-j'}(\B\mathsf{P}^{-1})\right)&\leq \mathbf{I}\left(  \left(\begin{array}{cc}
 	0_{(d-1)\times(d-1) }  & 1_{(d-1)\times 1}  \\
 	1_{1\times(d-1)}  & 0
 	\end{array}\right)\left(\begin{array}{cc}
 	0_{(d-1)\times (d-1)} & 1_{(d-1)\times 1}  \\
 	1_{1\times(d-1)} & 0
 	\end{array}\right)\right)\\
 	&\leq\left(\begin{array}{cc}
 	1_{(d-1)\times (d-1)} & 0_{(d-1)\times 1}  \\
 	0_{1\times(d-1)} & 1
 	\end{array}\right),\ \ \ j',j-j'\in \mathbb{N}\backslash 2\mathbb{N}.
 	\end{align*}
 	Hence we obtain 
 	\begin{align}\label{evre}
	 \mathbf{I}(\partial_d^j\mathsf{D}_{M+1}(x_0))\leq\left(\begin{array}{cc}
 	1_{(d-1)\times (d-1)} & 0_{(d-1)\times 1}  \\
 	0_{1\times(d-1)} & 1
 	\end{array}\right),\ \ \ j\in2\mathbb{N}, j\leq 2k.
 	\end{align}
 	Similarly, if $j\in\mathbb{N}\backslash 2\mathbb{N}$, then we have $j'\in 2\mathbb{N}$, $j-j'\in \mathbb{N}\backslash 2\mathbb{N}$ or  $j'\in \mathbb{N}\backslash 2\mathbb{N}$, $j-j'\in 2\mathbb{N}$, then by \eqref{indum} and \eqref{indi} we have
 	\begin{align*}
	 \mathbf{I}\left(\partial_d^{j'}\mathsf{D}_M\ \partial_d^{j-j'}(\B\mathsf{P}^{-1})\right)&\leq \mathbf{I}\left(  \left(\begin{array}{cc}
 	1_{(d-1)\times(d-1) }  & 0_{(d-1)\times 1}  \\
 	0_{1\times(d-1)}  & 1
 	\end{array}\right)\left(\begin{array}{cc}
 	0_{(d-1)\times (d-1)} & 1_{(d-1)\times 1}  \\
 	1_{1\times(d-1)} & 0
 	\end{array}\right)\right)\\
 	&\leq\left(\begin{array}{cc}
 	0_{(d-1)\times (d-1)} & 1_{(d-1)\times 1}  \\
 	1_{1\times(d-1)} & 0
 	\end{array}\right),\ \ \ j'\in 2\mathbb{N}, j-j'\in \mathbb{N}\backslash 2\mathbb{N},
 	\end{align*}
 	and 
 	\begin{align*}
	 \mathbf{I}\left(\partial_d^{j'}\mathsf{D}_M\ \partial_d^{j-j'}(\B\mathsf{P}^{-1})\right)&\leq \mathbf{I}\left(  \left(\begin{array}{cc}
 	0_{(d-1)\times(d-1) }  & 1_{(d-1)\times 1}  \\
 	1_{1\times(d-1)}  & 0
 	\end{array}\right)\left(\begin{array}{cc}
 	1_{(d-1)\times (d-1)} & 0_{(d-1)\times 1}  \\
 	0_{1\times(d-1)} & 0
 	\end{array}\right)\right)\\
 	&\leq\left(\begin{array}{cc}
 	0_{(d-1)\times (d-1)} & 0_{(d-1)\times 1}  \\
 	1_{1\times(d-1)} & 0
 	\end{array}\right),\ \ \ j'\in\mathbb{N}\backslash 2\mathbb{N},  j-j'\in 2\mathbb{N}.
 	\end{align*}
 	This implies that 
 	\begin{align}\label{odre}
	 \mathbf{I}(\partial_d^j\mathsf{D}_{M+1}(x_0))\leq\left(\begin{array}{cc}
 	0_{(d-1)\times (d-1)} & 1_{(d-1)\times 1}  \\
 	1_{1\times(d-1)} & 0
 	\end{array}\right),\ \ \ j\in\mathbb{N}\backslash 2\mathbb{N}, j\leq 2k.
 	\end{align}
	By \eqref{evre} and \eqref{odre}, \eqref{indum} holds for $M+1$. This yields \eqref{eveA} and \eqref{oddA} and completes the proof.
 \end{proof}

 The next lemma shows that the Dirichlet boundary condition is inherited by the Laplacian for forcing terms with this special structure.
 \begin{lemma}\label{lembdc}
 	Assume $g\in C((0,T];C^{2}(\overline{\mathbb{R}^d_+}))$ with $g(t,x)=0$ for any $(t,x)\in(0,T]\times\partial\mathbb{R}_+^d$. Let $u\in C^1((0,T];C(\overline{\mathbb R_+^d}))\cap C((0,T];C^2(\overline{\mathbb R_+^d}))$ be a solution to 
 	\begin{equation}\label{bddmcbdyeq}
 	\begin{aligned}
 	&\partial_tu(t,x)-\A[\nabla g](t,x):\nabla^2u(t,x)=F(t,x),\quad \text{in}\ [0,T]\times\mathbb{R}^d_+,\\
 	&u(0,x)=u_0(x),\quad \text{in}\ \mathbb{R}^d_+,\\
 	&u(t,x)=0,\quad \text{on}\ [0,T]\times\partial\mathbb{R}^d_+,
 	\end{aligned}
 	\end{equation}
 	with
 	\begin{equation*}
 	F(t,x)=0,\quad \text{on}\ (0,T]\times \partial\mathbb{R}^d_+.
 	\end{equation*}
 	Then we have 
 	\begin{equation*}
 	\Delta u(t,x)=0,\quad\text{on}\ (0,T]\times\partial\mathbb{R}^d_+.
 	\end{equation*}
 	\end{lemma}
 	\begin{proof}
 		Since $u|_{\partial\mathbb{R}_+^d}=0$, one has $\partial_t u|_{\partial\mathbb{R}_+^d}=0$. Moreover, we have $F|_{\partial\mathbb{R}_+^d}=0$, so by \eqref{bddmcbdyeq}, one has
 		\begin{equation*}
	 	(\A[\nabla g](t):\nabla^2u(t))(x_0)=0.
 		\end{equation*}
 		Combining this with \eqref{mcbdfma}, and the fact that 
 		$\Delta_{x'}u( x_0)=0$, one has
 		\begin{equation*}
	 	\partial_{x_d}^2 u(t,x_0)=0.
 		\end{equation*}
 		So we have 
 		\begin{equation*}
 		\Delta u(t,x_0)=0,\ \ \ \forall t\in(0,T].
 		\end{equation*}
 		Since the argument holds for any $x_0\in\partial\mathbb{R}_+^d$, one has 
 		\begin{equation*}
 		\Delta u(t,x)=0,\quad\text {on}\ (0,T]\times \partial\mathbb{R}_+^d.
 		\end{equation*}
 		This completes the proof.
	\end{proof}

    The following boundary lemma expresses even normal derivatives through derivatives of the forcing and lower-normal-order derivatives of the solution.
 	 \begin{lemma}\label{mcbdlembdy}
	 	Let $F\in C^\infty ((0,T]\times\mathbb{R}_+^d)$, and assume that $f$ is sufficiently smooth on $(0,T]\times\overline{\mathbb R_+^d}$ to justify the derivatives below. If $f$ solves
 	 	\begin{equation}\label{mcbdbdeq}
 	 	\begin{aligned}
 	 	&\partial_t f-\Delta f=F,\ \ \ \text{in}\ [0,T]\times \mathbb{R}_+^d,\\
 	 	&f|_{t=0}=f_0,\ \ \ \text{in}\ \mathbb{R}_+^d,\\
 	 	&f=0,\ \ \text{on}\ [0,T]\times\partial 
 	 	\mathbb{R}_+^d,
 	 	\end{aligned}
 	 	\end{equation}
	 	then for any $k\geq1$, we have the boundary condition
 	 	\begin{equation}\label{mcbdbdyfr}
 	 	\begin{aligned}
	 	&\partial_d^{2k}f(t,x)=F^{k-1}(t,x)+R^{2k}(t,x),\quad \forall x\in \partial\mathbb{R}_+^d,\\
	 	&F^{k-1}=\sum_{2i+|\alpha|=2k-2}C^1_{i,\alpha}\partial_t^i\nabla^{\alpha}F,\quad R^{2k}=\sum_{\substack{l+|\beta|= 2k\\l<2k-1}}C^2_{l,\beta}\partial_{d}^{l}\nabla_{x'}^{\beta}f,
 	 	\end{aligned}
 	 	\end{equation}
	  Furthermore, applying $\nabla_{x'}^\beta$ gives the analogous decomposition
 	 	\begin{equation}\label{mcbdbdyfrm}
 	 	\begin{aligned}
 	 	&\partial_d^{2k}\nabla_{x'}^\beta f=F^{k-1}_\beta+R^{2k}_\beta,\quad x\in(0,T]\times \partial\mathbb{R}_+^d,\\
	 	&F^{k-1}_\beta=\sum_{2i+|\alpha|=2k-2}C^{1,\beta}_{i,\alpha}\partial_t^{i}\nabla^{\alpha}\nabla_{x'}^\beta F,\quad R^{2k}_\beta=\sum_{\substack{l+|\zeta|= 2k\\l<2k-1}}C_{l,\zeta}^{2,\beta}\partial_{d}^{l}\nabla_{x'}^{\zeta+\beta}f.
 	 	\end{aligned}
 	 	\end{equation}
	 \end{lemma}
 	 \begin{proof}
 	 	We prove by induction. First, for $k=1$, since $f|_{\partial 
 	 		\mathbb{R}_+^d}=0$, we have $\partial_tf|_{\partial 
 	 		\mathbb{R}_+^d}=0$, then by the equation \eqref{mcbdbdeq} we have
 	 	\begin{equation*}
 	 	\begin{aligned}
	 	\partial_{dd}f=-F-\Delta_{x'}f=:F^0+R^2,\qquad F^0=-F,\quad R^2=-\Delta_{x'}f,\quad\text{on}\ (0,T]\times \partial
 	 	\mathbb{R}_+^d.
 	 	\end{aligned}
 	 	\end{equation*}
	 	So we have proved the case $k=1$. For $k>1$, assume inductively that \eqref{mcbdbdyfr} holds for every $k\leq K$. Taking $\partial_d^{2K}$ on both sides of \eqref{mcbdbdeq}, we have
 	 	\begin{equation*}
 	 	\begin{aligned}
 	 	\partial_d^{2K+2}f&=\partial_t\partial_d^{2K}f-\Delta_{x'}\partial_d^{2K}f-\partial_d^{2K}F\\
	 	&=-\Delta_{x'}\partial_d^{2K}f-\partial_d^{2K}F+\partial_t(F^{K-1}+R^{2K})\\
	 	&=-\Delta_{x'}\partial_d^{2K}f-\partial_d^{2K}F+\sum_{2i+|\alpha|=2K-2}C^{1}_{i,\alpha}\partial_t^{i+1}\nabla^{\alpha}F+\sum_{\substack{l+|\zeta|=2K\\l<2K-1}}C_{l,\zeta}^{2}\partial_{d}^{l}\nabla_{x'}^{\zeta}(\Delta f+F)\\
	 	&=\sum_{2i+|\alpha|=2K}C^{1}_{i,\alpha}\partial_t^{i}\nabla^{\alpha}F+\sum_{\substack{l+|\zeta|= 2K+2\\l<2K+1}}C_{l,\zeta}^{2}\partial_{d}^{l}\nabla_{x'}^{\zeta}f,\quad \text{on}\ (0,T]\times\partial\mathbb{R}_+^d,
 	 	\end{aligned}
 	 	\end{equation*}
		for constants $C_{i,\alpha}^{1},C_{l,\beta}^{2}$ depending only on the displayed indices, which implies that \eqref{mcbdbdyfr} holds for $K+1$. Thus \eqref{mcbdbdyfr} follows by induction for every $k\geq1$. The identity \eqref{mcbdbdyfrm} follows by applying $\nabla^\beta_{x'}$ to \eqref{mcbdbdyfr}. This completes the proof.
	 \end{proof}

\section*{Declarations}
\noindent\textbf{Data availability.} This research has no associated data.

\smallskip\noindent
\textbf{Conflict of interest.} The authors declare that they have no conflicts of interest.

\smallskip\noindent
\textbf{Acknowledgments.} Quoc-Hung Nguyen is supported by the CAS Project for Young Scientists in Basic Research, Grant No. YSBR-031, and by the NSFC under Grant Nos. 1251101538 and 12595282. Ke Chen is supported by the Research Centre for Nonlinear Analysis at The Hong Kong Polytechnic University.

\smallskip\noindent
\textbf{Note.} This manuscript is Part III of a study initially posted as \cite{KHN2024} and subsequently divided into papers on free-boundary problems, the Muskat problem, and fixed-boundary problems.

\end{document}